\documentclass[reqno]{amsart}
\usepackage{enumitem}
\usepackage{color}
\usepackage{amsmath}
\usepackage{amssymb}
\numberwithin{equation}{section}
\allowdisplaybreaks
\newtheorem{definition}{Definition}
\newtheorem{theorem}{Theorem}
\newtheorem{rem}{Remark}
\newtheorem{lemma}{Lemma}
\newcommand{\loc}{\operatorname{loc}}

\newcommand{\tcb}{\textcolor{black}}
\newcommand{\tcr}{\textcolor{black}}

\title[Diffusive limit of radiative heat transfer system]{On the diffusive limits of radiative heat transfer system I: well prepared initial and boundary conditions}
\author{Mohamed Ghattassi}
\address{Department of Mathematics, New York University in Abu Dhabi, Saadiyat Island, P.O. Box 129188, Abu Dhabi, United Arab Emirates, {\sf mg6888@nyu.edu}}
\author{Xiaokai Huo}
\address{Institute for Analysis and Scientific Computing, Vienna University of Technology, Wiedner Hauptstraße 8–10, 1040 Wien, Austria, {\sf xiaokai.huo@tuwien.ac.at}}
\author{Nader Masmoudi}
\address{Department of Mathematics, New York University in Abu Dhabi, Saadiyat Island, P.O. Box 129188, Abu Dhabi, United Arab Emirates--
Courant Institute of Mathematical Sciences, New York University, 251 Mercer Street, New York, NY 10012, USA, {\sf nm30@nyu.edu}}

\begin{document}                
 
\begin{abstract} We study the diffusive limit approximation for a nonlinear radiative heat transfer system \tcb{that arises in the modeling of glass cooling, greenhouse effects and in astrophysics. The model is considered} with the reflective radiative boundary conditions for the radiative intensity, and periodic, Dirichlet and Robin boundary conditions for the temperature.  The global existence of weak solutions for this system is given by using a Galerkin method with a careful treatment of the boundary conditions. Using the compactness method, averaging lemma and Young measure theory, we prove our main result that the weak solution converges to a nonlinear diffusion model in the diffusive limit. Moreover, under more regularity conditions on the limit system, the diffusive limit is also analyzed by using a relative entropy method. In particular, we get a rate of convergence. The initial and boundary conditions are assumed to be well-prepared in the sense that no initial and boundary layer exist.
%\tcr{The model studied here has a wide application in modeling glass cooling, greenhouse effects and in astrophysics. OR WE REMOVE IT}
 \end{abstract}
\keywords{Radiative transfer system, Compactness method, Averaging lemma, Young measures, Relative entropy}

\maketitle   
\tableofcontents

\section{Introduction}\label{sec:intro}
Radiative transfer equation describes the physical phenomenon of energy transport in radiation. It has a variety of applications, such as cooling glass \cite{thommes2002numerical}, radiation hydrodynamics \cite{mihalas2013foundations}, astrophysics \cite{peraiah2002introduction}, and greenhouse effect \cite{bardos2021radiative,golse2021stratified}. In this paper we consider a model of glass cooling with the radiative heat transfer equation coupled with a heat equation. The model is given by
\begin{equation}\label{eq1}
\begin{cases}
\displaystyle c_{m}\rho_{m}\partial_{t}T=k_{h}\Delta T-\int_{0}^{\infty}\int_{\mathbb{S}^{2}}\kappa\left(\mathcal{B}-\psi\right)d\beta d\nu,~~t>0,x \in \Omega,\\
\displaystyle \frac{1}{c}\partial_{t}\psi +\beta \cdot\nabla \psi=\kappa\left(\mathcal{B}-\psi\right), ~~t>0, (x,\beta,\nu)\in \Omega\times\mathbb{S}^{2}\times\mathbb{R}_+.
\end{cases}
\end{equation}
Here $\Omega \subset \mathbb{R}^3$ is a bounded domain, $\mathbb{S}^2$ is the unit sphere in $\mathbb{R}^3$. The function $T=T(t,x)$ denotes the temperature of the medium and $\psi=\psi(t,x,\beta,\nu)$ describes the specific radiation intensity at  $x\in \Omega$ traveling in the direction $\beta\in \mathbb{S}^{2}$ with frequency $\nu>0$ at time $t>0$. The constants $c_{m}$, $\rho_{m}$, $k_{h}$, $\kappa$ and $c$ are the specific heat, the density, the thermal conductivity, the opacity coefficient, and the speed of light, respectively. Furthermore, $\mathcal{B}=\mathcal{B}(\nu,T)$ denotes the Planck's function
$$\mathcal{B}\left(\nu,T\right):=\frac{2h_{p}\nu^{3}}{c^{2}\left(e^{\frac{h_{p}\nu}{k_{b}T}}-1\right)}$$
for black body radiation in glass. Here $h_{p}$ is the Planck's constant and $k_{b}$ is the Boltzmann's constant. We refer the reader to \cite{frank2010optimal}, \cite{modest2013radiative} and references therein for more radiative heat transfer models. 

In order to solve the glass cooling model \eqref{eq1}, we need to provide initial and boundary conditions for $T$ and $\psi$. The initial conditions are taken to be
\begin{align*}
	T(t=0,x) = ~& T_0(x), \; \text{for any }x\in\Omega, \\
	\psi(t=0,x,\beta,\nu) = ~& \psi_0(x,\beta,\nu), \; \text{for any } (x,\beta,\nu)\in \Omega\times\mathbb{S}^2\times\mathbb{R}_+.
\end{align*}
The boundary condition for the temperature $T$ is the following Robin boundary condition
\begin{align}\label{eq:robinfort}
 k n \cdot\nabla T(t,x)=h_{c}\left(T_{b}(t,x)-T(t,x)\right), \text{ for any }t>0,\; x\in\partial\Omega.
\end{align}
Here $T_b=T_b(t,x)>0$ is a nonnegative function, $n=n(x)$ is the outward unit normal vector to the boundary $\partial\Omega$, $h_{c}$ is the convective heat transfer coefficient and $k \ge 0$ is a constant. When $k=0$, this corresponds to a nonhomogenous Dirichlet boundary condition. To give the boundary condition for $\psi$, we define the boundary set $\Sigma = \partial\Omega\times\mathbb{S}^2$ and
\begin{equation}\label{gt0.1}
\Sigma_{-}:=\left\{(x,\beta)\in\partial\Omega\times\mathbb{S}^{2},\; \beta\cdot n(x)<0\right\},
\end{equation}
\begin{equation}\label{gt2.1}
\Sigma_{+}:=\left\{ (x,\beta)\in \partial \Omega \times\mathbb{S}^{2},\;\beta\cdot n(x)> 0\right\}.
\end{equation}
% \begin{equation}\label{gt2.2}
% \Sigma_{0}:=\left\{ (x,\beta)\in \partial \Omega \times\mathbb{S}^{2} \  ; \  \beta\cdot n(x) =  0\right\},
% \end{equation}
The boundary condition for the specific radiation intensity is taken to be the following reflecting absorbing mixed condition
\begin{align*}
\psi \left(t,x,\beta,\nu\right)=\alpha \psi_b(t,x,\beta,\nu)+(1-\alpha) \psi(t,x,\beta',\nu), \quad  t>0,\;\left(x,\beta\right)\in \Sigma_{-},\;\nu \in\mathbb{R}_+.
\end{align*} 
Here $\psi_b=\psi_b(t,x,\beta,\nu)$ is a given function defined in the in-flow direction, which is on the half surface $\Sigma_-$ and it describes the radiative intensity transmitted into the medium from outside. The coordinate $\beta^{'}\in\mathbb{S}^2$ is the exiting radius  which specularly reflects into the incident radius $\beta$ as $
\beta' = \beta - 2 (n(x)\cdot \beta) n(x)$, and $\alpha \in (0,1)$ is a constant. 

Next, we give the dimensionless form of the system \eqref{eq1}. We introduce the nondimensional parameter $\varepsilon=1/\kappa_{r}x_{r}$, where $x_{r}$ and $\kappa_{r}$ are the length scale and reference absorption, respectively. Physically $\varepsilon$ represents the ratio of a typical photon mean free path to a typical length scale of the problem. The rescaled system is given by 
\begin{equation*}%\label{eq1}
\begin{cases}
\displaystyle \varepsilon^{2}\partial_{t}T=\varepsilon^{2}k\Delta T-\int_{0}^{\infty}\int_{\mathbb{S}^{2}}\kappa\left(\mathcal{B}-\psi\right) d\beta d\nu,~~ t>0,\;x\in\Omega,\\
\displaystyle \varepsilon^{2}\frac{1}{c}\partial_{t}\psi +\varepsilon\beta \cdot\nabla \psi=\kappa\left(\mathcal{B}-\psi\right),~~  t>0,\;(x,\beta,\nu)\in\Omega\times\mathbb{S}^{2}\times\left(0,\infty\right).
\end{cases}
\end{equation*}
See \cite{frank2010optimal} for mores details on the derivation. We consider the glass cooling in the grey medium, that is
$\mathcal{B}$ does not depend on the frequency $\nu$. The specular black body intensity $\mathcal{B}$ is then given by $\mathcal{B}=\frac{\sigma}{\pi} T^{4}$, according to the Stefan-Boltzmann law. For simplicity, we take all the constants in \eqref{eq1} and \eqref{eq:robinfort} to be the same $k=\kappa=h_{c}=c=1$, and take $\sigma=\pi$.  Since the solutions of the above system depends on $\varepsilon$, we introduce new notations $T_{\varepsilon} = T_\varepsilon(t,x)$ and $\psi_{\varepsilon}=\psi_\varepsilon(t,x,\beta)$ to represent the temperature and the radiative intensity, respectively. We introduce the notation $\langle \psi_{\varepsilon}\rangle :=\int_{\mathbb{S}^{2}}\psi_{\varepsilon} d\beta$ which is the radiative density, the system \eqref{eq1} then can be written as 
\begin{align}
	\partial_t T_\varepsilon = ~& \Delta T_\varepsilon +\frac{1}{\varepsilon^2}\langle \psi_\varepsilon - T_\varepsilon^4\rangle
	 % \frac{1}{\varepsilon^2}\int_{\mathbb{S}^2}(\psi_\varepsilon - T_\varepsilon^4) d\beta
	 , \label{eq:Teps} \\
	\partial_t \psi_\varepsilon + \frac{1}{\varepsilon} \beta \cdot \nabla \psi_\varepsilon = ~&-\frac{1}{\varepsilon^2}(\psi_\varepsilon - T_\varepsilon^4). \label{eq:psieps}
\end{align}
The initial conditions are taken to be
\begin{align}
	T_\varepsilon(t=0,x) = ~& T_{\varepsilon 0}(x), ~~ \text{ for any } x\in\Omega, \label{eq:ic1}\\
	\psi_\varepsilon(t=0,x,\beta) = ~& \psi_{\varepsilon 0}(x,\beta), ~~ \text{ for any } x\in \Omega, \beta\in\mathbb{S}^2. \label{eq:ic2}
\end{align}
The boundary condition for $\psi_\varepsilon$ is taken to be
\begin{align}\label{bpsi}
	\psi_\varepsilon(t,x,\beta) = \alpha \psi_b(t,x,\beta) + (1-\alpha)(L \psi_\varepsilon)(t,x,\beta),\,t>0, \quad (x,\beta) \in \Sigma_{-},
\end{align}
where the reflection operator $L$ is defined by
\begin{align}\label{eq:reflectop}
	L(f(x,\beta)):=f(x,\beta')=f(x,\beta-2(n(x)\cdot \beta)n(x)).
\end{align}
The boundary data for $T_\varepsilon$ is taken to be one of the following three conditions:
\begin{enumerate}[label=(\Alph*)]
	\item On the torus: 
	\begin{align}
		\Omega=\mathbb{T}^3, \label{b1}%\text{ or} \label{b1}
	\end{align}
	\item Dirichlet boundary condition: 
	\begin{align}
		T_\varepsilon(t,x)=T_b(t,x), \text{ for any } x \in\partial \Omega, \label{b3} %\text{ or} 
	\end{align} 
	\item Robin boundary condition:
	 \begin{align}
	 	\varepsilon^r n \cdot \nabla T_\varepsilon(t,x) = -T_\varepsilon(t,x) + T_b(t,x), \text{ for any } x\in \partial \Omega. \label{b2}
	 \end{align}
\end{enumerate}
% \begin{align}
% \text{(A)}\,\,&\text{In the torus:} \; \Omega=\mathbb{T}^3, \text{or} \label{b1}\\
% \text{(B)}\,\,&\text{Robin boundary condition:}\nonumber\\
% &\varepsilon^r n \cdot \nabla T_\varepsilon(t,x) = -T_\varepsilon(t,x) + T_b(t,x),\;   x\in \partial \Omega, \label{b2}\\
% &  \quad \text{or}\nonumber\\
% \text{(C)}\,\,&\text{Dirichlet boundary condition:}\nonumber\\ & \quad T_\varepsilon(t,x)=T_b(t,x), \; x \in\partial \Omega. \label{b3}
% \end{align}
Here $r \ge 0$ is a nonnegative constant.

The parameter $\varepsilon$ is usually small in applications and it plays an important role in the system \eqref{eq:Teps}-\eqref{eq:psieps}. It is interesting and physically meaningful to study the behavior of its solutions as $\varepsilon \to 0$. We call such a limit the diffusive limit. The objective of this paper is to study the diffusive limit rigorously. 
%  and we will focus on studying the diffusive limit in this paper. 
 First we derive the limit system formally.

\subsection{Formal derivation of the limit system.}
% Now we will show how to derive the diffusive limit system. Indeed, 
By equation \eqref{eq:psieps}, 
\begin{align}\label{eq:10102}
	\psi_{\varepsilon}=T_{\varepsilon}^{4}-\varepsilon\beta \cdot\nabla \psi_{\varepsilon} -\varepsilon^{2} \partial_{t}\psi_{\varepsilon}.
\end{align}
Therefore, for small $\varepsilon$,  we have 
\[
\psi_{\varepsilon}=T_{\varepsilon}^{4}-\varepsilon\beta \cdot\nabla\left( T_{\varepsilon}^{4}-\varepsilon\beta \cdot\nabla \psi_{\varepsilon}-\varepsilon^2\partial_t \psi_\varepsilon\right)-\varepsilon^{2} \partial_{t}\left(T_{\varepsilon}^{4}-\varepsilon\beta \cdot\nabla \psi_{\varepsilon} -\varepsilon^{2} \partial_{t}\psi_{\varepsilon}\right).
\]
Combing the terms with the same order gives
\[
\psi_{\varepsilon}=T_{\varepsilon}^{4}-\varepsilon \beta \cdot\nabla T_{\varepsilon}^{4}-\varepsilon^{2}\left(\partial_{t}T_{\varepsilon}^{4}-\beta \cdot\nabla\left(\beta \cdot\nabla \psi_\varepsilon\right)\right)+ \varepsilon^3(\beta \cdot \partial_t \psi_\varepsilon + \beta \cdot \nabla \partial_t \psi_\varepsilon) + \varepsilon^4 \partial_t^2 \psi_\varepsilon.
\]
We can use \eqref{eq:10102} again in the third term on the right of the above equation and obtain
\begin{align}\label{eq:psiasymp}
	\psi_\varepsilon = ~& T_\varepsilon^4 - \varepsilon \beta \cdot \nabla T_\varepsilon^4 - \varepsilon^2 \left(\partial_t T_\varepsilon^4 - \beta \cdot \nabla (\beta \cdot \nabla T_\varepsilon^4)\right) \nonumber\\
	&+ \varepsilon^3(-\beta \cdot \nabla(\beta \cdot \nabla (\beta \cdot \nabla \psi_\varepsilon))+\beta \cdot \partial_t \psi_\varepsilon + \beta \cdot \nabla \partial_t \psi_\varepsilon) \nonumber\\
	&+ \varepsilon^4 (- \beta \cdot\nabla(\beta \cdot \nabla \partial_t\psi_\varepsilon) +\partial_t^2 \psi_\varepsilon).
\end{align}
Assuming $T_\varepsilon \in C_{t,x}^2$ and $\psi_\varepsilon \in C_{t,x,\beta}^3$ are bounded, and assuming 
\[T_\varepsilon \to \overline{T}, \quad \psi_\varepsilon \to \overline{\psi} \quad \text{as }\varepsilon\to 0,\]
we can pass to the limit $\varepsilon\to0$ in \eqref{eq:psiasymp} and get 
\[\overline{\psi} = \overline{T}^4.\]
We can also use \eqref{eq:psiasymp} to find that the radiative density $\langle \psi_{\varepsilon}\rangle$ satisfies
\begin{align*}
	\langle \psi_\varepsilon \rangle =~& 4\pi T_{\varepsilon}^{4}-\varepsilon^{2}4\pi\partial_{t}T_{\varepsilon}^{4}+\varepsilon^{2}\frac{4\pi}{3}\Delta T_{\varepsilon}^{4} \\
	&+  \varepsilon^3\langle (-\beta \cdot \nabla(\beta \cdot \nabla (\beta \cdot \nabla \psi_\varepsilon))+\beta \cdot \partial_t \psi_\varepsilon + \beta \cdot \nabla \partial_t \psi_\varepsilon) \rangle \nonumber\\
	&+ \varepsilon^4 \langle(- \beta \cdot\nabla(\beta \cdot \nabla \partial_t\psi_\varepsilon) +\partial_t^2 \psi_\varepsilon)\rangle.
\end{align*}
This enables us to pass to the limit on the last term in \eqref{eq:Teps}:
\begin{align*}
	\frac{1}{\varepsilon^2}(\langle \psi_\varepsilon - T_\varepsilon^4\rangle) \to -4\pi\partial_t \overline{T}^4 + \frac{4\pi}{3}\Delta\overline{T}^4.
\end{align*}
We can then pass to the limit in \eqref{eq:Teps} and derive the following nonlinear limit system 
\begin{align}\label{hgm1.0}
	 \partial_{t}\left(\overline{T}+4\pi \overline{T}^{4}\right)=~&\Delta\left(\overline{T}+\frac{4\pi}{3}\overline{T}^{4}\right),\\
	 \overline{T}(t=0,x) =~& \overline{T}_0(x) = \lim_{\varepsilon \to 0} T_{\varepsilon 0}(x), \quad x\in\Omega,
\end{align}
associated with suitable boundary conditions which will be given in section \ref{section3}.
% associated with the boundary conditions (will be discussed in section \ref{section3})
% \begin{eqnarray*}
% 	&\text{Robin boundary condition: } &n\cdot \nabla \overline{T} = \overline{T}-T_b, \quad \text{on }\partial\Omega, \quad \text{or}\\
% 	&\text{Dirichlet boundary condition: } &\overline{T} = T_b, \quad \text{on }\partial\Omega.
% \end{eqnarray*}

\subsection{Main results of the paper}
Before introducing our main results in this work, we start by giving some assumptions on the initial and boundary values. 
\begin{itemize}
\item \textbf{Well-prepared initial conditions}
\begin{align}\label{eq:wellinitial}
	\lim_{\varepsilon \to 0}(\psi_{\varepsilon0}(x) - T_{\varepsilon0}^4)=0, \;\text{for all } x\in\Omega,
\end{align}
\item   \textbf{Well-prepared boundary conditions} in the case of Dirichlet boundary condition \eqref{b3}, namely
\begin{align}\label{eq:wellbc}
	\psi_b(t,x,\beta) = T_b^4(t,x), \; \text{for all } t>0, \text{ and }(x,\beta)\in \Sigma_-.
\end{align}
Notice that for the case of Robin boundary condition \eqref{b2}, the well-prepared boundary condition assumption is not needed. The case of general boundary conditions will be discussed in \cite{Bounadrylayer2019GHM2,Bounadrylayer2021GHM3}.

\end{itemize}
We now state the main results in the following theorem.
\begin{theorem}
Suppose the \tcb{ positive initial data satisfy} $T_{\varepsilon 0} \in L^5(\Omega), \psi_{\varepsilon0} \in L^2(\Omega\times\mathbb{S}^2)$ \tcb{and assumption \eqref{eq:wellinitial}},  and the \tcb{postive boundary data satisfy} $T_b \in L_{\loc}^5([0,\infty);L^5(\partial \Omega))$ and $\psi_b \in L_{\loc}^2([0,\infty); L^2(\Sigma_-;|n\cdot\beta| d\beta d\sigma_x))$ \tcb{and assumption \eqref{eq:wellbc}}. Then the following statements hold.
\begin{enumerate}
	\item[(1)] \textbf{Existence of weak solution:} There exists a weak positive solution for the system \eqref{eq:Teps}-\eqref{eq:psieps} with initial conditions \eqref{eq:ic1}-\eqref{eq:ic2} and boundary condition \eqref{bpsi} for $\psi_\varepsilon$ with boundary condition \eqref{b1},\eqref{b3} or \eqref{b2} for $T_\varepsilon$.
	\item[(2)] \textbf{Diffusive limit}: As $\varepsilon \to 0$, the weak solution $(T_\varepsilon,\psi_\varepsilon)$ to the system \eqref{eq:Teps}-\eqref{eq:psieps} converges to  $(\overline{T},\overline{T}^4)$, where $\overline{T}$ is the weak solution of the system \eqref{hgm1.0} with boundary conditions that $\Omega = \mathbb{T}^3$ for the case (A) and $\overline{T}=T_b$ on the boundary for the case (B) and (C).
	\item [(3)]  \textbf{Rate of convergence}: Assume $\overline{T}$ is a strong solution to the system \eqref{hgm1.0} which has a positive lower bound. Then
	\begin{align*}
		\|T_\varepsilon(t)-\overline{T}(t)\|_{L^4(\Omega)}^4+\|\psi_\varepsilon(t)-\overline{\psi}(t)\|_{L^2(\Omega)}^2 \le C \|T_{\varepsilon 0}-\overline{T}_0\|_{L^4(\Omega)}^4+C\varepsilon^s,
	\end{align*} 
	where $s>0$ is a positive constant and takes the value $s=2,1,\min{(1,r)}$ for the case of boundary conditions \eqref{b1}, \eqref{b3}, \eqref{b2}, respectively.
\end{enumerate}
\end{theorem}

%
%The main results of the paper are
%\begin{enumerate}
%	\item Global existence of weak solutions for the system \eqref{eq:Teps}-\eqref{eq:psieps}. The existence results are proved under three different boundary conditions \eqref{b1}, or \eqref{b3} or \eqref{b2} for $T_\varepsilon$ and are stated in Theorem \ref{thm:existencet}, Theorem \ref{thmexistd} and Theorem \ref{thmexistr} respectively.
%
%	\item Convergence of the system \eqref{eq:Teps}-\eqref{eq:psieps} towards equation \eqref{hgm1.0} under the diffusive limit. We are going to prove the weak solution $(T_\varepsilon,\psi_\varepsilon)$ of the system \eqref{eq:Teps}-\eqref{eq:psieps} converges strongly to $(\overline{T},\overline{T}^4)$, where $\overline{T}$ is the weak solution for the equation \eqref{hgm1.0}, satisfying suitable initial and boundary conditions. See Theorem \ref{LimitProof} for details.
%
%	\item Using the relative entropy method, we give the convergence rate of the diffusive limit. See Theorem \ref{thmre} for details.
%\end{enumerate}
%All the above results are established under the assumption of well-prepared intial conditions
%\begin{align}\label{eq:wellinitial}
%	\lim_{\varepsilon \to 0}(\psi_{\varepsilon0}(x) - T_{\varepsilon0}^4)=0, \;\text{for all } x\in\Omega,
%\end{align}
%and well-prepared boundary conditions in the case of Dirichlet boundary condition \eqref{b3}, namely
%\begin{align}\label{eq:wellbc}
%	\psi_b(t,x,\beta) = T_b(t,x)^4, \; \text{for all } t>0, \text{ and }(x,\beta)\in \Sigma_-.
%\end{align}
%Notice that for the case of Robin boundary condition \eqref{b2},
% % with $r=0$, 

The main contribution of the present work is to give a more rigorous study of the radiative heat transfer system and its diffusive limit. We prove the global existence of weak solutions for the system and the convergence of the weak solutions to a nonlinear diffusion model under the diffusive limit. Our work extend the analysis made by Klar and Schmeiser in \cite{klar2001numerical}, where the existence and diffusive limit were established for smooth solutions. In their work, some extra assumptions on the solutions (which are not known to hold) were needed. Here we do not need these assumptions. The major difficulties in our work lie in the nonlinearity and lack of compactness of the system \eqref{eq:Teps}-\eqref{eq:psieps}. To overcome the difficulties, we use Young measure theory and averaging lemmas. The Young measure is applied to deal with the nonlinearity and the averaging lemma is applied to get the compactness. The diffusive limit can thus be rigorously justified. Assuming additional regularity on the limit system, the relative entropy method can be used to give the rate of convergence for the diffusive limit \tcb{for the boundary conditions \eqref{b1}, \eqref{b3} and \eqref{b2} with $r>0$. The case $r=0$ can only be treated using the weak compactness method.}% However, this method does not work in the case of the Robin boundary condition ($r=0$), due to the boundary layers.

A lot of literature is devoted to the mathematical analysis and numerical computations of the radiative heat transfer system \cite{larsen2002simplified, asllanaj2004transient, ghattassi2016galerkin,pinnau2007analysis,chebotarev2016nondegeneracy,ghattassi2018reduced}. Besides the work \cite{klar2001numerical} on the same model considered here, there are some works on similar models \cite{porzio2004application,golse2008rosseland,amosov2017unique,amosov2016unique, ghattassi2018existence}. For example, the existence and uniqueness of strong solutions for the non-grey coupled convection-conduction radiation system were proved in \cite{porzio2004application} using accretive operators theory. In \cite{golse2008rosseland}, the authors discussed the existence of weak solutions for a grey radiative transfer system without diffusion term in the temperature equation in a bounded domain with non-homogeneous Dirichlet  boundary conditions. They supposed that the radiative boundary data do not depend on the direction in order to avoid the boundary layer. The main tools used to prove the existence of weak solutions are the compactness argument based on a maximum principle and velocity averaging lemma. Furthermore, the existence and uniqueness of weak solution for the stationary nonlinear heat equation and the integro-differential radiative transfer equation for semitransparent bodies were studied in \cite{amosov2017unique}, where the authors took into account the effects of reflection and refraction of radiation according to the Fresnel laws at the boundaries of bodies. More recently, in  \cite{ghattassi2018existence}, the authors proved the local existence and uniqueness of strong solutions for the radiative heat transfer system under different types of  boundary conditions by using the Banach fixed point theorem. The time derivative term in the radiative transfer equation \eqref{eq:psieps} was also neglected therein. \tcb{Besides analysis, there are a lot of literature on the applications on the radiative transfer models. For example, a radiative transfer system without thermal diffusion was used in \cite{bardos2021radiative} to study the greenhouse effect, where the authors found the model was not capable of explaining the greenhouse effect due to the greenhouse gases. Note here our model \eqref{eq:Teps}-\eqref{eq:psieps} include thermal diffusion. A more complicated model including thermal diffusion and coupled with fluid equations was studied in \cite{golse2021stratified} to model greenhouse effects.}

The diffusion limit in radiative heat transfer system can be studied via Rosseland approximations \cite{bardos1987rosseland,bardos1988nonaccretive}. In \cite{bardos1987rosseland}, the authors derived the Rosseland approximation on a different radiative transfer equation where the solution also depends on the frequency variable $\nu$. Using the so-called Hilbert's expansion method, they proved the strong convergence of the solution of the radiative transfer equation to the solution of the Rosseland equation  for well prepared boundary data. Then, in \cite{bardos1988nonaccretive},  under some weak hypotheses on the various parameters of the radiative transfer equation, the Rosseland approximation was proved in a weak sense. More recently, in \cite{debussche2015diffusion,debussche2016diffusion}, the authors studied the diffusive limit of a stochastic kinetic radiative transfer equation, which is nonlinear and includes a smooth random term. They used a stochastic averaging lemma to show the convergence in distribution to a stochastic nonlinear fluid model. Moreover, there exists a wide literature on the diffusion limits for other kinds of kinetic systems, with various viewpoints and applications \cite{diperna1979uniqueness,dafermos1979stability,dafermos1979second,carrillo2001entropy,masmoudi2007diffusion,dolbeault2007non,saint2009hydrodynamic,el2010diffusion,lattanzio2013relative}. For example, in \cite{masmoudi2007diffusion}, the authors studied the diffusive limit of a semiconductor Boltzmann-Poisson system. The method of moments and a velocity averaging lemma were used to prove the convergence of its renormalized solution towards a global weak solution of a drift-diffusion-Poisson model. Similar methods have been used to study the hydrodynamic limit of Boltzmann equation \cite{saint2009hydrodynamic,lions2001boltzmanna,lions2001boltzmannb}. The hydrodynamic limit of Boltzmann equation can also be studied using the relative entropy method, for example to show the incompressible limit to Euler and Navier-Stokes equations \cite{saint2009hydrodynamic}. The origins of the relative entropy method come from continuum  mechanics, see \cite{dafermos1979second} for more details. The principle of this method is to measure in a certain way the distance between two solutions in some given space. This method was also used in the stability and asymptotic limit for different type of PDEs, for instance see \cite{el2010asymptotic,demoulini2012weak,lattanzio2013relative,tzavaras2005relative}.

The paper is organized as follows. In the next section, the Galerkin approximation is used to show the existence of global weak solution of the radiative heat transfer system. Then we prove the convergence of the weak solutions to a nonlinear parabolic equation \eqref{hgm1.0} in the diffusive limit in Section \ref{section3}, by using the averaging lemma and the theory of Young measures. Moreover, we recover the boundary condition for the nonlinear parabolic limit equation by using trace theorems. In Section \ref{section4}, we give the convergence rate of the diffusive limit by using the relative entropy method.

{\bf Notations: } In this paper, we use $\|\cdot\|_{L^{p}}$ to denote the natural norm on  $L^{p}(\Omega)$, for $p\in[1,\infty]$ and $\|\cdot\|_{H^{s}}$ is the norm on the sobolev space $H^{s}(\Omega)$, $s>0$. We use $\langle \cdot \rangle$ to denote the integral over $\beta \in \mathbb{S}^2$. $C_{t,x}$ is the space of continuous functions in time and space.

%%%%%%%%%%%%%%%%%%%%%%%%%%%%%%%%%%%%%%%%%%%%%%%%%%%%%%%%%%%%%%%%%%%%%%%%%%%%%%%%%%%%%%%%%%%%%%%%
%%%%%%%%%%%%%%%%%%%%%%%%%%%%%%%%%%%%%%%%%%%%%%%%%%%%%%%%%%%%%%%%%%%%%%%%%%%%%%%%%%%%%%%%%%%%%%%%
%\subsection{Main results}
%%%%%%%%%%%%%%%%%%%%%%%%%%%%%%%%%%%%%%%%%%%%%%%%%%%%%%%%%%%%%%%%%%%%%%%%%%%%%%%%%%%%%%%%%%%%%%%%%
%%%%%%%%%%%%%%%%%%%%%%%%%%%%%%%%%%%%%%%%%%%%%%%%%%%%%%%%%%%%%%%%%%%%%%%%%%%%%%%%%%%%%%%%%%%%%%%%%
%{\color{red}{the outside radiation is assumed to be a Planckian $\psi_b=T_{b}^{4}$}  Moreover, we discuss the  initial and  the boundary layer problem  by showing the matching of inner and outer solution.}

\section{Global existence of weak solutions}\label{section2}
In this section we prove the global existence of weak solutions for the radiative heat transfer system \eqref{eq:Teps}-\eqref{eq:psieps} under three different boundary conditions: torus, nonhomogeneous Dirichlet condition and Robin condition. We first consider the case of torus, i.e. $\Omega=\mathbb{T}^3$.
% \begin{align}
% 	\Omega = \mathbb{T}^3 \label{b1}.
% \end{align}
% Then we will consider the case of boundary conditions \eqref{b2} and \eqref{b3}.

% \begin{align}
% 	&\partial_t T_\varepsilon = \Delta T_\varepsilon + \frac{1}{\varepsilon^2}\int_{\mathbb{S}^2}(\psi_\varepsilon - T_\varepsilon^4) d\beta, \label{eq:Teps} \\
% 	&\partial_t \psi_\varepsilon + \frac{1}{\varepsilon} \beta \cdot \nabla \psi_\varepsilon = -\frac{1}{\varepsilon^2}(\psi_\varepsilon - T_\varepsilon^4), \label{eq:psieps}
% \end{align}
% with initial conditions 
% \begin{align}
% 	&T_\varepsilon(t=0,x) = T_{\varepsilon 0}(t,x), \label{eq:ic1}\\
% 	&\psi_\varepsilon(t=0,x,\beta) = \psi_{\varepsilon 0}(t,x,\beta). \label{eq:ic2}
% \end{align}
% The boundary conditions for $T_\varepsilon$ are taken to be one of the following three cases:
% \begin{enumerate}[label=({B\arabic*})]
% 	\item \label{b1} (Periodic boundary condition) $\Omega = \mathbb{T}^3$.
% 	\item \label{b2} (Robin boundary condition) $\varepsilon^r n \cdot \nabla T_\varepsilon = T_\varepsilon - T_b$ on $\partial \Omega$.
% 	\item \label{b3} (Dirichlet boundary condition) $T_\varepsilon=T_b$ on $\partial \Omega$.
% \end{enumerate}
% In the case of \eqref{b2} and \eqref{b3}, the boundary condition for $\psi_\varepsilon$ is taken to be
% \begin{align}\label{bpsi}
% 	\psi_\varepsilon(t,x,\beta) = \alpha \psi_b(t,x,\beta) + (1-\alpha)\psi_\varepsilon(t,x,\beta'), \quad (x,\beta) \in \Sigma_{-}.
% \end{align}
% We now proceed to establish the existence theorems. First we consider the case of torus.

\subsection{The case of torus}
We first prove the existence theorem for the case of torus. The idea of the proof can be modified to deal with bounded domain, which will be done later in this section. Before stating the existence theorem, we first introduce the definition of weak solutions.
\begin{definition}\label{df1}
	\tcr{ Let $0 \le T_{\varepsilon0} \in L^5(\mathbb{T}^3)$ and $0\le\psi_{\varepsilon0} \in L^2(\mathbb{T}^3 \times \mathbb{S}^2)$. We say that $(T_\varepsilon,\psi_\varepsilon)$ is a nonnegative weak solution of the system \eqref{eq:Teps}-\eqref{eq:psieps} with initial conditions \eqref{eq:ic1}-\eqref{eq:ic2} if }
	\begin{align}\label{eq:funsp}
		&T_\varepsilon \in L^\infty (0,\infty;L^5(\mathbb{T}^3))\cap C_w([0,\infty);L^5(\mathbb{T}^3)),\;\nabla T_\varepsilon^{\frac{5}{2}} \in L^2([0,\infty);L^2(\mathbb{T}^3)),\\
		% \cap L^5_{\loc}(0,\infty;L^{15}(\mathbb{T}^3))
		&\psi_\varepsilon \in L^\infty(0,\infty; L^2(\mathbb{T}^3 \times \mathbb{S}^2))\cap C_w([0,\infty);L^2(\mathbb{T}^3\times\mathbb{S}^2)),
	\end{align}
	and it solves \eqref{eq:Teps}-\eqref{eq:psieps} in the sense of distributions, i.e., for any test functions $\varphi  \in C^\infty([0,\infty)\times\mathbb{T}^3)$ and $\rho\in C^\infty([0,\infty)\times\mathbb{T}^3 \times \mathbb{S}^2)$, the following equations hold:
	\begin{align}
		&-\iint_{[0,\infty)\times\mathbb{T}^3}  \left(T_\varepsilon\partial_t \varphi + T_\varepsilon \Delta \varphi + \frac{1}{\varepsilon^2}\int_{\mathbb{S}^2} \varphi(\psi_\varepsilon - T_\varepsilon^4) d\beta \right) dxdt = \int_{\mathbb{T}^3} T_{\varepsilon0} \varphi(0,\cdot)dx, \label{eq:weakt1}\\
		&-\iiint_{[0,\infty)\times\mathbb{T}^3 \times \mathbb{S}^2} \left(\psi_\varepsilon \partial_t \rho + \frac{1}{\varepsilon} \psi_\varepsilon \beta \cdot \nabla \rho - \frac{1}{\varepsilon^2} \rho(\psi_\varepsilon-T_\varepsilon^4)\right) d\beta dxdt \nonumber\\
		&\quad= \iint_{\mathbb{T}^3 \times \mathbb{S}^2} \psi_{\varepsilon0}\rho(0,\cdot) dx. \label{eq:weakt2}
	\end{align}
\end{definition}

Next we prove the following existence theorem: 
\begin{theorem}\label{thm:existencet}
\tcr{ Let $0 \le T_{\varepsilon0} \in L^5(\mathbb{T}^3)$ and $0\le \psi_{\varepsilon0} \in L^2(\mathbb{T}^3 \times \mathbb{S}^2)$. }Then there exists a global nonnegative weak solution $(T_\varepsilon,\psi_\varepsilon)$ to the system \eqref{eq:Teps}-\eqref{eq:psieps} with initial data \eqref{eq:ic1}-\eqref{eq:ic2}. Moreover the following energy inequality holds for all $t>0$:
	\begin{align}
		&\frac{1}{5}\|T_\varepsilon(t,\cdot)\|_{L^5(\mathbb{T}^3)}^5 + \frac{1}{2}\|\psi_\varepsilon(t,\cdot,\cdot)\|_{L^2(\mathbb{T}^3\times \mathbb{S}^2)}^2 + \frac{16}{25}\int_0^t \|\nabla T_\varepsilon^{\frac{5}{2}}(\tau,\cdot)\|_{L^2}^2 d\tau \nonumber \\
		&\quad\quad+ \frac{1}{\varepsilon^2} \int_0^t \|\psi_\varepsilon(\tau,\cdot,\cdot) - T_\varepsilon^4(\tau,\cdot)\|_{L^2(\mathbb{T}^3 \times \mathbb{S}^2)}^2 d\tau \nonumber \\
		&\quad \le \frac{1}{5} \|T_{\varepsilon0}\|_{L^5(\mathbb{T}^3)}^5 + \frac{1}{2} \|\psi_{\varepsilon0}\|_{L^2(\mathbb{T}^3\times \mathbb{S}^2)}^2. \label{eq:energythm1}
	\end{align}
\end{theorem}
\begin{proof}

To prove Theorem \ref{thm:existencet}, we construct an approximate system using Galerkin approximations in finite dimensions, and then show the system converges as the dimension goes to infinity with the limit satisfying \eqref{eq:weakt1}-\eqref{eq:weakt2}.

\textbf{Construction of a Galerkin approximate system.} We first construct a finite dimensional approximations to the system \eqref{eq:Teps}-\eqref{eq:psieps} using Fourier series. We take the Fourier series of a $L^s$ ($s\ge1$) function to be 
	\[f(x) = \sum_{k \in \mathbb{Z}^d} \hat{f}(k) e^{ik\cdot x},\]
	and define the operator $\mathbb{P}_m:L^s \mapsto L^s$ ($s \ge 1$) as 
	\[\mathbb{P}_m f(x) = \sum_{|k| \le m} \hat{f}(k) e^{ik\cdot x}.\]
	Notice that $\mathbb{P}_m$ commutes with derivatives and convolutions. For function $g=g(x,\beta)$ defined on $\mathbb{T}^3\times \mathbb{S}^2$,
	\[\mathbb{P}_m g(x,\beta) =\sum_{|k| \le m} \hat{g}(k,\beta) e^{ik\cdot x}.\]
	We take the $m$-th Galerkin approximate system to be 
	\begin{align}
		\partial_t T_\varepsilon^m =~& \Delta T_\varepsilon^m + \frac{1}{\varepsilon^2} \int_{\mathbb{S}^2} \left(\psi_{\varepsilon}^m - \mathbb{P}_m\left((T_\varepsilon^m)^4\right)\right) d\beta, \label{eq:gl1}\\
		\partial_t \psi_\varepsilon^m + \frac{1}{\varepsilon} \beta \cdot \nabla \psi_\varepsilon^m =~& -\frac{1}{\varepsilon^2}\left(\psi_\varepsilon^m - \mathbb{P}_m\left((T_\varepsilon^m)^4\right)\right). \label{eq:gl2}
	\end{align}
The initial data is taken to be
\begin{align*}
	T_\varepsilon^m \big|_{t=0} = \mathbb{P}_mT_{\varepsilon0},\quad \psi_{\varepsilon}^m \big|_{t=0} = \mathbb{P}_m \psi_{\varepsilon0}.
\end{align*}
We make a change of variable $\xi=x-\frac{1}{\varepsilon}\beta t$ and equation \eqref{eq:gl2} changes into 
\[\frac{d}{dt} \psi_\varepsilon^m (t,\xi)= -\frac{1}{\varepsilon^2}\left(\psi_\varepsilon^m(t,\xi)- \mathbb{P}_m\left((T_\varepsilon^m)^4\right)\right),\]
which is an ODE in finite dimensional space. Since \eqref{eq:gl1} is also an ODE in  finite dimensional space, the system \eqref{eq:gl1}-\eqref{eq:gl2} has a unique solution $(T_\varepsilon^m,\psi_\varepsilon^m)$ on a maximal time interval $t_m$, according to the Cauchy-Lipschitz  theorem. The maximal existence time $t_m$ is characterized by
\begin{align*}
	\lim_{t\to t_m^{-}}\sup \left(\|T_\varepsilon^m\|_{L^5(\mathbb{T}^3)}^5 + \|\psi_\varepsilon^m\|_{L^2(\mathbb{T}^3\times\mathbb{S}^2)}^2 \right) = \infty.
\end{align*}
As will see next the norms above are bounded uniformly in time and so the Galerkin approximate system \eqref{eq:gl1}-\eqref{eq:gl2} is globally well-posed. \tcb{Note that the solutions to \eqref{eq:gl1}-\eqref{eq:gl2} are non-negative, which is proved in Appendix A.}

\textbf{Uniform estimate of the Galerkin system.}
Next we derive the energy estimate of the system \eqref{eq:gl1}-\eqref{eq:gl2}. Multiplying \eqref{eq:gl1} by $(T_\varepsilon^m)^4$ and \eqref{eq:gl2} by $\psi_\varepsilon^m$, adding the results together, and using the fact that $\mathbb{P}_m$ is a self-adjoint operator, we obtain
\begin{align*}
	\frac{d}{dt} &\left(\frac{1}{5} \|T_\varepsilon^m\|_{L^5(\mathbb{T}^3)}^5 + \frac{1}{2} \|\psi_\varepsilon^m\|_{L^2(\mathbb{T}^3 \times \mathbb{S}^2)}^2\right) + \frac{16}{25} \|\nabla (T_{\varepsilon}^m)^{\frac{5}{2}}\|_{L^2(\mathbb{T}^3)}^2 \\
	&+ \frac{1}{\varepsilon^2} \|\psi_\varepsilon^m - \mathbb{P}_m((T_\varepsilon^m)^4)\|_{L^2(\mathbb{T}^3\times \mathbb{S}^2)}^2  = \frac{1}{5} \|\mathbb{P}_mT_{\varepsilon0}\|_{L^5(\mathbb{T}^3)}^5 + \frac{1}{2} \|\mathbb{P}_m\psi_{\varepsilon0}\|_{L^2(\mathbb{T}^3 \times \mathbb{S}^2)}^2.
\end{align*}
Integrating it over $[0,t]$ and using the fact that $\|\mathbb{P}_m f\|_{L^s} \le \|f\|_{L^s}$, we obtain the energy inequality
\begin{align}
		&\frac{1}{5} \|T_\varepsilon^m(t)\|_{L^5(\mathbb{T}^3)}^5 + \frac{1}{2} \|\psi_\varepsilon^m(t)\|_{L^2(\mathbb{T}^3 \times \mathbb{S}^2)}^2+ \frac{16}{25} \int_0^t \|\nabla (T_{\varepsilon}^m)^{\frac{5}{2}}(\tau)\|_{L^2(\mathbb{T}^3)}^2d\tau \nonumber\\
		&\quad \quad + \frac{1}{\varepsilon^2}\int_0^t \|\psi_\varepsilon^m - \mathbb{P}_m((T_\varepsilon^m)^4)(\tau)\|_{L^2(\mathbb{T}^3\times \mathbb{S}^2)}^2 d\tau \nonumber \\
	&\quad \le \frac{1}{5} \|T_{\varepsilon0}\|_{L^5(\mathbb{T}^3)}^5 + \frac{1}{2} \|\psi_{\varepsilon0}\|_{L^2(\mathbb{T}^3 \times \mathbb{S}^2)}^2, \label{eq:enegl}
\end{align}
for all $t>0$.

It follows from the above energy inequality that, up to a subsequence
\begin{align}
	&\{T_\varepsilon^m\}_{m>0} \text{ is uniformly bounded in } L^\infty([0,\infty);L^5(\mathbb{T}^3)) , \label{eq:unibd1} \\
	&\{\nabla (T_\varepsilon^m)^{\frac{5}{2}}\}_{m>0} \text{ is uniformly bounded in } L^2([0,\infty);L^2(\mathbb{T}^3)), \label{eq:unibd2}\\
	&\{\psi_\varepsilon^m\}_{m>0} \text{ is uniformly bounded in } L^\infty([0,\infty);L^2(\mathbb{T}^3\times \mathbb{S}^2)), \label{eq:unibd3}\\
	&\left\{\frac{1}{\varepsilon}\left(\psi_\varepsilon^m - \mathbb{P}_m((T_\varepsilon^m)^4)\right) \right\}_{m>0}\text{ is uniformly bounded in } L^2([0,\infty);L^2(\mathbb{T}^3\times \mathbb{S}^2)). \label{eq:unibd4}
\end{align}
Using \eqref{eq:unibd2} and the Sobolev inequality
\[\|(T_\varepsilon^m)^{\frac{5}{2}} \|_{L^6(\mathbb{T}^3)} \le C \|\nabla (T_{\varepsilon}^m)^{\frac{5}{2}}\|_{L^2} + C \|(T_\varepsilon^m)^{\frac{5}{2}}\|_{L^2} \le  C \|(\nabla T_{\varepsilon}^m)^{\frac{5}{2}}\|_{L^2} + C \|T_\varepsilon^m\|_{L^5}^{\frac{5}{2}},\]
we have 
\begin{align}\label{eq:l15}
	\int_0^t &\|T_\varepsilon^m\|_{L^{15}(\mathbb{T}^3)}^5 d\tau \nonumber\\
	=& \int_0^t \|(T_\varepsilon^m)^{\frac{5}{2}}\|_{L^6(\mathbb{T}^3)}^2 d\tau \nonumber\\
	\le& C \int_0^t  \|\nabla (T_{\varepsilon}^m)^{\frac{5}{2}}\|_{L^2}^2 + C\int_0^t \|T_\varepsilon^m\|_{L^5}^2d\tau \nonumber\\
	\le& C \|\nabla (T_\varepsilon^m)^{\frac{5}{2}}\|_{L^2([0,\infty);L^2(\mathbb{T}^3))}d\tau + C t \|T_\varepsilon^m\|_{L^\infty([0,\infty);L^5(\mathbb{T}^3))}^2,
\end{align}
which is bounded. Therefore, 
\begin{align}
	&\{(T_\varepsilon^m)^{\frac{5}{2}}\}_{m>0}\text{ is uniformly bounded in }L^2_{\loc}([0,\infty);L^{6}(\mathbb{T}^3)),\label{eq:bdl52}\\
	&\{T_\varepsilon^m\}_{m>0} \text{ is uniformly bounded in }L^5_{\loc}([0,\infty);L^{15}(\mathbb{T}^3)).\label{eq:bdl15}
\end{align}

\textbf{Passing to the limit in the Galerkin system.}
Using \eqref{eq:unibd1}-\eqref{eq:unibd4} and \eqref{eq:bdl52}-\eqref{eq:bdl15}, 
there exists subsequences $\{T_\varepsilon^{m_k}\}_{k>0}$ and $\{\psi_\varepsilon^{m_k}\}_{k>0}$ such that
\begin{align}
	&T_\varepsilon^{m_k} \rightharpoonup T_\varepsilon, \text{ weakly in } L^2_{\loc}([0,\infty);L^5(\mathbb{T}^5))\cap L^{5}_{\loc}([0,\infty);L^{15}(\mathbb{T}^3)), \label{eq:wkconv1} \\
	&T_\varepsilon^{m_k} \rightharpoonup^* T_\varepsilon, \text{ weakly star in } L^\infty([0,\infty);L^5(\mathbb{T}^5)), \label{eq:wkstarconvt} \\
	&(T_\varepsilon^{m_k})^{\frac{5}{2}} \rightharpoonup \overline{(T_\varepsilon^{m_k})^{\frac{5}{2}}}, \text{ weakly in }L^2_{\loc}([0,\infty);H^1(\mathbb{T}^3)),\label{eq:asm2}\\
	&\psi_\varepsilon^{m_k} \rightharpoonup \psi_\varepsilon, \text{ weakly in } L^2_{\loc}([0,\infty);L^2(\mathbb{T}^3\times\mathbb{S}^2)), \label{eq:wkconv2} \\
	&\psi_\varepsilon^{m_k} \rightharpoonup^* \psi_\varepsilon, \text{ weakly star in } L^\infty([0,\infty);L^2(\mathbb{T}^3\times\mathbb{S}^2)), \label{eq:wkstarconvpsi} \\
	&\frac{1}{\varepsilon}(\psi_\varepsilon^{m_k} - \mathbb{P}_m((T_\varepsilon^{m_k})^4)) \rightharpoonup A, \text{ weakly in } L^2_{\loc}([0,\infty);L^2(\mathbb{T}^3\times\mathbb{S}^2)), \label{eq:weakconv3b}
\end{align}
as $k\to \infty$.
Notice that here $A \in L^2_{\loc}([0,\infty);L^2(\mathbb{T}^3\times\mathbb{S}^2))$ is a bounded function. Due to the property of the operator $\mathbb{P}_m$,
\begin{align*}
	(T_\varepsilon^m)^4 - \mathbb{P}_m((T_\varepsilon^m)^4) \to 0,
\end{align*}
as $m\to \infty$, so with \eqref{eq:weakconv3b}, we can conclude that 
\begin{align}
	&\frac{1}{\varepsilon}(\psi_\varepsilon^{m_k} - (T_\varepsilon^{m_k})^4) \rightharpoonup A, \text{ weakly in } L^2_{\loc}([0,\infty);L^2(\mathbb{T}^3\times\mathbb{S}^2)). \label{eq:weakconv3}
\end{align}
From the energy estiamte \eqref{eq:enegl}, we have
\begin{align*}
	&\partial_t T_\varepsilon^m \text{ is uniformly bound in } L^2_{\loc}([0,t];H^{-2}(\mathbb{T}^3)),\\
	&\partial_t \psi_\varepsilon^m \text{ is uniformly bound in } L^2_{\loc}([0,t];H^{-1}(\mathbb{T}^3)),
\end{align*}
these, together with \eqref{eq:unibd1} and \eqref{eq:unibd2}, by the Aubin-Lions Lemma, imply that
\begin{align}
	&T_\varepsilon^m \to T_\varepsilon, \text{ strongly in }L_{\loc}^2([0,t];L^2(\mathbb{T}^3)), \label{eq:tconv1}\\
	&\psi_\varepsilon^m \to \psi_\varepsilon, \text{ strongly in }L_{\loc}^2([0,t];L^2(\mathbb{T}^3\times\mathbb{S}^2)),\label{eq:psiconv1}\\
	&\partial_t T_\varepsilon^m \rightharpoonup \partial_t T_\varepsilon \text{ weakly in }L^2_{\loc}([0,\infty);H^{-2}(\mathbb{T}^3)), \\
	&\partial_t \psi_\varepsilon^m \rightharpoonup \partial_t \psi_\varepsilon \text{ weakly in }L^2_{\loc}([0,\infty);H^{-1}(\mathbb{T}^3\times\mathbb{S}^2)).
\end{align}
Therefore, the Galerkin approximate system converges. 
It also follows from the above fact that
\begin{align}\label{eq:wkcontinuity}
	T_\varepsilon \in C_{w}([0,\infty);L^2(\mathbb{T}^3)), \quad \psi_\varepsilon \in C_{w}([0,\infty);L^2(\mathbb{T}^3\times\mathbb{S}^2)).
\end{align}
This means that $T_\varepsilon$ and $\psi_\varepsilon$ are weakly continuous with values in $L^5(\mathbb{T}^3)$ and $L^2(\mathbb{T}^3\times\mathbb{S}^2)$, respectively.

\textbf{The limit satisfies the system \eqref{eq:Teps}-\eqref{eq:psieps}.}
To show the limit satisfies the system \eqref{eq:Teps}-\eqref{eq:psieps} in the sense of distributions, we apply test functions on \eqref{eq:gl1}-\eqref{eq:gl2}. We take the convergence subsequence obtained in the previous step. Here we will drop the subscript $k$ for simplicity. Fix $t>0$ and apply smooth function $\varphi \in C^\infty([0,t]\times \mathbb{T}^3)$ and $\rho \in C^\infty([0,t]\times \mathbb{T}^3\times \mathbb{S}^2)$ to the equations \eqref{eq:gl1} and \eqref{eq:gl2}, respectively, we arrive at
	\begin{align}
		&\int_{\mathbb{T}^3} T_\varepsilon^m(t) \cdot \varphi(t) dx - \int_0^t \int_{\mathbb{T}^3} T_\varepsilon^m \partial_t \varphi dxd\tau - \int_0^t\int_{\mathbb{T}^3} T_\varepsilon^m \Delta \varphi dxd\tau \nonumber\\
		&\quad- \int_0^t\iint_{\mathbb{T}^3\times \mathbb{S}^2}  \frac{1}{\varepsilon^2} \varphi (\psi_\varepsilon^m-\mathbb{P}_m((T_\varepsilon^m)^4))d\beta dxd\tau = \int_{\mathbb{T}^3} \mathbb{P}_mT_{\varepsilon0} \cdot \varphi(0) dx, \label{eq:glw1}\\
		&\iint_{\mathbb{T}^3\times\mathbb{S}^2} \psi_\varepsilon^m (t) \cdot \rho(t) d\beta dx - \int_0^t \iint_{\mathbb{T}^3\times\mathbb{S}^2} \psi_\varepsilon^m \partial_t \rho d\beta dxd\tau \nonumber \\
		&\qquad- \int_0^t \iint_{\mathbb{T}^3\times\mathbb{S}^2} \frac{1}{\varepsilon} \psi_\varepsilon^m \beta \cdot \nabla \rho d\beta dxd\tau  - \int_0^t\iint_{\mathbb{T}^3\times \mathbb{S}^2}  \frac{1}{\varepsilon^2}\rho (\psi_\varepsilon^m-\mathbb{P}_m((T_\varepsilon^m)^4))d\beta dxd\tau \nonumber\\
		&\quad= \iint_{\mathbb{T}^3\times\mathbb{S}^2} \mathbb{P}_m \psi_{\varepsilon0} \cdot \rho(0) d\beta dx. \label{eq:glw2}
	\end{align}
From the property of the operator $\mathbb{P}_m$, $\|f-\mathbb{P}_mf\|_{L^2} \to 0$ as $m\to\infty$, we get
\begin{align*}
	 &\int_{\mathbb{T}^3} \mathbb{P}_mT_{\varepsilon0} \cdot \varphi(0) dx \to  \int_{\mathbb{T}^3} T_{\varepsilon0} \cdot \varphi(0) dx, \\
	 & \iint_{\mathbb{T}^3\times\mathbb{S}^2} \mathbb{P}_m \psi_{\varepsilon0} \cdot \rho d\beta dx \to  \iint_{\mathbb{T}^3\times\mathbb{S}^2}  \psi_{\varepsilon0} \cdot \rho(0) d\beta dx.
\end{align*}
For the terms involving $(T_\varepsilon^m)^4$, we can use the strong convergence of $T_\varepsilon^m$ in \eqref{eq:tconv1} to get
\begin{align*}
	&\int_0^t\iint_{\mathbb{T}^3\times \mathbb{S}^2}\varphi (\mathbb{P}_m((T_\varepsilon^m)^4)-T_\varepsilon^4)d\beta dxd\tau \\
	&\quad = \int_0^t\iint_{\mathbb{T}^3\times \mathbb{S}^2}\varphi (\mathbb{P}_m((T_\varepsilon^m)^4)- \mathbb{P}_m T_\varepsilon^4 + \mathbb{P}_m T_\varepsilon^4- T_\varepsilon^4)d\beta dxd\tau \\
	&\quad= \int_0^t\iint_{\mathbb{T}^3\times \mathbb{S}^2}\mathbb{P}_m\varphi \cdot ((T_\varepsilon^m)^4 - T_\varepsilon^4) d\beta dxd\tau +  \int_0^t\iint_{\mathbb{T}^3\times \mathbb{S}^2} T_\varepsilon^4(\mathbb{P}_m \varphi - \varphi) d\beta dxd\tau\\
	&\le C(\|T_\varepsilon^m\|^{3}_{L^\infty([0,t];L^3(\mathbb{T}^3))} + \|T_\varepsilon\|^{3}_{L^\infty([0,t];L^3(\mathbb{T}^3))}) \|\mathbb{P}_m\varphi\|_{L^2(\mathbb{T}^3\times \mathbb{S}^2)}\|T_\varepsilon^m - T_\varepsilon\|_{L^2([0,t];L^2(\mathbb{T}^3))} \\
	&\quad + C\|T_\varepsilon\|^{4}_{L^8([0,t];L^8(\mathbb{T}^3))} \|\mathbb{P}_m\varphi - \varphi\|_{L^2([0,t];L^2(\mathbb{T}^3))}.
\end{align*}
From the strong convergence \eqref{eq:tconv1}, the first term on the right hand side of the above inequality goes to zero as $m \to \infty$. From the  property of $\mathbb{P}_m$, the second also goes to zero. Therefore, we conclude that 
\begin{align*}
	\int_0^t\iint_{\mathbb{T}^3\times \mathbb{S}^2}\varphi \mathbb{P}_m((T_\varepsilon^m)^4)d\beta dxd\tau \to \int_0^t\iint_{\mathbb{T}^3\times \mathbb{S}^2}\varphi T_\varepsilon^4 d\beta dxd\tau.
\end{align*}
The following convergence result can obtained in a similar way:
\begin{align*}
	\int_0^t\iint_{\mathbb{T}^3\times \mathbb{S}^2}\rho \mathbb{P}_m((T_\varepsilon^m)^4)d\beta dxd\tau \to \int_0^t\iint_{\mathbb{T}^3\times \mathbb{S}^2}\rho T_\varepsilon^4 d\beta dxd\tau.
\end{align*}
Lastly, from \eqref{eq:tconv1} and \eqref{eq:psiconv1}, we have
\begin{align*}
	&\int_{\mathbb{T}^d} T_\varepsilon^m(t) \varphi(t) dx \to \int_{\mathbb{T}^d}  T_\varepsilon (t) \varphi(t) dx, \\
	&\int_{\mathbb{T}^d} \psi_\varepsilon^m(t) \rho(t) dx \to \int_{\mathbb{T}^d}  \psi_\varepsilon (t) \rho(t) dx.
\end{align*}
Notice that the weakly continuity \eqref{eq:wkcontinuity} get rid of the possible bad zero measure set in time.

Using \eqref{eq:tconv1} and \eqref{eq:psiconv1}, we can pass to the limit in the other terms in equations \eqref{eq:glw1} and \eqref{eq:glw2}. Finally we arrive at \eqref{eq:weakt1} and \eqref{eq:weakt2}. Thus, for any $t>0$, ($T_\varepsilon$, $\psi_\varepsilon$) solves the system \eqref{eq:Teps}-\eqref{eq:psieps} in the sense of distributions and it satisfies \eqref{eq:funsp}.

\textbf{The energy inequality.}
To show  the energy inequality, we consider the inequality \eqref{eq:enegl}. Notice that since we have the strong convergence of $T_{\varepsilon}^m,\psi_\varepsilon^m$ according to \eqref{eq:tconv1} and \eqref{eq:psiconv1}, we can take a subsequence that converges everywhere. Let's take this subsequence to recover the energy estimate. The weak star convergences in \eqref{eq:wkstarconvt} and \eqref{eq:wkstarconvpsi} imply that
\begin{align}\label{eq:lsuptpsi}
	\|T_\varepsilon(t)\|_{L^5(\mathbb{T}^3)}^5 \le \limsup_{m\to\infty}\|T_\varepsilon^m\|_{L^5(\mathbb{T}^3)}^5, ~~ \|\psi_\varepsilon(t)\|_{L^2(\mathbb{T}^3\times\mathbb{S}^2)}^2 \le \limsup_{m\to\infty}\|\psi_\varepsilon^m\|_{L^2(\mathbb{T}^3\times\mathbb{S}^2)}^2.
\end{align}
The weak convergence in \eqref{eq:weakconv3} implies that 
\[\psi_\varepsilon^m - (T_\varepsilon^m)^4 \rightharpoonup \psi_\varepsilon -T_\varepsilon^4 \quad \text{in} \quad L^2([0,t];L^2(\mathbb{T}^3\times \mathbb{S}^2)). \]
Therefore,
\begin{align}\label{eq:liminfpsi-t4}
	\int_0^t \|\psi_\varepsilon - T_\varepsilon^4\|_{L^2(\mathbb{T}^3\times\mathbb{S}^2)}^2 d\tau \le \liminf_{m\to\infty} \int_0^t \|\psi_\varepsilon^m -\mathbb{P}_m((T_\varepsilon^m))^4\|_{L^2(\mathbb{T}^3\times \mathbb{S}^2)}^2 d\tau.
\end{align}
From \eqref{eq:unibd2} and the strong convergence \eqref{eq:tconv1}, we have
\[(T_\varepsilon^m)^{\frac{5}{2}} \rightharpoonup (T_\varepsilon)^{\frac{5}{2}}\quad \text{in} \quad L^2([0,\infty);{H}^1(\mathbb{T}^3)), \]
and thus
\begin{align}\label{eq:liminfnabla}
	\int_0^t \|(\nabla T_\varepsilon)^{\frac{5}{2}}\|_{L^2(\mathbb{T}^3)}^2 d\tau \le \liminf_{m\to \infty} \int_0^t \|(\nabla T_\varepsilon^m)^{\frac{5}{2}}\|_{L^2(\mathbb{T}^3)}^2 d\tau.
\end{align}
Taking the $\limsup_{m\to\infty}$ in the energy inequality \eqref{eq:enegl} and using the above estimates, we arrive at \eqref{eq:energythm1} and finish the proof.
\end{proof}

\subsection{Case of nonhomogeneous Dirichlet boundary condition}
We now consider the case of nonhomogeneous Dirichlet boundary condition \eqref{b3} under the assumption of well-prepared initial and boundary conditions. For simplicity, here we assume the boundary data $T_b$ and $\psi_b$ are time independent such that $T_b=T_b(x)$ and $\psi_b = \psi_b(x)$. The case of time dependent boundary can be treated similarily. Before we give the definition of weak solutions, we introduce the trace operators that extends the functions in Sobolev spaces to the boundary. 

We take $\gamma^1:H^1(\Omega) \to L^2(\partial\Omega)$ to be the trace operator. We learn from the trace theorem (see \cite{evans1997partial}) that if $\Omega$ is bounded and $\partial\Omega \in C^1$, then there exists a trace operator $\gamma^1$ such that 
\begin{align*}
	\gamma^1 f = f|_{\partial\Omega}, \text{ if } f \in H^1(\Omega),
\end{align*}
and 
\begin{align*}
	\|\gamma^1 f\|_{L^2(\partial\Omega)} \le C\|f\|_{H^{1}(\Omega)}.
\end{align*}
Here the constant $C$ only depends on $\Omega$. For the weak formulation of the equation \eqref{eq:Teps} we can apply the trace operator on $T_\varepsilon$ to get
\begin{align*}
	\gamma^1 T_\varepsilon = T_b.
\end{align*}

To consider the boundary condition \eqref{bpsi}, we define the trace operator following  \cite[Appendix (B.1)]{saint2009hydrodynamic} as
\begin{align}\label{eq:gamma2def}
	\gamma^2: \psi \in W^2(\mathbb{R}_+\times\Omega\times\mathbb{S}^2) \mapsto \psi|_{\partial \Omega} \in L^2(\mathbb{R}_+ \times \partial \Omega \times \mathbb{S}^2, |n\cdot \beta| d\beta d\sigma_x dt).
\end{align}
Here $W^2$ is the space
\begin{align}\label{eq:w2def}
	W^2(\mathbb{R}\times\Omega\times\mathbb{S}^2) := \{\psi \in L^2(\mathbb{R}_+\times\Omega\times\mathbb{S}^2): (\varepsilon \partial_t + \beta \cdot \nabla)\psi \in L^2(\mathbb{R}_+\times\Omega\times\mathbb{S}^2)\}.
\end{align}

The following lemma was proved in \cite[Proposition B.1]{saint2009hydrodynamic}.
\begin{lemma}\label{lm:trace}
	The trace operator defined in \eqref{eq:gamma2def} is continuous.
\end{lemma}
\begin{proof}
	For any bounded function $\rho \in C^1(\bar{\Omega}\times \mathbb{S}^2)$, we use Green's formula to get
	\begin{align}
	 	2\iiint& \rho(x,\beta) \psi (\varepsilon\partial_t + \beta \cdot \nabla_x)\psi  d\beta dxdt + \iiint (\beta \cdot \nabla_x) \rho(x,\beta) \psi^2 d\beta dxdt \nonumber\\=& \iiint \rho(x,\beta) \psi^2(t,x,\beta) (n\cdot \beta) d\beta d\sigma_xdt. \label{eq:psibcal}
	 \end{align} 
	 Choose $\rho(x,v)=n\cdot \beta/|n\cdot \beta|$ and we get 
	 \[\|\psi|_{\partial\Omega}\|_{L^2(\mathbb{R}_+\times\Omega\times\mathbb{S}^2,|n\cdot \beta|d\beta d\sigma_xdt)} \le C(\|\psi\|_{L^2(\mathbb{R}_+\times\Omega\times\mathbb{S}^2)} + \|(\varepsilon\partial_t + \beta \cdot \nabla_x )\psi)\|_{L^2(\mathbb{R}_+\times\Omega\times\mathbb{S}^2)}).\]
\end{proof}

\begin{definition}\label{def3}
	Assume $\Omega\in C^1$. \tcr{ Let $0 \le T_{\varepsilon0} \in L^5(\Omega)$ and $0 \le\psi_{\varepsilon0} \in L^2(\Omega \times \mathbb{S}^2)$. Let $0 \le T_b \in L_{\loc}^5([0,\infty);L^5(\partial \Omega))$ and $0 \le \psi_b \in L_{\loc}^2([0,\infty); L^2(\Sigma_-;|n\cdot\beta| d\beta d\sigma_x))$.} We say that $(T_\varepsilon,\psi_\varepsilon)$ is a weak solution of the system \eqref{eq:Teps}-\eqref{eq:psieps} with initial conditions \eqref{eq:ic1}-\eqref{eq:ic2} and boundary conditions \eqref{bpsi}, \eqref{b3} if
		\begin{align*}
		&T_\varepsilon \in L_{\loc}^\infty (0,\infty;L^5(\Omega)),
		% \cap L^5_{\loc}(0,\infty;L^{15}(\Omega)), 
		\quad T_\varepsilon^{\frac{5}{2}} \in L^2_{\loc}(0,\infty;{H}^1(\Omega)),\\
		&\psi_\varepsilon \in L_{\loc}^\infty(0,\infty; L^2(\Omega \times \mathbb{S}^2))\cap W^2_{\loc}([0,\infty)\times\Omega\times\mathbb{S}^2),
	\end{align*}
	and it solves \eqref{eq:Teps}-\eqref{eq:psieps} in the sense of distributions, i.e. for any test functions $\varphi  \in C_0^\infty([0,\infty)\times\Omega)$ and $\rho\in C_0^\infty([0,\infty)\times\Omega\times \mathbb{S}^2)$, the following equations hold:
	\begin{align}
		&-\iint_{[0,\infty)\times\Omega} \left(T_\varepsilon\partial_t \varphi + T_\varepsilon \Delta \varphi + \frac{1}{\varepsilon^2}\int_{\mathbb{S}^2} \varphi(\psi_\varepsilon - T_\varepsilon^4) d\beta \right) dxdt 
	 \nonumber\\
		&\quad 
		 + \iint_{[0,\infty)\times \partial \Omega} (\gamma^1 T_\varepsilon) n\cdot \nabla \varphi|_{\partial\Omega} d\sigma_x dt 
		% &\quad
		= \int_{\Omega} T_{\varepsilon0} \varphi(0,\cdot)dx, \label{eq:weakd1}\\
		&-\iiint_{[0,\infty)\times\Omega \times \mathbb{S}^2} \left(\psi_\varepsilon \partial_t \rho + \frac{1}{\varepsilon} \psi_\varepsilon \beta \cdot \nabla \rho - \frac{1}{\varepsilon^2} \rho(\psi_\varepsilon-T_\varepsilon^4)\right) d\beta dxdt 
		\nonumber \\
		% &\quad +\iiint_{[0,\infty)\times \Sigma} (n \cdot \beta)\rho|_{\Sigma} \cdot (\gamma^2\psi_\varepsilon) d\beta d\sigma_xdt 
		&\quad= \iint_{\Omega \times \mathbb{S}^2} \psi_{\varepsilon0}\rho(0,\cdot) dx, \label{eq:weakd2}
	\end{align}
	where
	\begin{align*}
		&\gamma^1 T_\varepsilon \big|_{\partial\Omega} = T_b,\\
		&\gamma^2 \psi_\varepsilon \big|_{\Sigma_{-}} = \alpha \psi_b + (1-\alpha)  L\psi_\varepsilon\big|_{\Sigma_{+}}.
	\end{align*}
\end{definition}

Next we prove the following existence theorem.
\begin{theorem}\label{thmexistd}
	 Assume $\partial\Omega\in C^1$. \tcr{ Let $0 \le T_{\varepsilon0} \in L^5(\Omega)$ and $0 \le\psi_{\varepsilon0} \in L^2(\Omega \times \mathbb{S}^2)$. Let $0 \le T_b \in L_{\loc}^5([0,\infty);L^5(\partial \Omega))$ and $0 \le \psi_b \in L_{\loc}^2([0,\infty); L^2(\Sigma_-;|n\cdot\beta| d\beta d\sigma_x))$.} Let the boundary condition to be well prepared such that $\psi_b = T_b^4$. Then there exists a global nonnegative weak solution $(T_\varepsilon,\psi_\varepsilon)$ of the system \eqref{eq:Teps}-\eqref{eq:psieps} with initial conditions \eqref{eq:ic1}-\eqref{eq:ic2} and boundary conditions \eqref{bpsi}, \eqref{b3}. Furthermore, the following inequality holds
	 	\begin{align}\label{eq:estdirichlet}
		&\|T_{\varepsilon}(t)\|_{L^5(\Omega))}^5 + \|\psi_\varepsilon(t)\|_{L^2(\Omega\times\mathbb{S}^2)}^2 + \int_0^t \|\nabla (T_\varepsilon)^{\frac{5}{2}}\|_{L^2(\Omega)}^2 d\tau \nonumber\\
		&\qquad+ \frac{1}{\varepsilon^2} \int_0^t \|\psi_\varepsilon -T_\varepsilon^4)\|_{L^2(\Omega\times\mathbb{S}^2)}^2 d\tau \nonumber \\
		&\qquad+\frac{2\alpha - \alpha^2}{2\varepsilon} \int_0^t \|\psi_\varepsilon - \psi_b\|_{L^2(\Sigma_+;|n\cdot\beta| d\beta d\sigma_x)}^2 d\tau \nonumber\\
		&\quad \le C(\|T_{\varepsilon0}\|_{L^5(\Omega)}^5 + \|\psi_{\varepsilon0}\|_{L^2(\Omega\times\mathbb{S}^2)}^2),
	\end{align}
	for any $t>0$. \tcb{Here $C$ depends on $t$ and $\Omega$ but is independent of $\varepsilon$. The norm $\|\cdot\|_{L^2(\Sigma_+;|n\cdot\beta| d\beta d\sigma_x)}$ in the above inequality is defined by 
	$$\|f\|_{L^2(\Sigma_+;|n\cdot\beta| d\beta d\sigma_x)}:=\int_{\partial\Omega} \int_{\mathbb{S}^2\cap \beta\cdot n>0} f(x,\beta) d\sigma_x d\beta,$$
	where $\sigma_x$ is the surface integral element.}
\end{theorem}

\begin{proof}
Since the boundary conditions \eqref{b3} and \eqref{bpsi} for $T_\varepsilon$ and $\psi_\varepsilon$ are not homogeneous, we need to lift up the boundary data and make the Galerkin approximations after subtracting the lifted data. For the boundary condition \eqref{eq:Teps}, we introduce $\widetilde{T}=\widetilde{T}(x)$ as the solution to the problem
% \begin{align}
% 	&\partial_t \widetilde{T} - \Delta \widetilde{T} = 0,\quad t>0, \; x\in \Omega, \label{eq:Gequation}\\
% 	&\widetilde{T}(t,x) = T_b(t,x), \quad t>0,\; x\in\partial \Omega, \nonumber \\
% 	& \widetilde{T}(t=0,x)=0,\quad \; x\in\Omega. \nonumber
% \end{align}
\begin{align}
	& \Delta \widetilde{T} = 0,\quad  x\in \Omega, \label{eq:Gequation}\\
	&\widetilde{T}(x) = T_b(x), \quad x\in\partial \Omega. \nonumber 
\end{align}
Owing to the well prepared boundary assumption \eqref{eq:wellbc}, we take $\widetilde{\psi}=\widetilde{T}^4$. After we introduce these variables, we can see that on the boundary,
\begin{align*}
	&T_\varepsilon - \widetilde{T} =0, \text{ on } \partial\Omega,\\
	&\psi_\varepsilon - \widetilde{\psi}= (1-\alpha) (L(\psi_\varepsilon - \widetilde{\psi}))(t,x,\beta), \text{ on } \in \Sigma_-. 
\end{align*}

In order to find the Galerkin approximations of \eqref{eq:Teps}-\eqref{eq:psieps}, we also need to define some truncation operators. We can take the complete set of the eigenvectors $\{w_k=w_k(x)\}_{k=1}^\infty$ of $H_0^1(\Omega)$ which is also an orthonormal basis in $L^2(\Omega)$. We take the operator $\mathbb{P}_m:L^2(\Omega)\to L^2(\Omega)$ as 
\begin{align*}
	\mathbb{P}_m f = \sum_{k\le m} (f,w_k) w_k(x),  
\end{align*}
where $(\cdot,\cdot)$ is the inner product in $L^2(\Omega)$. Note that $\mathbb{P}_m f =0$ on the boundary $\partial\Omega$.

By Lemma \ref{lem:basis} in Appendix \ref{appendixb}, we can also find an orthonormal basis $\{\varphi_k\}_{k=1}^\infty$ in $L^2(\Omega\times\mathbb{S}^2)$.
We define the operator $\mathbb{Q}_m$ as
 \begin{align}\label{eq:operatorQ}
 	\mathbb{Q}_m\psi(x,\beta)=\sum_{k=1}^m ((\psi,\varphi_k))\varphi_k(x,\beta).
 \end{align}
 Here $((\cdot,\cdot))$ denotes the inner product in the space $L^2(\Omega\times\mathbb{S}^2)$.
%  Notice that since any function $f\in H_0^1(\Omega)$ independent of $\beta$ satisfies
%  \begin{align}\label{eq:reflect0}
%  	f|_{\Sigma^+} = (1-\alpha) f|_{\Sigma_-}=0,
%  \end{align}
%  and thus also belongs to $V_\alpha$, so that $\mathbb{P}_m$ can also be applied to functions in $H_0^1(\Omega)$. Suppose 
%  \begin{align*}
%  	\mathbb{P}_m f(x) = \sum_{k=1}^m (f,\varphi_k)\varphi_k(x,\beta).
%  \end{align*}
%  We can integrate it over $\mathbb{S}^2$ to get
%  \begin{align*}
%  	4\pi \mathbb{P}_m f(x) = \sum_{k=1}^m (f,\varphi_k) \int_{\mathbb{S}^2} \varphi d\beta.
%  \end{align*}
% Therefore, any function in $H_0^1(\Omega)$ can be written as the linear combinations of $\{\langle \varphi_k \rangle\}_{k=1}^\infty$.
% We take $\{w_k=w_k(x)\}_{k=1}^\infty$ to be an orthogonal basis of $H_0^1(\Omega)$. For example, we can take it to be the complete set of appropriate normalized eigenfunctions of $-\Delta$ in $H_0^1(\Omega)$. We define the operator $\mathbb{P}_m$ on bounded domain as
% \begin{align*}
% 	\mathbb{P}_m f(x) = \sum_{k=1}^m (f,w_k) w_k(x), \text{ for any } f \in L^2, f|_{\partial\Omega}=0.
% \end{align*}
% with $(\cdot,\cdot)$ being the $L^2$ inner product. With this definition, we take
% \begin{align*}
% 	T_\varepsilon^m = \mathbb{P}_m (T_\varepsilon - \widetilde{T}) + \widetilde{T}.
% \end{align*}
After these preparations, now we proceed to prove Theorem \ref{thmexistd}.

As in the proof of Theorem \ref{thm:existencet},
	we take the Galerkin approximations to be 
	\begin{align}
		\partial_t T_\varepsilon^m =~& \Delta T_\varepsilon^m + \frac{1}{\varepsilon^2}\left(\int_{\mathbb{S}^2} \mathbb{P}_m \psi_\varepsilon^m - \mathbb{P}_m ((T_\varepsilon^m)^4)d\beta\right), \label{eq:gl21}\\
		\partial_t \psi_\varepsilon^m + \frac{1}{\varepsilon}\beta \cdot \nabla \psi_\varepsilon^m =~& -\frac{1}{\varepsilon^2}\left(\psi_\varepsilon^m - \mathbb{Q}_m\mathbb{P}_m((T_\varepsilon^m)^4)\right), \label{eq:gl22}
	\end{align} 
	with initial conditions 
	\begin{align*}
		T_{\varepsilon}^m(t=0,x) =~& \mathbb{P}_mT_{\varepsilon0}(x), \\
		\psi_{\varepsilon}^m(t=0,x,\beta) =~&\mathbb{Q}_m \psi_{\varepsilon0}(\beta,x),
	\end{align*}	
	and boundary conditions
	\begin{align}
		T_{\varepsilon}^m(t,x) =~& T_b, \text{ for } x\in \partial\Omega, \label{eq:btgl}\\
		\psi_\varepsilon^m |_{\Sigma_-}=~&\alpha \psi_b + L\psi_\varepsilon^m|_{\Sigma_+}.\label{eq:bpsigl} 
	\end{align}

	We can take %$T_\varepsilon^m$ and $\psi_\varepsilon^m$ are given by
	\begin{align*}
		&T_\varepsilon^m - \widetilde{T} = \sum_{k=1}^m d_k(t)\langle \varphi_k \rangle(x) , \\
		% & \quad \text{and}\\
		&\psi_\varepsilon^m - \widetilde{\psi}= \sum_{k=1}^m \phi_k(t)\varphi_k(\beta,x),
	\end{align*} 
	into \eqref{eq:gl21} and \eqref{eq:gl22}, and get a system of ordinary differential equations for $d_k$ and $\phi_k$. 	From the Cauchy-Lipschitz theorem, the ODE system has a global unique solution if $T_\varepsilon^m$ and $\psi_\varepsilon^m$ is bounded uniformly in time, which will be shown below.
	It follows that the system \eqref{eq:gl21}-\eqref{eq:gl22} has a global unique solutions. \tcb{Note the solutions $T_\varepsilon^m,\psi_\varepsilon^m\ge 0$, which is shown in Appendix A.}

	Next we derive the energy estimate for the system \eqref{eq:gl21}-\eqref{eq:gl22}.
	We multiply equation \eqref{eq:gl21} by $(T_\varepsilon^m)^4 - \widetilde{T}^4$ and equation \eqref{eq:gl22} by $\psi_\varepsilon^m - \widetilde{\psi}$, integrate over $[0,t]\times\Omega$ and $[0,t]\times\Omega\times\mathbb{S}^2$ respectively, and add the results together. We obtain
	\begin{align}
		\int_{\Omega}&\left(\frac{(T_\varepsilon^m)^5}{5} - \widetilde{T}^4T_\varepsilon^m\right)(t) dx + \frac{1}{2} \iint_{\Omega \times \mathbb{S}^2}(\psi_\varepsilon^m - \widetilde{\psi})^2(t) d\beta dx \nonumber \\
		=&\int_{\Omega}\frac{(T_{\varepsilon0}^m)^5}{5} dx + \frac{1}{2} \iint_{\Omega \times \mathbb{S}^2}(\psi_{\varepsilon0}^m)^2 d\beta dx - \frac{1}{2\varepsilon} \int_0^t\iint_{\Omega\times\mathbb{S}^2} \beta \cdot \nabla (\psi_\varepsilon^m - \widetilde{\psi})^2 d\beta dx d\tau  \nonumber\\
		&+ \int_0^t \iint_{\Omega\times\mathbb{S}^2} \beta \cdot \nabla \widetilde{\psi} (\psi_\varepsilon^m - \widetilde{\psi}) d\beta dxd\tau+\int_0^t \int_{\Omega} ((T_\varepsilon^m)^4 - \widetilde{T}^4)\Delta T_\varepsilon^m dxd\tau  \nonumber\\
		&- \frac{1}{\varepsilon^2}\int_0^t\iint_{\Omega\times\mathbb{S}^2}\left(\psi_\varepsilon^m - \mathbb{P}_m((T_\varepsilon^m)^4)\right) (\psi_\varepsilon^m - \widetilde{\psi} - \mathbb{P}_m ((T_\varepsilon^m)^4) + \widetilde{T}^4)d\beta dxd\tau \nonumber\\
		=& I_1+I_2+I_3+I_4+I_5+I_6. \label{eq:i123456}
	\end{align}
	The first term on the left hand side of the above equation can be estimated by
	\begin{align}
		\int_{\Omega}&\left(\frac{(T_\varepsilon^m)^5}{5} - \widetilde{T}^4T_\varepsilon^m\right) dx \ge \int_{\Omega}\frac{(T_\varepsilon^m)^5}{5} dx - \int_{\Omega}\frac{1}{2}\frac{(T_\varepsilon^m)^5}{5} dx - 2\int_\Omega \frac{(\widetilde{T}^4)^{\frac{5}{4}}}{\frac{5}{4}} dx \nonumber \\
		\ge& \frac{1}{10} \|T_\varepsilon^m\|_{L^5(\Omega)}^5 - \frac{8}{5} \|\widetilde{T}\|_{L^5(\Omega)}^5. \label{eq:leftest}
	\end{align}
	The second term on the left can be estimated by 
	\begin{align}
		\frac{1}{2} \iint_{\Omega \times \mathbb{S}^2}(\psi_\varepsilon^m - \widetilde{\psi})^2d\beta dx \ge& \frac{1}{2} \iint_{\Omega \times \mathbb{S}^2}(\psi_\varepsilon^m)^2d\beta dx - \frac{1}{2} \iint_{\Omega \times \mathbb{S}^2}(\widetilde{\psi})^2d\beta dx \nonumber\\
		\ge&\frac{1}{2} \iint_{\Omega \times \mathbb{S}^2}(\psi_\varepsilon^m)^2d\beta dx - 2\pi \int_{\Omega} \widetilde{T}^8 dx . \label{eq:left2nd}
	\end{align}
	Next we estimate the right terms of \eqref{eq:i123456}. First, for $I_1$ and $I_2$, we can use the property of the operators $\mathbb{P}_m$ and $\mathbb{Q}_m$ to get
	 % $\|\mathbb{P}_m \varphi\|_{L^2(\Omega\times\mathbb{S}^2)} \le \|\varphi\|_{L^2(\Omega\times\mathbb{S}^2)}$ to get
	\begin{align}
		I_1 + I_2 =& \int_{\Omega}\frac{(T_{\varepsilon0}^m)^5}{5} dx + \frac{1}{2} \iint_{\Omega \times \mathbb{S}^2}(\psi_{\varepsilon0}^m)^2 d\beta dx 
		\le \frac{1}{5}\|T_{\varepsilon0}\|_{L^5(\Omega)}^5 + \frac{1}{2}\|\psi_{\varepsilon0}\|_{L^2(\Omega\times\mathbb{S}^2)}^2.
	\end{align}
	Using the boundary condition \eqref{eq:bpsigl}, $I_3$ can be calculated as
	\begin{align}
		I_3 =& - \frac{1}{2\varepsilon} \int_0^t\iint_{\Omega\times\mathbb{S}^2} \beta \cdot \nabla (\psi_\varepsilon^m - \widetilde{\psi})^2 d\beta dx d\tau = -\frac{1}{2\varepsilon} \int_0^t \iint_{\Sigma} \beta \cdot n (\psi_\varepsilon^m - \widetilde{\psi})^2 d\beta d\sigma_xd\tau \nonumber \\
		=& -\frac{1}{2\varepsilon} \int_0^t \iint_{\Sigma_+} \beta \cdot n (\psi_\varepsilon^m - \widetilde{\psi})^2 d\beta d\sigma_x d\tau  -\frac{1}{2\varepsilon} \int_0^t \iint_{\Sigma_-} \beta \cdot n (\psi_\varepsilon^m - \widetilde{\psi})^2 d\beta d\sigma_x d\tau \nonumber\\
		=&-\frac{1}{2\varepsilon} \int_0^t \iint_{\Sigma_+} \beta \cdot n (\psi_\varepsilon^m - \psi_b)^2 d\beta d\sigma_x d\tau  \nonumber\\
		&-\frac{1}{2\varepsilon} \int_0^t \iint_{\Sigma_-} \beta \cdot n (1-\alpha)^2(\psi_\varepsilon^m(\beta') -\psi_b)^2 d\beta d\sigma_x d\tau \nonumber\\
		=&-\frac{2\alpha - \alpha^2}{2\varepsilon} \int_0^t \iint_{\Sigma_+} \beta \cdot n (\psi_\varepsilon^m - \psi_b)^2 d\beta d\sigma_x d\tau . \label{eq:i3est}
	\end{align}
	The term $I_4$ can be estimated by using the fact that $\tilde{\psi}$ is independent of $\beta$ as
	\begin{align}
		I_4 =& \int_0^t \iint_{\Omega\times\mathbb{S}^2} \beta \cdot \nabla \widetilde{\psi} (\psi_\varepsilon^m - \widetilde{\psi}) d\beta dxd\tau \nonumber\\
		=& \int_0^t \iint_{\Omega\times\mathbb{S}^2} \beta \cdot \nabla \widetilde{\psi} (\psi_\varepsilon^m - \mathbb{P}_m((T_\varepsilon^m)^4)) d\beta dxd\tau \nonumber \\
		\le& 2\varepsilon^2 \int_0^t \iint_{\Omega\times\mathbb{S}^2} |\beta \cdot \nabla \widetilde{\psi}|^2 d\beta dxd\tau + \frac{1}{2\varepsilon^2} \int_0^t \iint_{\Omega\times\mathbb{S}^2} (\psi_\varepsilon^m - \mathbb{P}_m((T_\varepsilon^m)^4))^2 d\beta dxd\tau \nonumber \\
		\le& 8\pi \varepsilon^2 \int_0^t \iint_{\Omega\times} | \nabla \widetilde{T}^4|^2  dxd\tau + \frac{1}{2\varepsilon^2} \int_0^t \iint_{\Omega\times\mathbb{S}^2} (\psi_\varepsilon^m - \mathbb{P}_m((T_\varepsilon^m)^4))^2 d\beta dxd\tau	.\label{eq:i4est}
	\end{align}
	We estimate the term $I_5$ by
	\begin{align*}
		I_5 =& \int_0^t \int_{\Omega}  ((T_\varepsilon^m)^4 - \widetilde{T}^4)\Delta T_\varepsilon^m dxd\tau\\
		=& -\int_0^t \int_{\Omega} \nabla ((T_\varepsilon^m)^4 - \widetilde{T}^4) \nabla T_\varepsilon dxd\tau \\
		=& -  \int_0^t\int_\Omega \frac{16}{25} |\nabla (T_\varepsilon^m)^{\frac{5}{2}}|^2dxd\tau  - \int_{0}^t \int_\Omega T_\varepsilon^m \Delta \widetilde{T}^4  dxd\tau + \int_0^t \int_{\partial\Omega} T_b n\cdot \nabla \widetilde{T}^4 d\sigma_xd\tau \\
		\le& - \int_0^t\int_\Omega \frac{16}{25} |\nabla (T_\varepsilon^m)^{\frac{5}{2}}|^2dxd\tau  + \int_{0}^t \int_\Omega \frac{(T_\varepsilon^m)^5}{5} dxd\tau + \int_0^t\int_{\Omega} \frac{(\Delta \widetilde{T}^4)^{\frac{5}{4}}}{\frac{5}{4}}  dxd\tau \\
		&+ \int_0^t \int_{\partial\Omega} 4T_b^4 n\cdot \nabla \widetilde{T} d\sigma_x d\tau.
	\end{align*}
	Multiplying equation \eqref{eq:Gequation} by $\widetilde{T}^4$ and integrating over $[0,t]\times \Omega$ leads to 
	\begin{align*}
		\int_{\Omega} \frac{\widetilde{T}^5}{5} (t)dx  + \frac{16}{25} \int_0^t\int_{\Omega} |\nabla \widetilde{T}^{\frac{5}{2}}|^2 dxd\tau  - \int_0^t\int_{\partial\Omega} T_b^4 n\cdot \nabla \widetilde{T} d\sigma_x d\tau=0,
	\end{align*}
	thus
	\begin{align}
		I_5 \le& -  \int_0^t\int_\Omega \frac{16}{25} |\nabla (T_\varepsilon^m)^{\frac{5}{2}}|^2dxd\tau  + \int_{0}^t \int_\Omega \frac{(T_\varepsilon^m)^5}{5} dxd\tau + \int_0^t\int_{\Omega} \frac{(\Delta \widetilde{T}^4)^{\frac{5}{4}}}{\frac{5}{4}}  dxd\tau \nonumber\\
		&+ \frac{4}{5} \int_{\Omega} {\widetilde{T}^5} (t)dx  + \frac{64}{25}\int_0^t \int_{\Omega} |\nabla \widetilde{T}^{\frac{5}{2}}|^2 dxd\tau. \label{eq:i5est}
	\end{align}
	The term $I_6$ can be treated by
	\begin{align}\label{eq:i6est}
		I_6=&- \frac{1}{\varepsilon^2}\int_0^t\iint_{\Omega\times\mathbb{S}^2}\left(\psi_\varepsilon^m - \mathbb{P}_m((T_\varepsilon^m)^4)\right) (\psi_\varepsilon^m - \widetilde{\psi} - \mathbb{P}_m ((T_\varepsilon^m)^4) + \widetilde{T}^4)d\beta dxd\tau \nonumber\\
		=& -\frac{1}{\varepsilon^2} \int_0^t \iint_{\Omega\times\mathbb{S}^2} (\psi_\varepsilon^m - \mathbb{P}_m((T_\varepsilon^m)^4))^2d\beta dxd\tau.
	\end{align}
	Here we use $\tilde{\psi}=\widetilde{T}^4$ in the above equality. Taking the estimates \eqref{eq:leftest}-\eqref{eq:i6est} into \eqref{eq:i123456} leads to the estimate
	\begin{align*}
		&\frac{1}{10}\|T_\varepsilon^m(t)\|_{L^5(\Omega)}^5 + \frac{1}{2}\|\psi_\varepsilon^m(t)\|^2_{L^2(\Omega\times\mathbb{S}^2)} + \frac{16}{25}\int_0^t \|\nabla (T_\varepsilon^m)^{\frac{5}{2}}\|_{L^2(\Omega)}^2 d\tau \\
		&\quad+ \frac{1}{\varepsilon^2} \int_0^t \iint_{\Omega\times\mathbb{S}^2} (\psi_\varepsilon^m - \mathbb{P}_m((T_\varepsilon^m)^4))^2 d\beta dxd\tau \\
		&\quad +\frac{2\alpha - \alpha^2}{2\varepsilon} \int_0^t \iint_{\Sigma_+} \beta \cdot n (\psi_\varepsilon^m - \psi_b)^2 d\beta d\sigma_x d\tau \nonumber \\
		&\qquad\le \frac{1}{5} \|T_{\varepsilon0}\|_{L^5}^5 + \frac{1}{2}\|\psi_{\varepsilon0}\|_{L^2(\Omega\times\mathbb{S}^2)}^2+ \frac{8}{5}\|\widetilde{T}\|_{L^5(\Omega)}^5 + 2\pi \|\widetilde{T}\|_{L^8(\Omega)}^8 \nonumber \\
		&\qquad+ 8\pi \varepsilon^2 \int_0^t \iint_{\Omega} | \nabla \widetilde{T}^4|^2  dxd\tau + \frac{1}{5} \int_0^t\int_{\Omega} \|T_\varepsilon^m\|_{L^5(\Omega)}^5 + \frac{4}{5}\int_0^t\|\Delta \widetilde{T}^4\|_{L^{\frac{5}{4}}(\Omega)}^{\frac{5}{4}}d\tau \nonumber\\
		&\qquad+ \frac{4}{5} \|\widetilde{T}\|_{L^5(\Omega)}^5 + \frac{64}{25}\int_0^t \|\nabla \widetilde{T}^{\frac{5}{2}}\|_{L^2(\Omega)}^2d\tau.
	\end{align*}
	Since $\widetilde{T}$ is the solution of the Laplace's equation \eqref{eq:Gequation}, it is smooth and thus the terms including $\widetilde{T}$ of the above inequality is bounded. We can apply Gronwall's inequality to obtain
	\begin{align}\label{eq:gl2est}
		&\|T_{\varepsilon}^m(t)\|_{L^5(\Omega))}^5 + \|\psi_\varepsilon^m(t)\|_{L^2(\Omega\times\mathbb{S}^2)}^2 + \int_0^t \|\nabla (T_\varepsilon^m)^{\frac{5}{2}}\|_{L^2(\Omega)}^2 d\tau \nonumber\\
		&\qquad+ \frac{1}{\varepsilon^2} \int_0^t \|\psi_\varepsilon^m - \mathbb{P}_m((T_\varepsilon^m)^4)\|_{L^2(\Omega\times\mathbb{S}^2)}^2 d\tau \nonumber \\
		&\qquad+\frac{2\alpha - \alpha^2}{2\varepsilon} \int_0^t \|\psi_\varepsilon^m - \psi_b\|_{L^2(\Sigma_+;|n\cdot\beta| d\beta d\sigma_x)}^2 d\tau \nonumber\\
		&\quad \le C e^{Ct}(\|T_{\varepsilon0}\|_{L^5}^5 + \|\psi_{\varepsilon0}\|_{L^2(\Omega\times\mathbb{S}^2)}^2),
	\end{align}
	\tcb{where $C$ depends on $\widetilde{T}.$}

	From this estimate, we can get the same bounds on $T_\varepsilon^m$ and $\psi_\varepsilon^m$ inside the domain, i.e. \eqref{eq:unibd1}-\eqref{eq:unibd4} still hold. We can follow the same proof of Theorem \ref{thm:existencet} to get 
	\begin{align}
		&T_\varepsilon^m \to T_\varepsilon \text{ almost everywhere}, \label{eq:gl2strongt}\\
		&\psi_\varepsilon^m \rightharpoonup \psi_\varepsilon \text{ weakly in }L^2_{\loc}([0,\infty);L^2(\Omega\times\mathbb{S}^2)),\label{eq:gl2strongpsi}
	\end{align}
	and that $(T_\varepsilon,\psi_\varepsilon)$ satisfies \eqref{eq:weakd1}-\eqref{eq:weakd2}.

	Next we consider the boundary conditions. According to \eqref{eq:gl2est}, we have
	\begin{align*}
		(T_\varepsilon^m)^{\frac{5}{2}} \text{ is uniformly bounded in }L^2_{\loc}([0,\infty);{H}^1(\Omega)]),
	\end{align*}
	so by the Rellich–Kondrachov emdedding theorem, as $m\to\infty$,
	\begin{align*}
		(T_\varepsilon^m)^{\frac{5}{2}} \to T_\varepsilon^{\frac{5}{2}}, \text{ strongly in } L^2_{\loc}([0,\infty);{H}^{1-\delta}(\Omega)])
	\end{align*}
	for $\delta>0$ small. We can thus use the continuity of the trace operator to get
	\begin{align}\label{eq:strongtracethm2}
		\tcb{\gamma^1 (T_\varepsilon^m)^{\frac{5}{2}} = T_b^{\frac52} \to \gamma^1 T_\varepsilon^{\frac52}=T_b^{\frac52},}
		% (\gamma^1 T_\varepsilon^m)^{\frac{5}{2}} = T_b^{\frac{5}{2}} \to \gamma^1 T_\varepsilon^{\frac{5}{2}} = (\gamma^1 T_\varepsilon)^{\frac{5}{2}}, 
		\text{ strongly in }L^2_{\loc}([0,\infty);{H}^{\frac{1}{2}-\delta}(\Omega)]).
	\end{align}
	Therefore, we get
	\begin{align*}
		\gamma^1 T_\varepsilon = T_b.
	\end{align*}
	To pass to the limit on the boundary of $\psi_\varepsilon^m$, we learn from Lemma \ref{lm:trace} that 
	\begin{align*}
		\|\psi_\varepsilon^m|_{\partial\Omega}\|_{L^2(\mathbb{R}_+\times\Omega\times\mathbb{S}^2,|n\cdot \beta|d\beta d\sigma_xdt)}
	\end{align*}
	is bounded.
 
	 Therefore, there exists a subsequence $\{\psi_\varepsilon^{m_k}\}_{k>0}$ such that 
	\begin{align*}
		\gamma^2 \psi_\varepsilon^{m_k} \rightharpoonup \overline{\gamma^2 \psi_\varepsilon^{m_k}} \text{ weakly in } L^2_{\loc}([0,\infty);L^2(\Sigma;|n\cdot\beta| d\beta d\sigma_x d\tau)).
	\end{align*}
	We can thus take the weak limit in \eqref{eq:bpsigl} to obtain
	\begin{align}\label{eq:Mbd}
		 \overline{\gamma^2 \psi_\varepsilon^{m_k}} |_{\Sigma_-} = \alpha \psi_b + L \overline{\gamma^2 \psi_\varepsilon^{m_k}} |_{\Sigma_+}.
	\end{align}
	To show $ \overline{\gamma^2 \psi_\varepsilon^{m_k}} =\gamma \psi_\varepsilon$, we use the fact that $(T_\varepsilon,\psi_\varepsilon)$ solves
	\begin{align*}
	 	\varepsilon \partial_t \psi_\varepsilon + \beta \cdot \nabla \psi_\varepsilon = -\frac{1}{\varepsilon}(\psi_\varepsilon - T_\varepsilon^4).
	 \end{align*} 
	We apply test function $\rho \in C^\infty([0,\infty)\times\Omega\times\mathbb{S}^2)$
	on the above equation and deduce
	\begin{align*}
		\iint_{\Omega\times\mathbb{S}^2}&\psi_\varepsilon(t) \rho(t) d\beta dx - \iint_{\Omega}\psi_{\varepsilon0} \rho(0) d\beta dx - \int_0^t\iint_{\Omega\times\mathbb{S}^2} \psi_\varepsilon \partial_t \rho d\beta dxd\tau \\
		&- \frac{1}{\varepsilon} \int_0^t\int_{\Omega\times\mathbb{S}^2} \psi_\varepsilon \beta \cdot \nabla \rho d\beta dxd\tau + \frac{1}{\varepsilon}\int_0^t \iint_{\Sigma} (n\cdot\beta) \gamma^2 \psi_\varepsilon \rho d\beta dxd\tau\\
		& = -\frac{1}{\varepsilon^2}\int_0^t\int_{\Omega\times\mathbb{S}^2}(\psi_\varepsilon^m - \mathbb{P}_m((T_\varepsilon)^4)) \rho d\beta dxd\tau.
	\end{align*}
	We can also apply the same test function on the equation \eqref{eq:gl22} and get 
	\begin{align*}
		\iint_{\Omega\times\mathbb{S}^2}&\psi_\varepsilon^m(t) \rho(t) d\beta dx - \iint_{\Omega}\psi_{\varepsilon0}^{m} \rho(0) d\beta dx - \int_0^t\iint_{\Omega\times\mathbb{S}^2} \psi_\varepsilon^m \partial_t \rho d\beta dxd\tau \\
		&- \frac{1}{\varepsilon} \int_0^t\int_{\Omega\times\mathbb{S}^2} \psi_\varepsilon^m \beta \cdot \nabla \rho d\beta dxd\tau + \frac{1}{\varepsilon}\int_0^t \iint_{\Sigma} (n\cdot\beta) \gamma^2\psi_\varepsilon^m \rho d\beta dxd\tau\\
		& = -\frac{1}{\varepsilon^2}\int_0^t\int_{\Omega\times\mathbb{S}^2}(\psi_\varepsilon^m - \mathbb{P}_m((T_\varepsilon^m)^4)) \rho d\beta dxd\tau.
	\end{align*}
	Passing $m\to \infty$ in the above equation and comparing with the previous one lead to
	\begin{align*}
		\overline{\gamma^2 \psi_\varepsilon^m} = \gamma^2 \psi_\varepsilon.
	\end{align*}
	This combing with \eqref{eq:Mbd} implies that $\psi_\varepsilon$ satisfies
	% satisfies equation \eqref{eq:psieps} with the boundary 
	\begin{align}
		 {\gamma^2 \psi_\varepsilon} |_{\Sigma_-} = \alpha \psi_b + L \gamma^2 \psi_\varepsilon |_{\Sigma_+},
	\end{align}	
	on the boundary.
 The energy inequality can be shown as in the proof of Theorem \ref{thm:existencet}, except that here we need to use
	\begin{align*}
		\int_0^t \|\psi_\varepsilon - \psi_b \|_{L^2(\Sigma_+;|n\cdot \beta|d\beta d\sigma_x)} d\tau \le \liminf_{m\to\infty} \int_0^t \|\psi_\varepsilon^m - \psi_b \|_{L^2(\Sigma_+;|n\cdot \beta|d\beta d\sigma_x)} d\tau,
	\end{align*}
	which is due to the weak convergence of $\psi_\varepsilon^m$ on the boundary. 
	% as well as $\limsup(a_m+b_m) \le \limsup a_m + \liminf b_m$, we will arrive at \eqref{eq:estdirichlet} and finish the proof.

	% To pass to the limit on the boundary of $\psi_\varepsilon^m$, we recall that from the strong convergence \eqref{eq:gl2strongpsi} of $\psi_\varepsilon^m$ and the strong convergence \eqref{eq:weakconv3}, we can see that
	% \begin{align*}
	% 	\psi_\varepsilon^m \to \psi_\varepsilon, \text{ strongly in } {W}^2.
	% \end{align*}
	% Notice that here we use the fact that 
	% \begin{align*}
	% 	(\varepsilon \partial_t + \beta \cdot \nabla )\psi_\varepsilon^m = -\frac{1}{\varepsilon} (\psi_\varepsilon^m - \mathbb{P}_m((T_\varepsilon^m)^m)) \in L^2_{\loc}([0,\infty);L^2(\Omega\times\mathbb{S}^2)).
	% \end{align*}
	% Using Lemma \ref{lm:trace}, we can get
	% \begin{align*}
	% 	\gamma \psi_\varepsilon^m \to \gamma \psi_\varepsilon, \text{ strongly in }L^2_{\loc}([0,\infty);L^2(\Sigma;|\beta\cdot n|d\beta d\sigma_x)),
	% \end{align*}
	% so it also holds that
	% \begin{align*}
	% 	\gamma L\psi_\varepsilon^m \to \gamma L\psi_\varepsilon, \text{ strongly in }L^2_{\loc}([0,\infty);L^2(\Sigma;|\beta\cdot n|d\beta d\sigma_x)),
	% \end{align*}
	% Recalling the boundary condition \eqref{eq:bpsigl}, we get
	% \begin{align}
	% 	\psi_\varepsilon|_{\Sigma_-} = \alpha \psi_b + (1-\alpha) L\psi_\varepsilon |_{\Sigma_+}.
	% \end{align}
\end{proof} 

\begin{rem}
	In the above proof we assume the boundary condition to be well prepared $\psi_b=T_b^4$. \tcb{When the boundary condition is not well-prepared, a similar estimate on} the term \eqref{eq:i3est} gives
	% can be estimated by 
	\begin{align*}
		I_3 =& - \frac{1}{2\varepsilon} \int_0^t\iint_{\Omega\times\mathbb{S}^2} \beta \cdot \nabla (\psi_\varepsilon^m - \widetilde{\psi})^2 d\beta dx d\tau = -\frac{1}{2\varepsilon} \int_0^t \iint_{\Sigma} \beta \cdot n (\psi_\varepsilon^m - \widetilde{\psi})^2 d\beta d\sigma_xd\tau \nonumber \\
		=& -\frac{1}{2\varepsilon} \int_0^t \iint_{\Sigma_+} \beta \cdot n (\psi_\varepsilon^m - \widetilde{\psi})^2 d\beta d\sigma_x d\tau  -\frac{1}{2\varepsilon} \int_0^t \iint_{\Sigma_-} \beta \cdot n (\psi_\varepsilon^m - \widetilde{\psi})^2 d\beta d\sigma_x d\tau \nonumber\\
		=&-\frac{1}{2\varepsilon} \int_0^t \iint_{\Sigma_+} \beta \cdot n (\psi_\varepsilon^m - T_b^4)^2 d\beta d\sigma_x d\tau  \nonumber\\
		&-\frac{1}{2\varepsilon} \int_0^t \iint_{\Sigma_-} \beta \cdot n \left(\alpha(\psi_b - T_b^4)+(1-\alpha)(\psi_\varepsilon^m(\beta') -T_b^4) \right)^2 d\beta d\sigma_x d\tau \nonumber\\
		\le&-\frac{(2\alpha-\alpha^2)}{2\varepsilon} \int_0^t \iint_{\Sigma_+} \beta \cdot n (\psi_\varepsilon^m - T_b^4)^2 d\beta d\sigma_x d\tau \nonumber\\ 
		&+\frac{\alpha^2}{2\varepsilon} \int_0^t \iint_{\Sigma_+} \beta \cdot n (\psi_b - T_b^4)^2 d\beta d\sigma_x d\tau \nonumber\\
		&+\frac{2\alpha (1-\alpha)}{2\varepsilon} \int_0^t \iint_{\Sigma_+} \beta \cdot n (\psi_\varepsilon^m - T_b^4)(\psi_b-T_b^4) d\beta d\sigma_x d\tau\nonumber \\
		=& -\frac{(2\alpha-\alpha^2)}{2\varepsilon}  \int_0^t \iint_{\Sigma_+} \beta \cdot n \left(\psi_\varepsilon^m-T_b^4- \frac{1-\alpha}{2-\alpha}(\psi_b - T_b^4)\right)^2 d\beta d\sigma_x d\tau \nonumber\\
		&+\frac{\alpha}{2(2-\alpha)}  \int_0^t \iint_{\Sigma_+} \beta \cdot n (\psi_b - T_b^4)^2 d\beta d\sigma_x d\tau.
	\end{align*}
	So the estimate \eqref{eq:estdirichlet} holds with the constant $C$ depending on $\varepsilon$. We then get the existence of the weak solutions for fixed $\varepsilon$. However, the well-prepared assumption on the boundary conditions is required to study the diffusive limit in the section \ref{section3}.
\end{rem}

\subsection{Case of Robin boundary condition}
We now proceed to consider the case of Robin boundary condition \eqref{b2}.
 % on bounded domain $\Omega$. 
% We first define the trace operator on the boundary. After that we will give the definition of weak solutions. Finally we will give the existence theorem and its proof.
% We define an operator $\gamma^1$ mapping a function in $W^{1,p}(\Omega)$ to $L^p(\partial \Omega)$:
% \begin{align}\label{eq:gamma1def}
% 	\gamma^1: f \in W^{1,p}(\Omega) \mapsto f|_{\partial \Omega} \in L^p(\mathbb{T}^3).
% \end{align}
% From the trace theorem in Sobolev space, assuming $\Omega$ is bounded and $\partial \Omega$ is $C^1$, then for any function $f$ in $W^{1,p}$, there exists a unique bounded linear operator $\gamma^1$ being the trace operator. 
We first give the definition of the weak solution.
\begin{definition}\label{def2}
	Assume $\partial\Omega\in C^1$. \tcr{ Let $0 \le T_{\varepsilon0} \in L^5(\Omega)$ and $0 \le\psi_{\varepsilon0} \in L^2(\Omega \times \mathbb{S}^2)$. Let $0 \le T_b \in L_{\loc}^5([0,\infty);L^5(\partial \Omega))$ and $0 \le \psi_b \in L_{\loc}^2([0,\infty); L^2(\Sigma_-;|n\cdot\beta| d\beta d\sigma_x))$.} We say that $(T_\varepsilon,\psi_\varepsilon)$ is a weak solution of the system \eqref{eq:Teps}-\eqref{eq:psieps} with initial conditions \eqref{eq:ic1}-\eqref{eq:ic2} and boundary conditions \eqref{bpsi}, \eqref{b2} if 	\begin{align*}
		&T_\varepsilon \in L_{\loc}^\infty (0,\infty;L^5(\Omega)),
		% \cap L^5_{\loc}(0,\infty;L^{15}(\Omega)), 
		\quad T_\varepsilon^{\frac{5}{2}} \in L^2_{\loc}(0,\infty;{H}^1(\Omega)),\\
		&\psi_\varepsilon \in L_{\loc}^\infty(0,\infty; L^2(\Omega \times \mathbb{S}^2))\cap W^2_{\loc}([0,\infty)\times\Omega\times\mathbb{S}^2),
	\end{align*}
	and it solves \eqref{eq:Teps}-\eqref{eq:psieps} in the sense of distributions, i.e. for any test functions $\varphi  \in C^\infty([0,\infty),\Omega)$ and $\rho\in C^\infty([0,\infty);\Omega\times \mathbb{S}^2)$, the following equations hold:
	\begin{align}
		&-\iint_{[0,\infty)\times\Omega} \left(T_\varepsilon\partial_t \varphi + T_\varepsilon \Delta \varphi + \frac{1}{\varepsilon^2}\int_{\mathbb{S}^2} \varphi(\psi_\varepsilon - T_\varepsilon^4) d\beta \right) dxdt \nonumber\\
		&\quad- \iint_{[0,\infty)\times \partial \Omega} \varphi \cdot \frac{T_b-(\gamma^1T_\varepsilon)}{\varepsilon^r} d\sigma_x dt  + \iint_{[0,\infty)\times \partial \Omega} (\gamma^1 T_\varepsilon) n\cdot \nabla \varphi d\sigma_x dt \nonumber\\
		&\qquad= \int_{\Omega} T_{\varepsilon0} \varphi(0,\cdot)dx, \label{eq:weakr1}\\
		&-\iiint_{[0,\infty)\times\Omega \times \mathbb{S}^2} \left(\psi_\varepsilon \partial_t \rho + \frac{1}{\varepsilon} \psi_\varepsilon \beta \cdot \nabla \rho - \frac{1}{\varepsilon^2} \rho(\psi_\varepsilon-T_\varepsilon^4)\right) d\beta dxdt \nonumber \\
		&\quad +\iiint_{[0,\infty)\times \Sigma} (n \cdot \beta)\rho \cdot (\gamma^2\psi_\varepsilon) d\beta d\sigma_xdt = \iint_{\Omega\times \mathbb{S}^2} \psi_{\varepsilon0}\rho(0,\cdot) dx, \label{eq:weakr2}
	\end{align}
	where
	\begin{align}\label{eq:thm3bd}
		\gamma^2 \psi_\varepsilon \big|_{\Sigma_{-}} = \alpha \psi_b + (1-\alpha) \gamma^2 L\psi_\varepsilon\big|_{\Sigma_{+}},
	\end{align}
	with the reflection operator $L$ defined in \eqref{eq:reflectop}.
	% \[Lf(x,\beta):=f(x,\beta-2(n(x)\cdot \beta)n(x)).\]
\end{definition}

Next, we prove the following existence theorem.
\begin{theorem}\label{thmexistr}
Assume $\partial\Omega \in C^1$. \tcr{ Let $0 \le T_{\varepsilon0} \in L^5(\Omega)$ and $0 \le\psi_{\varepsilon0} \in L^2(\Omega \times \mathbb{S}^2)$. Let $0 \le T_b \in L_{\loc}^5([0,\infty);L^5(\partial \Omega))$ and $0 \le \psi_b \in L_{\loc}^2([0,\infty); L^2(\Sigma_-;|n\cdot\beta| d\beta d\sigma_x))$.} Then there exists a global nonnegative weak solution $(T_\varepsilon,\psi_\varepsilon)$ of the system \eqref{eq:Teps} and \eqref{eq:psieps} with initial conditions \eqref{eq:ic1}-\eqref{eq:ic2} and boundary conditions \eqref{bpsi}, \eqref{b2}. Moreover, the following energy inequality holds for all $t>0$:
	 	\begin{align}\label{eq:energyrobin}
		&\|T_{\varepsilon}(t)\|_{L^5(\Omega))}^5 + \|\psi_\varepsilon(t)\|_{L^2(\Omega\times\mathbb{S}^2)}^2 + \int_0^t \|\nabla T_\varepsilon^{\frac{5}{2}}\|_{L^2(\Omega)}^2 d\tau \nonumber\\
		&\qquad + \frac{1}{\varepsilon^2} \int_0^t \|\psi_\varepsilon - T_\varepsilon^4\|_{L^2(\Omega\times\mathbb{S}^2)}^2 d\tau+\frac{1}{\varepsilon^r} \int_0^t \|\gamma^1 T_\varepsilon - T_b\|_{L^5(\partial\Omega)}^5d\tau \nonumber \\
		&\qquad+\frac{2\alpha - \alpha^2}{2\varepsilon} \int_0^t \|\gamma^2\psi_\varepsilon - \psi_b\|_{L^2(\Sigma_+;|n\cdot\beta| d\beta d\sigma_x)}^2 d\tau \nonumber\\
		&\quad \le C (\|T_{\varepsilon0}\|_{L^5}^5 + \|\psi_{\varepsilon0}\|_{L^2(\Omega\times\mathbb{S}^2)}^2).
	\end{align}
	Here $C$ is a positive constant independent on $\varepsilon$.
% \begin{align}\label{eq:energyrobin}
% 	\frac{1}{5}& \|T_\varepsilon^m(t) \|_{L^5}^5 + \frac{1}{2}\|\psi_\varepsilon^m(t)\|_{L^2(\Omega\times\mathbb{S}^2)}^2 + \int_0^t \|\nabla T_\varepsilon^m)^{\frac{5}{2}}\|_{L^2(\Omega)}^2 d\tau \nonumber\\
% 	&\quad + \int_0^t \left\|\psi_\varepsilon^m - \mathbb{P}_m((T_\varepsilon^m)^4) \right\|_{L^2(\Omega\times\mathbb{S}^2)}^2d\tau + \frac{1}{5\varepsilon}\int_0^t \|\gamma T_\varepsilon\|_{L^5(\partial\Omega)}^5d\tau \nonumber\\
% 	&\quad+\frac{2\alpha-\alpha^2}{2\varepsilon}\int_0^t \iint_{\Sigma_+} n\cdot \beta (\gamma^2\psi_\varepsilon^m)^2 d\beta d\sigma_x d\tau \nonumber\\
% 	\le& \frac{1}{5} \|T_{\varepsilon0}^m \|_{L^5}^5 + \frac{1}{2}\|\psi_{\varepsilon0}\|_{L^2(\Omega\times\mathbb{S}^2)}^2 + \frac{1}{5\varepsilon^r} \|T_b\|_{L^5(\partial\Omega)}^5d\tau \nonumber
% 	\\	&+\frac{2\alpha(1-\alpha)}{(2\alpha-\alpha^2)\varepsilon} \int_0^t\iint_{\Sigma_-}|n\cdot\beta| \psi_b^2 d\beta d\sigma_xd\tau.
% \end{align}
\end{theorem}
\begin{proof}
The positivity of the solution is given on Appendix \ref{Postivsolution}. To deal with the case of Robin boundary condition \eqref{b2}, for $s\ge -\frac{1}{2}$,
% we assume the boundary $\partial\Omega$ to be $C^1$ function and 
we define the Robin map $\mathcal{R}:H^s(\partial\Omega) \to H^{s+\frac{3}{2}}(\Omega)$ (for example, see \cite{guo2014systems}) with $f=\mathcal{R} g$ as the weak solution for the equation
\begin{align}
	&\Delta f = 0, \text{ in }\Omega, \label{eq:robinf} \\
	&\varepsilon^r n\cdot \nabla f + f = g, \text{ on } \partial\Omega.
\end{align}
We define the operator $\Delta_r$ in $L^2(\Omega)$ by 
\begin{align*}
	&\Delta_r: \mathcal{D}(\Delta_r) \subset L^2(\Omega) \to L^2(\Omega),\\
	& \Delta_r = -\Delta,\,\, \mathcal{D}(\Delta_r)=\Big\{f\in H^1(\Omega):\Delta f \in L^2(\Omega),\;\varepsilon^r n\cdot \nabla f+f=0 \text{ on }\partial\Omega \Big\}.
\end{align*}
% It can be extended to a continuous operator $\Delta_R:H^1(\Omega) \to (H^1(\Omega))^*$ by
% \begin{align*}
% 	(\Delta_R f,h) = (\nabla f,\nabla h)_{\Omega} + \frac{1}{\varepsilon^r}(\gamma f, \gamma h)_{\partial\Omega}
% \end{align*}
The space $\mathcal{D}(\Delta_r)$ is equipped with the norm
\begin{align*}
	\|f\|_{\mathcal{D}(\Delta_r)} = (\|\nabla f\|_{L^2(\Omega)}^2 + \frac{1}{\varepsilon^r}\|\gamma f\|_{L^2(\partial\Omega)}^2)^{\frac{1}{2}}.
\end{align*}
for all $f,h\in H^1(\Omega)$.
With the above definitions, we can see that $T_\varepsilon - \mathcal{R}T_b \in \mathcal{D}(\Delta_r)$ satisfies the following condition on the boundary
\begin{align*}
	\varepsilon^r n\cdot \nabla (T_\varepsilon - \mathcal{R}T_b) + T_\varepsilon - \mathcal{R}T_b = 0, \text{ on }\partial \Omega.
\end{align*}

% Since here the above Robin boundary condition is not consistent with the reflective boundary condition, i.e. \eqref{eq:reflect0} fails. So we choose a different basis for $T_\varepsilon$. 
We take $\{w_m(x)\}_{m=1}^\infty$ to be an orthogonal basis in $\mathcal{D}(\Delta_r)$, for example we can take the complete set of the eigenvectors of $-\Delta_r$ as  the basis. It is also orthonormal in $L^2(\Omega)$.
% We take 
% \begin{align}\label{eq:T-rtb}
% 	&T_\varepsilon^m - \mathcal{R}T_b = \sum_{k=1}^m d_k(t) w_k(x). 
% \end{align}
We take the operator $\mathbb{P}_m$ to be 
\begin{align*}
	\mathbb{P}_m f= \sum_{k=1}^m (f,w_k)w_k(x).
\end{align*}
Here we take $\mathbb{Q}_m$ as the same with \eqref{eq:operatorQ}.

% To deal with the boundary condition for $\psi_\varepsilon$, we also introduce a lift operator $\mathcal{S}: L^2(\Sigma;|\beta\cdot n|d\beta d\sigma_x) \to V$ where $V$ is defined in \eqref{eq:Vspacedef}, with $\varphi = \mathcal{S}\phi$ is the weak solution to
% \begin{align*}
% 	&\beta \cdot \nabla \varphi = 0, \text{ in } \Omega\times\mathbb{S}^2, \\
% 	&\varphi = \phi, \text{ on }\Sigma_{-},\;\varphi =0, \text{ on }\Sigma_+.
% \end{align*}
% The well-posedness of the above equation can be found in \cite{di2011mathematical}. With this definition, we can see that 
% \begin{align*}
% 	(\psi_\varepsilon - \mathcal{S}(\alpha\psi_b) )|_{\Sigma_-}= (1-\alpha)L(\psi_\varepsilon - \mathcal{S}(\alpha\psi_b))|_{\Sigma_+}.
% \end{align*}
% Therefore, we take
% \begin{align}\label{eq:psi-spsib}
% 	\psi_\varepsilon^m - \mathcal{S}(\alpha\psi_b)= \sum_{k=1}^m \phi_k(t)\varphi_k(x,\beta).
% \end{align}

% We define the following operators
% \begin{align*}
% 	&\mathbb{P}_m f= \sum_{k=1}^m (f,w_k)w_k,\\
% 	&\mathbb{Q}_m \psi = \sum_{i=1}^m (\psi,\varphi_k)\varphi_k.
% \end{align*}
We consider the Galerkin approximate system
\begin{align}
	&\partial_t T_\varepsilon^m = \Delta_r(T_\varepsilon^m  - \mathcal{R}T_b) + \frac{1}{\varepsilon^2} \int_{\mathbb{S}^2} (\mathbb{P}_m\psi_\varepsilon^m- \mathbb{P}_m ((T_\varepsilon^m)^4))d\beta, \label{eq:gl31}\\
	&\partial_t \psi_\varepsilon^m + \frac{1}{\varepsilon}\beta \cdot\nabla \psi_\varepsilon^m = - \frac{1}{\varepsilon^2}(\psi_\varepsilon^m-\mathbb{Q}_m\mathbb{P}_m(T_\varepsilon^m)^4)).\label{eq:gl32}
\end{align}
We take $\widetilde{T}$ as defined in \eqref{eq:Gequation} but with boundary data 
\begin{align*}
	\widetilde T = \psi_b^{\frac{1}{4}}, \text{ on }\partial\Omega.
\end{align*}
 and by $\widetilde{\psi}=\widetilde{T}^4$, respectively. 
We take 
\begin{align*}%\label{eq:T-rtb}
	&T_\varepsilon^m - \mathcal{R}T_b = \sum_{k=1}^m d_k(t) w_k(x),\\ 
% \end{align}
% \begin{align}%\label{eq:psi-spsib}
	&\psi_\varepsilon^m - \widetilde{\psi}
	% \mathcal{S}(\alpha\psi_b)
	= \sum_{k=1}^m \phi_k(t)\varphi_k(x,\beta),
\end{align*}
into the above system to get an ODE system of $d_k(t)$ and $\phi_k(t)$ with $k=1,\ldots,m$. The existence of the ODE system is guaranteed by the Cauchy-Lipschitz theorem. \tcb{Note the solutions $T_\varepsilon^m,\psi_\varepsilon\ge 0$, which is proved in Appendix A.}
We next derive the energy estimate. We multiply \eqref{eq:gl31} by $(T_\varepsilon^m)^4 - \widetilde{T}^4$ and \eqref{eq:gl32} by $\psi_\varepsilon^m - \widetilde{\psi}$ and integrate over time and space to get (same as \eqref{eq:i123456})
	\begin{align}
		\int_{\Omega}&\left(\frac{(T_\varepsilon^m)^5}{5} - \widetilde{T}^4T_\varepsilon^m\right)(t) dx + \frac{1}{2} \iint_{\Omega \times \mathbb{S}^2}(\psi_\varepsilon^m - \widetilde{\psi})^2(t) d\beta dx \nonumber \\
		=&\int_{\Omega}\frac{(T_{\varepsilon0}^m)^5}{5} dx + \frac{1}{2} \iint_{\Omega \times \mathbb{S}^2}(\psi_{\varepsilon0}^m)^2 d\beta dx - \frac{1}{2\varepsilon} \int_0^t\iint_{\Omega\times\mathbb{S}^2} \beta \cdot \nabla (\psi_\varepsilon^m - \widetilde{\psi})^2 d\beta dx d\tau  \nonumber\\
		&+ \int_0^t \iint_{\Omega\times\mathbb{S}^2} \beta \cdot \nabla \widetilde{\psi} (\psi_\varepsilon^m - \widetilde{\psi}) d\beta dxd\tau+\int_0^t \int_{\Omega}  ((T_\varepsilon^m)^4 - \widetilde{T}^4)\Delta T_\varepsilon^m dxd\tau  \nonumber\\
		&- \frac{1}{\varepsilon^2}\int_0^t\iint_{\Omega\times\mathbb{S}^2}\left(\psi_\varepsilon^m - \mathbb{P}_m((T_\varepsilon^m)^4)\right) (\psi_\varepsilon^m - \widetilde{\psi} - \mathbb{P}_m ((T_\varepsilon^m)^4) + \widetilde{T}^4)d\beta dxd\tau \nonumber\\
		=& I_1+I_2+I_3+I_4+I_5+I_6. 
	\end{align}
	The terms can be treated as in the proof of Theorem \ref{thmexistd} except $I_5$, which can be estimated by
	\tcb{\begin{align*}
		I_5 =& \int_0^t \int_{\Omega}  ((T_\varepsilon^m)^4 - \widetilde{T}^4)\Delta T_\varepsilon^m dxd\tau \\
		=&-\frac{16}{25}\int_0^t \int_\Omega |\nabla (T_\varepsilon^m)^{\frac{5}{2}}|^2dxd\tau + \int_0^t\int_{\partial\Omega}((T_\varepsilon^m)^4-\widetilde{T}^4) n\cdot \nabla T_\varepsilon^m  d\sigma_x d\tau \\
		& - \int_0^T\int_{\Omega} T_\varepsilon^m \Delta \widetilde{T} dxd\tau + \int_0^t\int_{\partial\Omega} T_b n\cdot \nabla \widetilde{T}^4 d\sigma_x d\tau \\
		=& -\frac{16}{25}\int_0^t \int_\Omega |\nabla (T_\varepsilon^m)^{\frac{5}{2}}|^2dxd\tau+\int_0^t \int_{\partial\Omega} ((T_\varepsilon^m)^4-T_b^4) \frac{1}{\varepsilon^r}(-T_\varepsilon^m + T_b) d\sigma_x d\tau\\
		&+ \int_{0}^t \int_\Omega \frac{(T_\varepsilon^m)^5}{5} dxd\tau + \int_0^t\int_{\Omega} \frac{(\Delta \widetilde{T}^4)^{\frac{5}{4}}}{\frac{5}{4}}  dxd\tau \nonumber\\
		&+ \frac{4}{5} \int_{\Omega} {\widetilde{T}^5} (t)dx  + \frac{64}{25}\int_0^t \int_{\Omega} |\nabla \widetilde{T}^{\frac{5}{2}}|^2 dxd\tau \\
		=&-\frac{16}{25}\int_0^t \int_\Omega |\nabla (T_\varepsilon^m)^{\frac{5}{2}}|^2dxd\tau + \int_{0}^t \int_\Omega \frac{(T_\varepsilon^m)^5}{5} dxd\tau \nonumber\\
		&+ \int_0^t\int_{\Omega} \frac{(\Delta \widetilde{T}^4)^{\frac{5}{4}}}{\frac{5}{4}}  dxd\tau + \frac{4}{5} \int_{\Omega} {\widetilde{T}^5} (t)dx  + \frac{64}{25}\int_0^t \int_{\Omega} |\nabla \widetilde{T}^{\frac{5}{2}}|^2 dxd\tau \\
		&- \frac{1}{\varepsilon^r} \int_0^t \int_{\partial\Omega} ((T_\varepsilon^m)^2+T_b^2)(T_\varepsilon^m + T_b) (T_\varepsilon^m - T_b)^2 d\sigma_x d\tau.
	\end{align*}
	The last term appears additional to the estimate \eqref{eq:i5est} of $I_5$ for the Dirichlet case. Using the positivity of $T_\varepsilon^m$ $T_{b}$}, we get 
	\begin{align*}
		&((T_\varepsilon^m)^2+T_b^2)(T_\varepsilon^m + T_b) (T_\varepsilon^m - T_b)^2 - {(T_\varepsilon^m-T_b)^5} \\
		&\quad= {(T_\varepsilon^m - T_b)^2}2T_b(2(T_\varepsilon^m)^2-T_\varepsilon^m T_b + T_b^2) \ge 0,
	\end{align*}
	and 
	\begin{align*}
		&((T_\varepsilon^m)^2+T_b^2)(T_\varepsilon^m + T_b) (T_\varepsilon^m - T_b)^2 + {(T_\varepsilon^m-T_b)^5} \\
		&\quad= {(T_\varepsilon^m - T_b)^2}2\tcb{T_\varepsilon^m}((T_\varepsilon^m)^2-T_\varepsilon^m T_b + 2T_b^2) \ge 0,
	\end{align*}
	so that 
	\begin{align}\label{EstiL5}
		((T_\varepsilon^m)^2+T_b^2)(T_\varepsilon^m + T_b) (T_\varepsilon^m - T_b)^2 \ge  |T_\varepsilon^m-T_b|^5. 	
	\end{align}
Consequently, we obtain 
	\begin{align*}
		I_5 \le& -\frac{16}{25}\int_0^t \int_\Omega |\nabla (T_\varepsilon^m)^{\frac{5}{2}}|^2dxd\tau -\frac{1}{\varepsilon^r} \int_0^t \|\gamma^1 T_\varepsilon - T_b\|_{L^5(\partial\Omega)}^5d\tau.
	\end{align*}
	The energy inequality \eqref{eq:gl2est} then becomes
	\begin{align}\label{eq:gl3est}
		&\|T_{\varepsilon}^m(t)\|_{L^5(\Omega))}^5 + \|\psi_\varepsilon^m(t)\|_{L^2(\Omega\times\mathbb{S}^2)}^2 + \int_0^t \|\nabla (T_\varepsilon^m)^{\frac{5}{2}}\|_{L^2(\Omega)}^2 d\tau \nonumber\\
		&\qquad + \frac{1}{\varepsilon^2} \int_0^t \|\psi_\varepsilon^m - \mathbb{P}_m((T_\varepsilon^m)^4)\|_{L^2(\Omega\times\mathbb{S}^2)}^2 d\tau+\frac{1}{\varepsilon^r} \int_0^t \|\gamma^1 T_\varepsilon - T_b\|_{L^5(\partial\Omega)}^5d\tau \nonumber \\
		&\qquad+\frac{2\alpha - \alpha^2}{2\varepsilon} \int_0^t \|\gamma^2\psi_\varepsilon^m - \psi_b\|_{L^2(\Sigma_+;|n\cdot\beta| d\beta d\sigma_x)}^2 d\tau \nonumber\\
		&\quad \le C e^{Ct}\left(\|T_{\varepsilon0}\|_{L^5}^5 + \|\psi_{\varepsilon0}\|_{L^2(\Omega\times\mathbb{S}^2)}^2\right).
	\end{align}
	We can pass to the limit $m\to \infty$ and use trace theorem like in the proof of Theorem \ref{thmexistd} to get
	\begin{align*}
		&\gamma^1 T_\varepsilon^m \to \gamma^1 T_\varepsilon, \text{ strongly in }L^5_{\loc}([0,\infty);L^5(\partial\Omega))\\
		&\gamma^2\psi_\varepsilon^m \rightharpoonup \gamma^2 \psi_\varepsilon, \text{ weakly in }L^2_{\loc}([0,\infty);L^2(\Sigma;|n\cdot\beta| d\beta d\sigma_x d\tau)).
	\end{align*}
	With this we can apply test functions on the Galerkin system \eqref{eq:gl31}-\eqref{eq:gl32} and pass to the limit $m\to \infty$ to show \eqref{eq:weakr1}-\eqref{eq:weakr2} holds.

	To show the energy inequality \eqref{eq:energyrobin}, we can pass to the limit $m\to \infty$ in the above energy estimate and follow the proof of Theorem \ref{thm:existencet} and Theorem \ref{thmexistd}, except here we additionally need
	\begin{align*}
		\int_0^t \|\gamma^1 T_\varepsilon - T_b\|_{L^5(\partial\Omega)}^5d\tau \le \liminf_{m\to \infty}\int_0^t \|\gamma^1 T_\varepsilon^m - T_b\|_{L^5(\partial\Omega)}^5 d\tau
	\end{align*}
	which comes from the convergence of $\gamma^1 T_\varepsilon$.
\end{proof}

\tcb{\subsection{Uniform boundness of the solutions.}
We finish this section by showing the non-negative weak solutions are uniformly bounded. We have the following lemma.
\begin{lemma}\label{lm:uniformbd}
	Assume the intial and boundary data satisfy $0\le T_{\varepsilon0}\le \gamma,0\le \psi_{\varepsilon0}\le \gamma^4$ and $0\le T_b\le \gamma,0\le \psi_b \le \gamma^4$ for some constant $\gamma>0$. Then the weak solutions $(T_\varepsilon,\psi_\varepsilon)$ to system \eqref{eq:Teps}-\eqref{eq:psieps} is uniformly bounded, i.e. $0\le T_\varepsilon\le \gamma^4,0\le \psi_\varepsilon \le \gamma^4.$
\end{lemma}
The proof of the above lemma is given in Appendix A.}

\section{Passage to the limit with weak compactness method}\label{section3}

Now we state the main theorem of this paper.
\begin{theorem}\label{LimitProof}
Consider a family of nonnegative weak solution $\left(T_\varepsilon,\psi_\varepsilon \right)$ of \eqref{eq:Teps}-\eqref{eq:psieps} with initial conditions \eqref{eq:ic1}-\eqref{eq:ic2} and boundary condition \eqref{bpsi} for $\psi_\varepsilon$ and boundary conditions \eqref{b1}, or \eqref{b3}  or \eqref{b2} for $T_\varepsilon$, defined in Definition \ref{df1},  Definition \ref{def3} and  Definition  \ref{def2}, respectively. Assume the nonnegative initial data satisfy 
\[
\|T_{\varepsilon0}-\overline{T}_{0}\|_{L^{5}(\Omega)} \rightarrow 0\quad \text{and}\quad \|\psi_{\varepsilon0}-\overline{T}_0^4\|_{L^{2}(\Omega\times\mathbb{S}^{2})} \rightarrow 0,\; \quad\text{as} \; \varepsilon \rightarrow 0,
\]
where $\overline{T}_0 \in L^8(\Omega)$. We also suppose that the well prepared data condition \eqref{eq:wellbc} is satisfied and the boundary data satisfy $T_b,\psi_b\ge 0$. Then, when $\varepsilon \to 0$, we can extract a subsequence of $\left(T_\varepsilon,\psi_\varepsilon \right)$ such that for $t>0$,
\begin{align}
	&T_\varepsilon \to  \overline{T} \quad \text{almost everywhere}, \label{eq:tstrong}\\
	&\psi_\varepsilon \to   \overline{T}^{4} \quad \text{strongly in } L^2([0,t]\times\Omega\times\mathbb{S}^{2}).\label{eq:psistrong}
\end{align}
Moreover, $\overline{T}=\overline{T}(t,x)$ is the weak solution of the limit equation
\begin{align}\label{Limitsystem}
	\partial_{t}\left(\overline{T}+4\pi \overline{T}^{4}\right)=\Delta\left(\overline{T}+\frac{4\pi}{3}\overline{T}^{4}\right),
\end{align}
with initial condition
\begin{align}\label{eq:inicond}
 \overline{T}(0,x)=\overline{T}_{0}(x),\,& x \in \Omega,
\end{align}
and boundary condition % for Dirichlet and Robin boundary conditions
\begin{align}\label{eq:Boundarycond}
	\overline{T}(t,x)= T_{b}(t,x), \, t>0 \text{ and } x\in\partial \Omega.
\end{align}
\end{theorem}
\begin{proof}
The proof can be divided into two steps. First we show the convergence of the solutions of the system \eqref{eq:Teps}-\eqref{eq:psieps}, i.e. \eqref{eq:tstrong} and \eqref{eq:psistrong} hold. Then we show that the limit $\overline{T}$ satisfies the equation \eqref{Limitsystem} as well as the initial and boundary conditions.

\textbf{Convergence of the solutions for system \eqref{eq:Teps}-\eqref{eq:psieps}.}
From Theorem \ref{thm:existencet}, or Theorem \ref{thmexistd} or Theorem \ref{thmexistr}, we get, under any of the three type boundary conditions considered in the above theorems,
% \eqref{b1}-\eqref{b3}, for any finite $t>0$, the following uniform estimate holds
\begin{align}\label{eq:estC}
	&\|T_\varepsilon(t)\|_{L^5(\Omega)}^5  +\|\psi_\varepsilon(t)\|_{L^2(\Omega\times\mathbb{S}^2)}^2 + \int_0^t \|\nabla T_\varepsilon^{\frac{5}{2}}\|_{L^2}^2 d\tau \nonumber\\
	&\quad+ \int_0^t \left\|\frac{1}{\varepsilon}(\psi_\varepsilon-T_\varepsilon^4)\right\|_{L^2(\Omega\times\mathbb{S}^2)}^2d\tau \le C.
\end{align}
Here $C$ does not depend on $\varepsilon$.

Therefore, it follows that up to a subsequence,
\begin{align}
	&\psi_\varepsilon \rightharpoonup \overline{\psi}, \text{ weakly in }L^2([0,t];L^2(\Omega\times\mathbb{S}^2)), \label{eq:wk1}\\
	&T_\varepsilon \rightharpoonup \overline{T}, \text{ weakly in }L^5([0,t];L^5(\Omega)), \label{eq:wk2}\\
	&T_\varepsilon^{\frac{5}{2}} \rightharpoonup \overline{T_\varepsilon^{\frac{5}{2}}}, \text{ weakly in } L^2([0,t];H^1(\Omega)),\label{eq:wk3}\\
	&\psi_\varepsilon - T_\varepsilon^4 \to 0, \text{ strongly in }L^2([0,t];L^2(\Omega\times\mathbb{S}^2)),\label{eq:wk4} \\
	&\frac{1}{\varepsilon}(\psi_\varepsilon-T_\varepsilon) \rightharpoonup A, \text{ weakly in }L^2([0,t];L^2(\Omega\times\mathbb{S}^2)). \label{eq:wk5}
\end{align}
Here and below, $\overline{f_\varepsilon}$ denotes the weak limit of $\{f_\varepsilon\}_{\varepsilon>0}$ while $\varepsilon\to 0$.

Since
\begin{align*}
	4\pi \|T_\varepsilon^4\|_{L^2(\Omega)}=\|T_\varepsilon^4\|_{L^2(\Omega\times\mathbb{S}^2)} \le \|\psi_\varepsilon-T_\varepsilon^4\|_{L^2(\Omega\times\mathbb{S}^2)} + \|\psi_\varepsilon\|_{L^2(\Omega\times\mathbb{S}^2)},
\end{align*}
we have
\begin{align*}
	T_\varepsilon \text{ in uniformly bounded in }L^8([0,t];L^8(\Omega)).
\end{align*}
It follows that
\begin{align}\label{eq:wklp}
	T_\varepsilon^p \rightharpoonup \overline{T_\varepsilon^p}, \text{ weakly in }L^{q_1}([0,t];L^{q_2}(\Omega)),
\end{align}
for any $1\le p\le 8$ and $q_1 \le \frac{8}{p}, q_2 \le \frac{8}{p}$.

Taking the integral of \eqref{eq:psieps} over $\beta\in\mathbb{S}^2$ and adding \eqref{eq:Teps}, we get
	\begin{align}
	\partial_t \left(T_\varepsilon +\langle \psi_\varepsilon \rangle \right) + \frac{1}{\varepsilon} \nabla \cdot \langle \psi_\varepsilon \beta\rangle = \Delta T_\varepsilon.
	\end{align}
	Since for all $t>0$,
	\begin{align*}
	\|T_\varepsilon\|_{L^1([0,t];L^1(\Omega))} \le C \|T_\varepsilon\|_{L^5([0,t];L^5(\Omega))}
	% \int_0^t \|T_\varepsilon\|_{L^1(\Omega)}d\tau \le C \int_0^t \|T_\varepsilon\|_{L^8(\Omega)} d\tau \le C \left(\int_0^t \|T_\varepsilon\|_{L^8}^8 d\tau\right)^{\frac{1}{8}} \left(\int_0^td\tau\right)^{\frac{7}{8}} \le Ct^{\frac{8}{7}}
	\end{align*}
	is uniform bounded in $\varepsilon$, we get that $ \Delta T_\varepsilon$ is bounded in  ${L^{1}([0,t]; W^{-2,1}(\Omega))}$. Moreover, using \eqref{eq:estC}, we have  
	\begin{align*}
	\int_0^t& \int_\Omega \int_{\mathbb{S}^{2}} \frac{1}{\varepsilon} \psi_\varepsilon \beta d\beta dx d\tau \\
	=&  \int_0^t \int_\Omega \int_{\mathbb{S}^{2}} \frac{1}{\varepsilon} (\psi_\varepsilon -T_\varepsilon^4)\beta d\beta dxd\tau\\
	\le& \left(\frac{1}{\varepsilon^2}\int_0^t \int_\Omega \int_{\mathbb{S}^{2}} (\psi_\varepsilon-T_\varepsilon^4)^2 d\beta dxd\tau\right)^{\frac{1}{2}} \left(\int_0^t \int_\Omega \int_{\mathbb{S}^{2}} |\beta|^2 d\beta dxd\tau\right)^{\frac{1}{2}} \\
	\le & C \int_0^t \left\|\frac{1}{\varepsilon}(\psi_\varepsilon-T_\varepsilon^4)^2 d\beta dxd\tau\right\|_{L^2(\Omega\times\mathbb{S}^2)} \le C.
	\end{align*}
	Therefore, $\frac{1}{\varepsilon} \nabla \cdot \langle \psi_\varepsilon \beta \rangle $ is bounded in  ${L^{1}([0,t]; W^{-1,1}(\Omega))}$. Consequently, we have 
	\begin{align}\label{eq:partialtl1}
	\partial_t (T_\varepsilon + \langle \psi_\varepsilon \rangle) \in L^{1}([0,t]; W^{-2,1}(\Omega)). 
	\end{align}
	In addition, from \eqref{eq:estC}, we can get
	\begin{align}\label{eq:tpsimoml2}
		T_\varepsilon +\langle \psi_\varepsilon \rangle  \in L^2([0,t]\times\Omega).	
	\end{align}
	From \eqref{eq:wk1} and \eqref{eq:wk2}, we deduce
	% and that 
	% \begin{align}
	% 	&T_\varepsilon \rightharpoonup \overline{T},\quad \text{in } L^2([0,t]\times\Omega)\cap L^8([0,t]\times\Omega), \label{eq:thgm1.0}\\
	% 	&\psi_\varepsilon\rightharpoonup	\overline{\psi} ,\quad \text{in } L^2([0,t]\times\Omega\times\mathbb{S}^2),	\label{eq:psipsibar}\\
	% 	&\int_{\mathbb{S}^2} \psi_\varepsilon d\beta \rightharpoonup \int_{\mathbb{S}^2} \overline{\psi} d\beta, \quad \text{in } L^2([0,t]\times\Omega). \label{eq:psimomenhgm1.0}
	% \end{align}
	% Thus, we obtain
	\begin{align}\label{eq:tpluspsiconver}
		T_\varepsilon +  \langle \psi_\varepsilon \rangle \rightharpoonup \overline{T}+\langle \overline{\psi} \rangle, \quad \text{ weakly in } L^2([0,t]\times\Omega).
	\end{align}
	On the other hand, from \eqref{eq:wk3}, we have
	% \begin{align*}
	% 	\int_0^t \|T_\varepsilon^{\frac{5}{2}}\|_{L^2(\Omega)}^2 d\tau =& \int_0^t \|T_\varepsilon\|_{L^5}^5 d\tau \\
	% 	\le& \int_0^t \|T_\varepsilon\|_{L^{15}}^5 d\tau \le \int_0^t \|T_\varepsilon^{\frac{5}{2}}\|_{L^{6}}^2 d\tau \le \int_0^t \|\nabla T_\varepsilon^{\frac{5}{2}}\|_{L^2}^2 d\tau \le C,
	% \end{align*}
	% so 
	\begin{align}\label{eq:wealconvT52}
		T_{\varepsilon}^{\frac{5}{2}} \rightharpoonup	\overline{T_\varepsilon	^{\frac{5}{2}}},\quad \text{ weakly in } L^2([0,t]\times\Omega).	
	\end{align}
	Then Lemma \ref{Lemmalions1996mathematical}, with its assumptions verified by \eqref{eq:partialtl1}- \eqref{eq:wealconvT52}, implies that 
	\begin{align}\label{eq:weakprod}
		\left(T_\varepsilon	 + \langle \psi_\varepsilon	\rangle\right) T_\varepsilon^{\frac{5}{2}} 	\rightharpoonup	 \left(\overline{T} + \langle \overline{\psi} \rangle \right) \overline{T_\varepsilon	^{\frac{5}{2}}}, \quad \text{in the sense of distributions}.
	\end{align}
	Moreover, due to \eqref{eq:wk4},
	% \[\|\frac{1}{\varepsilon}(\psi_\varepsilon - T_\varepsilon^4) \|_{L^2([0,t]\times\Omega\times\mathbb{S}^2)} \le C,\]
	% % it follows that
	% % \[\frac{1}{\varepsilon} (\psi_\varepsilon - T_\varepsilon^4) \rightharpoonup A, \quad \text{in } L^2([0,t]\times\Omega\times\mathbb{S}^2).\]
	% we have						
	% \begin{align}\label{eq:psiminusstrong}
	% 	\psi_\varepsilon	- T_\varepsilon^4 \to 0, \quad \text{strongly in }L^2([0,t]\times\Omega\times\mathbb{S}^2).
	% \end{align}
	we have $\overline{\psi} - \overline{T_\varepsilon	^4}=0$. Taking this into \eqref{eq:weakprod}, we conclude that 
	\begin{align}\label{eq:13523}
	&\left(T_{\varepsilon}+\langle \psi_\varepsilon \rangle \right)T_\varepsilon^{\frac{5}{2}} \rightharpoonup (\overline{T} + 4\pi\overline{T_\varepsilon^4}) \overline{T_\varepsilon^{\frac{5}{2}}}, \quad \text{in the sense of distributions.}
	\end{align}
	On the other hand, using the weak convergence \eqref{eq:wklp} with $p=\frac{7}{2},\frac{13}{2}$ and the strong convergence \eqref{eq:wk4}, we get
	% \begin{align*}
	% 	&\|T_\varepsilon^{\frac{7}{2}}\|_{L^2([0,t]\times\Omega)}^2 = \|T_\varepsilon\|_{L^7([0,t]\times\Omega)}^7 \le \|T_\varepsilon\|_{L^8([0,t]\times\Omega)}^8 \le C,	 \\
	% 	&\|T_\varepsilon^{\frac{13}{2}}\|_{L^{\frac{16}{13}}([0,t]\times\Omega)}^2= \|T_\varepsilon\|_{L^8([0,t]\times\Omega)}^{13}  \le C,
	% \end{align*}
	% we have
	\begin{align*}
		&T_\varepsilon^{\frac{7}{2}} \rightharpoonup \overline{T_\varepsilon^{\frac{7}{2}}}, \quad \text{weakly in }L^2([0,t]\times\Omega), \\
		& \langle \psi_\varepsilon\rangle T_\varepsilon^{\frac{5}{2}} = \langle \psi_\varepsilon - T_\varepsilon^4 \rangle T_\varepsilon^{\frac{5}{2}} + 4\pi T_\varepsilon^{\frac{13}{2}}  \rightharpoonup 4\pi \overline{T_\varepsilon^{\frac{13}{2}}}, \quad \text{weakly in }L^{\frac{16}{13}}([0,t]\times\Omega).
	\end{align*}
	% From \eqref{eq:wealconvT52} and \eqref{eq:psiminusstrong}, we have 
	% \begin{align*}
	% 	\int_{\mathbb{S}^2} \psi_\varepsilon d\beta \cdot T_\varepsilon^{\frac{5}{2}} = \int_{\mathbb{S}^2}(\psi_\varepsilon - T_\varepsilon	^4) d\beta \cdot T_\varepsilon^{\frac{5}{2}} +4\pi T_\varepsilon^{\frac{13}{2}} \rightharpoonup	0 + \overline{T_\varepsilon^{\frac{13}{2}}}
	% \end{align*}
	% in the sense of distributions. 
	Therefore,
	\[\left(T_{\varepsilon}+\int_{\mathbb{S}^{2}} \psi_\varepsilon \right)T_\varepsilon^{\frac{5}{2}} \rightharpoonup \overline{T_\varepsilon^{\frac{7}{2}}}+\overline{T_\varepsilon^{\frac{13}{2}}}, \text{ in the sense of distributions}.\]
	% \quad \text{weakly in }L^2([0,t]\times\Omega).\]
	Comparing \eqref{eq:13523} and using the uniqueness of weak limits, we arrive at
	\begin{align}\label{eq:weaklimiteq}
	\overline{T_\varepsilon^{\frac{7}{2}}}+4\pi\overline{T_\varepsilon^{\frac{13}{2}}}=\overline{T}\overline{T_\varepsilon^{\frac{5}{2}}}+4\pi\overline{T_\varepsilon^{4}}\overline{T_\varepsilon^{\frac{5}{2}}}. 		
	\end{align}
 
Next we use the family of Young measures $\{\nu_{x}\}_{x\in \Omega}$ (see \cite[Theorems 2.2,2.3]{chen2000compactness} and \cite{tartar1979compensated,balder1995lectures,ball1989version}) associated with the $\{{T_{\varepsilon_{n}},}_{n\in \mathbb{N}} \}$ to prove that \eqref{eq:weaklimiteq} implies the strong convergence of $\overline{T_\varepsilon}$ to $\overline{T}$. Indeed, we have
\begin{align}
	(T_{\varepsilon_{n}})^p \rightharpoonup \int_{\mathbb{R}} \lambda^p d\nu_x(\lambda).
\end{align}
for any $p \ge 1$.
 % Indeed, 
 % 	\begin{align*}
	% &T^{13/2}_{\varepsilon}   \rightharpoonup \int_{0}^{1}\lambda^{13/2}d \nu_{x}(\lambda)=\int_{0}^{1}\mu^{13/2}d \nu_{x}(\mu)=\overline{T^{13/2}_{\varepsilon}} \\
	% &T^{7/2}_{\varepsilon}   \rightharpoonup \int_{0}^{1}\lambda^{7/2}d \nu_{x}(\lambda)=\int_{0}^{1}\mu^{7/2}d \nu_{x}(\mu)=\overline{T^{7/2}_{\varepsilon}} \\
	% &T^{4}_{\varepsilon}   \rightharpoonup \int_{0}^{1}\lambda^{4}d \nu_{x}(\lambda)=\int_{0}^{1}\mu^{4}d \nu_{x}(\mu)=\overline{T^{4}_{\varepsilon}} \\
	% &T^{5/2}_{\varepsilon}   \rightharpoonup \int_{0}^{1}\lambda^{5/2}d \nu_{x}(\lambda)=\int_{0}^{1}\mu^{5/2}d \nu_{x}(\mu)=\overline{T^{5/2}_{\varepsilon}} \\
	% &T_{\varepsilon}   \rightharpoonup \int_{0}^{1}\lambda d \nu_{x}(\lambda)=\int_{0}^{1}\mu d\nu_{x}(\mu)=\overline{T_{\varepsilon}}.
	% \end{align*}
Hence,
\begin{align*}
	&\overline{T^{\frac{7}{2}}_{\varepsilon}}+\overline{T^{\frac{13}{2}}_{\varepsilon}}=\int_{\mathbb{R}} \lambda^{\frac{7}{2}} d\nu_x(\lambda) + \int_{\mathbb{R}} \lambda^{\frac{13}{2}} d\nu_x(\lambda),\\
	 % \mu^{4}\left(\mu^{5/2}-\lambda^{5/2}\right) d\nu_{\lambda}(\mu)d\nu_{x}(\mu)=\int_{0}^{1}\lambda^{4}\left(\lambda^{5/2}-\mu^{5/2}\right) d\nu_{\lambda}(\mu)d\nu_{x}(\mu),
	 & \overline{T}\overline{T_\varepsilon^{\frac{5}{2}}}+\overline{T_\varepsilon^{4}}\overline{T_\varepsilon^{\frac{5}{2}}} = \int_{\mathbb{R}}\int_{\mathbb{R}} \mu\lambda^{\frac{5}{2}} d\nu_x(\lambda) d\nu_x(\mu) + \int_{\mathbb{R}}\int_{\mathbb{R}} \lambda^4\mu^{\frac{5}{2}} d\nu_x(\lambda) d\nu_x(\mu) .
\end{align*}
From \eqref{eq:weaklimiteq}, the above two equations equal, that is
\begin{align*}
	 \int_{\mathbb{R}}\int_{\mathbb{R}} \left(\lambda^{\frac{5}{2}}(\lambda-\mu) + \lambda^4(\lambda^{\frac{5}{2}}-\mu^{\frac{5}{2}})\right)d\nu_x(\lambda)d\nu_x(\mu) =0.
\end{align*}
Using the symmetric property of the above formula leads to
\begin{align*}
	0 \le\int_{\mathbb{R}}\int_{\mathbb{R}} \left((\lambda^{\frac{5}{2}}-\mu^{\frac{5}{2}})(\lambda-\mu) + (\lambda^4-\mu^4)(\lambda^{\frac{5}{2}}-\mu^{\frac{5}{2}})\right)d\nu_x(\lambda)d\nu_x(\mu) =0.
\end{align*}
Since the function inside the integral is strictly positive unless $\lambda=\mu$, we can conclude that $\nu_x(\lambda)$ reduces almost all points of $x$ to a family of Dirac masses concentrated at $\nu_x=\delta_{\overline{T}(x)}$. \tcb{With this and the uniformly boundness of $T_\varepsilon,\psi_\varepsilon$ according to Lemma \ref{lm:uniformbd}}, we can apply \cite[Theorem 2.3]{chen2000compactness}, and conclude that 
% Since 
% 	$$
% 	\overline{T^{13/2}_{\varepsilon}}-\overline{T^{4}_{\varepsilon}}\overline{T^{5/2}_{\varepsilon}}=\int_{0}^{1}\mu^{4}\left(\mu^{5/2}-\lambda^{5/2}\right) d\nu_{\lambda}(\mu)d\nu_{x}(\mu)=\int_{0}^{1}\lambda^{4}\left(\lambda^{5/2}-\mu^{5/2}\right) d\nu_{\lambda}(\mu)d\nu_{x}(\mu),
% 	$$
% 	$$
% 	2\left(\overline{T^{13/2}_{\varepsilon}}-\overline{T^{4}_{\varepsilon}}\overline{T^{5/2}_{\varepsilon}}\right)=\int_{0}^{1}\left(\lambda^{4}-\mu^{4}\right)\left(\lambda^{5/2}-\mu^{5/2}\right) d\nu_{\lambda}(\mu)d\nu_{x}(\mu)\geqslant 0.
% 	$$
% 	Therefore, $d\nu_x(\lambda)=\delta_{\{\lambda=\overline{T}\}}$ and so
\begin{align}\label{eq:tc}
	T_\varepsilon \to \overline{T}, \text{ almost everywhere and strongly in } L^p([0,t]\times\Omega), 1<p\le \infty.
\end{align}
% \tcb{Xiaokai: In \cite[Theorem 2.3]{chen2000compactness}, we need the uniform boundness of $T_\varepsilon.$ It seems that we need to show this in the second section.}
From this, we get $T_\varepsilon^4 \to \overline{T}^4$ almost everywhere. This combines \eqref{eq:wk4} implies that 
\begin{align}\label{eq:psic}
	\psi_\varepsilon \to \overline{T}^4, \quad \text{strongly in } L^2([0,t]\times\Omega \times\mathbb{S}^2).
\end{align}
	Therefore, we have proved \eqref{eq:tstrong} and \eqref{eq:psistrong}.
	\bigskip

	\textbf{The limiting system.} % satisfies the equation \ref{Limitsystem}.}
	To show the limit function $\overline{T}$ satisfy equation \eqref{Limitsystem}, we define
	\[\rho_\varepsilon = \int_{\mathbb{S}^2} \psi_\varepsilon d\beta, \; j_\varepsilon  = \frac{1}{\varepsilon} \int_{\mathbb{S}^2} \psi_\varepsilon\beta d\beta \cdot \]
	We have 
	\begin{align}\label{eq:rho}
	\partial_t \rho_\varepsilon + \nabla \cdot j_\varepsilon = -\frac{1}{\varepsilon^2} \int_{\mathbb{S}^{2}} (\psi_\varepsilon-T_\varepsilon^4) d\beta \cdot
	\end{align}
	 Comparing the equation \eqref{eq:Teps} and \eqref{eq:rho}, we get 
	\begin{align*}%\label{eq:Teqrhoj}
	\partial_t  T_\varepsilon  - \Delta T_\varepsilon  = -(\partial_t \rho_\varepsilon +\nabla \cdot j_\varepsilon).
	\end{align*}
	Using \eqref{eq:tc}-\eqref{eq:psic} and $\rho_\varepsilon \to 4\pi \overline{T}^4$ we can pass to the limit in the above equation to get
	\begin{align}\label{eq:tbarjbar}
		\partial_t \overline{T} - \Delta \overline{T} = -4\pi\partial_t \overline{T}^4 + \nabla \cdot \overline{j_\varepsilon},
	\end{align}
	in the sense of distributions. Here we use $\overline{j_\varepsilon} $ to denote the weak limit of $j_\varepsilon$.

	Next we find the weak limit of $j_\varepsilon$.
	Using equation \eqref{eq:psieps}, we can get 
	\begin{align*}
	j_\varepsilon  = \frac{1}{\varepsilon} \int_{\mathbb{S}^2} \psi_\varepsilon\beta d\beta = \frac{1}{\varepsilon} \int_{\mathbb{S}^2} (\psi_\varepsilon-T_\varepsilon^4)\beta d\beta = -\varepsilon \partial_t\int_{\mathbb{S}^2} \psi_\varepsilon \beta d\beta -\nabla\cdot\int_{\mathbb{S}^2} (\psi_\varepsilon \beta\otimes \beta) d\beta.
	\end{align*}
	From the convergence of $\psi_\varepsilon$, we can see the 
	\begin{align*}
		\varepsilon \partial_t \int_{\mathbb{S}^2} \psi_\varepsilon \beta d\beta \rightharpoonup 0, \text{ weakly in }L^2([0,t];L^2(\Omega))
	\end{align*}
	as $\varepsilon \to 0$, and
	\begin{align*}
		\int_{\mathbb{S}^2} \psi_\varepsilon \beta \otimes\beta d\beta \to \overline{T}^4 \int_{\mathbb{S}^2} \beta \otimes\beta d\beta = \frac{4\pi}{3} \overline{T}^4 I 	
	\end{align*} 
	where $I$ is the identity matrix in $\mathbb{R}^3$. Therefore, we get
	\begin{align*}
		\nabla \cdot j_\varepsilon \rightharpoonup -\Delta (4\pi T_\varepsilon^4), \text{ in the sense of distributions.}
	\end{align*}
	It follows from this and \eqref{eq:tbarjbar} that 
	\begin{align*}
		\partial_t \overline{T} -\Delta \overline{T} = -4\pi \partial_t \overline{T}^4 +\frac{4\pi}{3} \Delta \overline{T}^4
	\end{align*}
	holds in the sense of distributions, i.e. equation \eqref{Limitsystem} holds.

	\bigskip
	\textbf{The initial condition.} Next we show the initial condition \eqref{eq:inicond} of the limit system \eqref{Limitsystem} holds in a weak sense. We consider
	% To show the initial condition, we consider
	\begin{align*}
		\int_0^t \int_{\Omega} \left(\frac{1}{\varepsilon} \int_{\mathbb{S}^2} \psi_\varepsilon \beta d\beta\right)^2 dxd\tau =& \int_0^t \int_{\Omega} \left(\frac{1}{\varepsilon} \int_{\mathbb{S}^2} (\psi_\varepsilon- T_\varepsilon^4) \beta d\beta\right)^2 dxd\tau \\
		\le&\frac{1}{\varepsilon^2}\int_0^t \int_{\Omega}  \int_{\mathbb{S}^2} \left(\psi_\varepsilon-T_\varepsilon^4 \right)^2 d\beta  dxd\tau 
	\end{align*}
	which is uniformly bounded due to \eqref{eq:estC}. This leads to
	\begin{align*}
		\partial_t (T_\varepsilon + \langle \psi_\varepsilon \rangle) \in L^2([0,t];H^{-2}(\Omega)),
	\end{align*}
	which combing with \eqref{eq:tc}-\eqref{eq:psic} implies that 
	\begin{align*}
		\overline{T} + 4\pi \overline{T}^4 \in C_w([0,t];L^2(\Omega)).
	\end{align*}
	Consequently, we get
	\[
	\overline{T}(t=0)+4\pi \overline{T}^4(t=0)= \lim_{\varepsilon\to0} \left(T_{\varepsilon0}+<\psi_{\varepsilon0}>\right)=\overline{T}_{0}+4\pi \overline{T}_{0}^4.
	\]
	Hence, from \eqref{eq:wellinitial} it follows 
	\begin{align*}
		\overline{T}(t=0)  = \overline{T}_0  = \lim_{\varepsilon\to0} T_{\varepsilon0},
	\end{align*}
	in a weak sense.
	% It follows that
	% \begin{align*}
	% 	\|\overline{T}_0 + 4\pi \overline{T}^4_0\|_{L^2(\Omega)} \le \liminf_{t\to 0^+ }\|\overline{T}(t) + 4\pi \overline{T}^4(t)\|_{L^2(\Omega)} .
	% \end{align*}
	% {\color{red} On the other hand, for the case of torus, we can pass to the limit $\varepsilon \to 0$ in \eqref{eq:energythm1} for all $t\ge 0$,
	% \begin{align}
	% 	&\frac{1}{5}\|\overline{T}(t)\|_{L^5(\mathbb{T}^3)}^5 + \frac{4\pi}{2}\|\overline{T}^4\|_{L^2(\mathbb{T}^3)}^2 + \frac{16}{25}\int_0^t \|\nabla \overline{T}^{\frac{5}{2}}\|_{L^2}^2 d\tau \nonumber \\
	% 	&\quad\quad+ \frac{4\pi}{3} \int_0^t \|\nabla \overline{T}^4\|_{L^2(\mathbb{T}^3 \times \mathbb{S}^2)}^2 d\tau \nonumber \\
	% 	&\quad \le \frac{1}{5} \|\overline{T}_0\|_{L^5(\mathbb{T}^3)}^5 + \frac{1}{2} \|\overline{T}_0^4\|_{L^2(\mathbb{T}^3\times \mathbb{S}^2)}^2. 
	% \end{align}
	%  From the above energy inequality we obtain
	% \begin{align*}
	% 	\limsup_{t\to 0^+} (\|\overline{T}(t)\|_{L^5(\mathbb{T}^3)}^5 + 10\pi \|\overline{T}^4(t)\|_{L^2(\mathbb{T}^3)}^2) \le \|\overline{T}_0\|_{L^5(\mathbb{T}^3)}^5 + 10\pi \|\overline{T}^4_0\|_{L^2(\mathbb{T}^3)}^2).
	% \end{align*}
	% It follows that
	% \begin{align*}
	% 	\|\overline{T}-\overline{T}_0\|_{L^8(\mathbb{T}^3)} \to 0, \text{ as }t \to 0^+.
	% \end{align*}}

\bigskip
Now let us deal with the boundary conditions.

\textbf{Dirichlet boundary condition.}
For the boundary condition of $\overline{T}$, due to \eqref{eq:wk3} and \eqref{eq:tc}, there exists a subsequence $\{T_{\varepsilon_k}\}_{k=1}^\infty$ satisfying 
	\begin{align*}
		T_{\varepsilon_k}^{\frac{5}{2}} \to \overline{T}^{\frac{5}{2}}, \text{ strongly in } L^2([0,t];H^{1-\delta}(\Omega))
	\end{align*}
	for $0<\delta<\frac12$ small.
	Then we can use the continuity of the trace operator to get
	\begin{align}\label{eq:traceconv}
		\gamma^1 T_{\varepsilon_k}^{\frac{5}{2}}\to \gamma^1 \overline{T}^{\frac{5}{2}}, \text{ strongly in } L^2([0,t];L^2(\partial\Omega)).
	\end{align}
	For the Dirichlet boundary condition \eqref{b3}, we have $\gamma^1 T_\varepsilon = T_b$, hence
	\begin{align*}
		\gamma^1 \overline{T} = T_b.
	\end{align*}
	This verifies \eqref{eq:Boundarycond} in the Dirichlet case. 

	\textbf{Robin boundary condition with $r>0$.} 
	% For the case of Robin boundary condition \eqref{b2} with $r > 0$, 
	We can use the following inequality from the energy inequality \eqref{eq:energyrobin}:
	\begin{align*}
		\frac{1}{\varepsilon^r} \int_0^t \|\gamma^1 T_\varepsilon - T_b\|_{L^5(\partial\Omega)}^5d\tau \le C,
	\end{align*}
	to deduce that 
	\begin{align*}
		\gamma^1 T_\varepsilon \to T_b, \text{ strongly in }L^5([0,t];L^5(\partial\Omega)).
	\end{align*}
	This combines \eqref{eq:traceconv} leads to 
	\begin{align*}
		\gamma^1 \overline{T} = T_b.
	\end{align*}
	The case of Robin boundary condition \eqref{b2} with $r=0$ will be shown later.

	% For the Robin boundary condition \eqref{b2}, we still \eqref{eq:traceconv}.
	\textbf{Boundary condition for $\overline{\psi}$.} To show the boundary condition for $\overline{\psi}$. From \eqref{eq:estC}, we get that 
	\begin{align*}
		&\psi_\varepsilon \in L^2([0,t];L^2(\Omega\times\mathbb{S}^2)),\\
		&(\varepsilon\partial_t + \beta \cdot \nabla) \psi_\varepsilon = \frac{1}{\varepsilon} (\psi_\varepsilon - T_\varepsilon^4) \in L^2([0,t];L^2(\Omega\times\mathbb{S}^2)),
	\end{align*}
	we can use the definition of the trace operator $\gamma^2$ to get 
	\begin{align*}
		\gamma^2 \psi_\varepsilon \rightharpoonup \overline{\gamma^2 \psi_\varepsilon} \text{ weakly in }L^2([0,t];L^2(\Sigma;|n\cdot\beta|d\beta d\sigma_x)).
	\end{align*}
	To show $\overline{\gamma^2 \psi_\varepsilon} = \gamma^2 \overline{T}^{4}$, we
	% Since $\psi_\varepsilon \in L^2([0,t]\times\Omega\times\mathbb{S}^2)$ and $(\varepsilon\partial_t + \beta \cdot \nabla) \psi_\varepsilon \in L^2([0,t]\times\Omega\times\mathbb{S}^2)$. Let $\rho\in C^{\infty}([0,t]\times\Omega\times\mathbb{S}^2)$ be a test function, we 
	multiply \eqref{eq:psieps} by $\varepsilon \rho$ with $\rho \in C^\infty([0,t]\times\Omega)$ and integrate over time and space, we get 
\begin{align*}% \label{WeaklimitBD} 
&\varepsilon \iint_{\Omega\times\mathbb{S}^2} \psi_\varepsilon(t) \rho(t) d\beta dx d\tau -\varepsilon \iint_{\Omega\times\mathbb{S}^2} \psi_{\varepsilon0} \rho(0) d\beta dx d\tau-\varepsilon\int_{0}^{t}\iint_{\Omega\times\mathbb{S}^2}\psi_\varepsilon\partial_{t}\rho d\beta dx\nonumber\\
&-\int_{0}^{t}\iint_{\Omega\times\mathbb{S}^2}\psi_\varepsilon\beta\cdot\nabla\rho d\beta dx d\tau+\int_{0}^{t}\iint_{\Sigma}(\beta\cdot n)\gamma^{2}\psi_\varepsilon\cdot\rho d\beta d\sigma_x d\tau\nonumber\\&=-\frac{1}{\varepsilon}\int_{0}^{t}\int_{\Omega\times\mathbb{S}^2}(\psi_\varepsilon-T_\varepsilon^4)\rho d\beta dx d\tau.
\end{align*}
We pass to the limit $\varepsilon\to 0$ in the above equation and use \eqref{eq:psic} and \eqref{eq:wk5} to get
\begin{align} \label{eq:bdbarpsiarho}
&-\int_{0}^{t}\iint_{\Omega\times\mathbb{S}^2}\overline{\psi}\beta\cdot\nabla\rho d\beta dx d\tau+\int_{0}^{t}\iint_{\Sigma}(\beta\cdot n)\overline{\gamma^{2}\psi_\varepsilon}\cdot\rho d\beta d\sigma_x d\tau\nonumber\\
&\quad=-\int_{0}^{t}\int_{\Omega\times\mathbb{S}^2}A\cdot \rho d\beta dx d\tau.
\end{align}
On the other hand, since we take the weak limit in 
\begin{align*}
	\varepsilon \partial_t \psi_\varepsilon + \beta \cdot \nabla \psi_\varepsilon = -\frac{1}{\varepsilon}(\psi_\varepsilon - T_\varepsilon^4),
\end{align*}
and get
\begin{align*}
	\beta \cdot \nabla \overline{\psi} = A.
\end{align*}
Applying the test function $\rho$ on this equation leads to 
\begin{align*} 
&-\int_{0}^{t}\iint_{\Omega\times\mathbb{S}^2}\overline{\psi}\beta\cdot\nabla\rho d\beta dx d\tau+\int_{0}^{t}\iint_{\Sigma}(\beta\cdot n)\gamma^{2}\overline{\psi}\cdot\rho d\beta d\sigma_x d\tau\nonumber\\
&\quad=-\int_{0}^{t}\int_{\Omega\times\mathbb{S}^2}A\cdot \rho d\beta dx d\tau.
\end{align*}
Comparing the above equation with \eqref{eq:bdbarpsiarho} leads to
\begin{align*}
	\gamma^2\overline{\psi} = \overline{\gamma^2\psi_\varepsilon}.
\end{align*}
On the other hand, we can use the boundary term in the energy inequality \eqref{eq:estdirichlet} or \eqref{eq:energyrobin}:
\begin{align*}
	&\qquad\frac{2\alpha - \alpha^2}{2\varepsilon} \int_0^t \|\gamma^2\psi_\varepsilon - \psi_b\|_{L^2(\Sigma_+;|n\cdot\beta| d\beta d\sigma_x)}^2 d\tau \le C,
\end{align*}
to deduce that
\begin{align*}
	\gamma^2 \psi_\varepsilon \to \psi_b, \text{ strongly in }L^2(\Sigma_+;|n\cdot\beta|d\beta d\sigma_x).
\end{align*}
It follows that
\begin{align*}
	\gamma^2 \overline{\psi}|_{\Sigma_+} = \psi_b.
\end{align*}
In addition, we can pass to the limit $\varepsilon\to 0$ in the boundary condition \eqref{bpsi} to get
\begin{align*}
	\gamma^2 \overline{\psi}|_{\Sigma_-} = \alpha \psi_b + (1-\alpha)\gamma^2 L\psi_\varepsilon|_{\Sigma_+} = \psi_b.
\end{align*}
Therefore,
\begin{align}
	\gamma^2 \overline{\psi} = \psi_b = T_b^4.
\end{align}

\textbf{Robin boundary condition with $r=0$.} 
% Notice that in the case of Robin boundary condition \eqref{b2} with $r=0$, 
We can use the above formula to get
\begin{align*}
	\gamma^2 \overline{T}^4 = T_b^4,
\end{align*}
from which we can deduce that
\begin{align*}
	\gamma^2 \overline{T} =T_b,
\end{align*}
i.e. \eqref{eq:Boundarycond} also holds in the case of Robin boundary condition with $r=0$. Hence we finish the proof.
\end{proof}

\begin{rem}
Here in the proof we use the continuity of $T_\varepsilon^{\frac{5}{2}}$ and Lemma \ref{Lemmalions1996mathematical} to show the strong convergence of $T_\varepsilon$. If we drop the Laplacian term in the equation \ref{eq:Teps}, we can no longer show this by the above proof. However, thanks to the averaging lemma, i.e. Lemma \ref{AveragingLemma}, we have
% Here in the proof we use the boundness of $T_\varepsilon^{\frac{5}{2}}$ in $L^2([0,t];H^1(\Omega))$ to get the space compactness in $T_\varepsilon$. We can also avoid use this but using averaging lemma on $\psi_\varepsilon^m$ instead. According to \eqref{eq:unibd3} and \eqref{eq:unibd4}, the assumptions of Lemma \ref{AveragingLemma} hold, and it follows 
that for any $\eta \in C^\infty(\mathbb{S}^2)$,
\begin{align}
	\left\|\int_{\mathbb{S}^2} (\psi_\varepsilon(\cdot,\cdot+y,\beta) - \psi_\varepsilon(\cdot,\cdot,\beta))\eta(\beta)d\beta\right\|_{L^2([0,t];L^2(\mathbb{T}^3))} \to 0,
\end{align}
as $y\to 0$ uniformly in $\varepsilon$.
Thus we can take $h$ to be 
\begin{align*}
	\langle \psi_\varepsilon \rangle = \int_{\mathbb{S}^2} \psi_\varepsilon d\beta
\end{align*}
instead of $(T_\varepsilon)^{\frac{5}{2}}$ in Lemma \ref{Lemmalions1996mathematical} and the equation \eqref{eq:weakprod} becomes
\begin{align}\label{eq:prodwk123}
 	% \left(T_\varepsilon^{m} + \int_{\mathbb{S}^2} \psi_\varepsilon^{m} d\beta\right) \int_{\mathbb{S}^2} \psi_\varepsilon^{m} d\beta \rightharpoonup \left(T_\varepsilon + \int_{\mathbb{S}^2} \psi_\varepsilon d\beta \right)  \int_{\mathbb{S}^2} \psi_\varepsilon d\beta.
 	\left(T_\varepsilon +\langle \psi_\varepsilon \rangle \right)\langle \psi_\varepsilon \rangle \rightharpoonup \left(\overline{T} + \langle \overline{\psi} \rangle \right)  \langle \overline{\psi} \rangle.
\end{align} 
Due to the strong convergence in \eqref{eq:wk4},
\begin{align*}
	\langle \psi_\varepsilon - T_\varepsilon^4 \rangle \to 0,
\end{align*}
which leads to
\begin{align*}
	 &T_\varepsilon \langle \psi_\varepsilon \rangle  =  T_\varepsilon \langle \psi_\varepsilon - T_\varepsilon^4 \rangle + 4\pi T_\varepsilon^5 \rightharpoonup \langle A \rangle \overline{T} + 4\pi \overline{T_\varepsilon^5}, \\
	 &\langle \psi_\varepsilon \rangle \langle \psi_\varepsilon \rangle = \langle \psi_\varepsilon - T_\varepsilon^4 \rangle \left(\langle \psi_\varepsilon - T_\varepsilon^4 \rangle + 4\pi T_\varepsilon^4\right) +4\pi T_\varepsilon^4 \langle \psi_\varepsilon - T_\varepsilon^4 \rangle + 16\pi^{2}T_\varepsilon^8 \\
	 &\qquad \rightharpoonup \langle A\rangle (\langle A\rangle + 4\pi \overline{T_\varepsilon^4}) + 4\pi \overline{T_\varepsilon^4}\langle A\rangle + 16\pi^2 \overline{T_\varepsilon^8},
\end{align*}
weakly in $L^2_{\loc}([0,\infty);L^2(\mathbb{T}^3))$ as $\varepsilon \to 0$. Thus we can also pass to the limit $\varepsilon \to 0$ in \eqref{eq:prodwk123} and get the same limit with
\begin{align*}
	\langle A \rangle & \overline{T} +  4\pi \overline{T_\varepsilon^5} + \langle A\rangle(\langle A\rangle + 4\pi \overline{T_\varepsilon^4})+4\pi \overline{T_\varepsilon^4}\langle A\rangle + 16\pi^2 \overline{T_\varepsilon^8}\\
	 =& (\overline{T} + (\langle A\rangle + 4\pi \overline{T_\varepsilon^4}))(\langle A\rangle + 4\pi \overline{T_\varepsilon^4}),
\end{align*}
which implies that
\begin{align*}
	 4\pi \overline{T_\varepsilon^5} + 16\pi^2 \overline{T_\varepsilon^8} = \overline{T}4\pi \overline{T_\varepsilon^4} + 16\pi^2 \overline{T_\varepsilon^4}\cdot\overline{T_\varepsilon^4}.
\end{align*}
We can also apply the Young measure theory to get 
\begin{align*}
	4\pi& \int_{\mathbb{R}} \lambda^5 d\nu_x(\lambda) +  16\pi^2 \int_{\mathbb{R}} \lambda^8 d\nu_x(\lambda)\\
	=& 4\pi \int_{\mathbb{R}} \int_{\mathbb{R}} \lambda \mu^4 d\nu_x(\lambda)d\nu_x(\mu) + 16\pi^2 \int_{\mathbb{R}}\int_{\mathbb{R}} \lambda^4\mu^4 d\nu_x(\lambda) d\nu_x(\mu).
\end{align*}
It implies that
\begin{align*}
	4\pi& \int_{\mathbb{R}}\int_{\mathbb{R}} \lambda(\lambda^4-\mu^4)  d\nu_x(\lambda) d\nu_x(\mu) + 16\pi^2 \int_{\mathbb{R}}\int_{\mathbb{R}} \lambda^4(\lambda^4-\mu^4)  d\nu_x(\lambda) d\nu_x(\mu) \\
	=& 4\pi \int_{\mathbb{R}}\int_{\mathbb{R}} (\lambda-\mu)(\lambda^4-\mu^4)  d\nu_x(\lambda) d\nu_x(\mu) + 16\pi^2 \int_{\mathbb{R}}\int_{\mathbb{R}}(\lambda^4-\mu^4)^2  d\nu_x(\lambda) d\nu_x(\mu) \\
	=&0.
\end{align*}
Therefore, $\nu_x$ is concentrated at $\delta_{T_\varepsilon(x)}$. We can thus conclude that \eqref{eq:tc} holds. Therefore, the above theorem also holds for the system
% The similar proof with the above argument also leads to the global existence of weak solutions to the following system 
\begin{align*}
	\partial_t T_\varepsilon =~&\frac{1}{\varepsilon^2}(\langle \psi_\varepsilon\rangle - 4\pi T_\varepsilon^4), \\
	\partial_t \psi_\varepsilon + \frac{1}{\varepsilon} \beta \cdot \nabla \psi_\varepsilon =~& -\frac{1}{\varepsilon^2}(\psi_\varepsilon - T_\varepsilon^4).
\end{align*} 
\end{rem}

\section{The relative entropy method}\label{section4}
The compactness method gives a clear justification of the diffusive limit of the system \eqref{eq:Teps}-\eqref{eq:psieps}. In this section, we give the rate of convergence of the diffusive limit under assumption on regularity of the limit system \eqref{hgm1.0}.

 % the the difference between the system \eqref{eq:Teps}-\eqref{eq:psieps} and the limit system \eqref{hgm1.0} when $\varepsilon$ is close to $0$. 
 
 We will introduce a relative entropy function to compare the solutions between \eqref{eq:Teps}-\eqref{eq:psieps} and \eqref{hgm1.0}. The difference of their solutions are estimated using this relative entropy function.

% From the previous section, the radiative heat transfer equations 
% \begin{align}
% \partial_t T_\varepsilon = \Delta T_\varepsilon +\frac{1}{\varepsilon^2} \int_{\mathbb{S}^{2}} (\psi_\varepsilon-(T_\varepsilon)^4  ) d\beta, \label{eq:Teps}\\
% \partial_t \psi_\varepsilon + \frac{1}{\varepsilon} \beta \cdot \nabla \psi_\varepsilon = -\frac{1}{\varepsilon^2}(\psi_\varepsilon-(T_\varepsilon)^4). \label{eq:psieps}
% \end{align}
% converge to the equation 
% \begin{align}\label{hgm1.0}
% \partial_t \overline{T} = \Delta \overline{T}-4\pi \partial_t \overline{T}^4 + \frac{4}{3}\pi \Delta \overline{T}^4
% \end{align}
% in the sense of distributions.

To compare solutions of the equations \eqref{eq:Teps}-\eqref{eq:psieps} and \eqref{hgm1.0}, we notice that the limit system \eqref{hgm1.0} does not include the equation for $\psi_\varepsilon$. To use the relative entropy method, we define $\overline{\psi}$ as follows 

% inspired by the formal derivation of the limit system in section \ref{sec:intro}:

\begin{align*}
\overline{\psi} = \overline{T}^4 - \varepsilon \beta \cdot \nabla \overline{T}^4 -\varepsilon^2 \partial_t \overline{T}^4 + \varepsilon^2 \beta \cdot \nabla (\beta \cdot  \nabla \overline{T}^4).
\end{align*}
so that $\overline{T}$ and $\overline{\psi}$ satisfies
\begin{align}
\partial_t \overline{T} =~& \Delta \overline{T} + \frac{1}{\varepsilon^2} \int_{\mathbb{S}^{2}} ( \overline{\psi}-\overline{T}^4) d\beta, \label{eq_1}\\
\partial_t \overline{\psi} + \frac{1}{\varepsilon} \beta \cdot \nabla \overline{\psi} =~& -  \frac{1}{\varepsilon^2}(\overline{\psi}-\overline{T}^4)+\overline{R}. \label{eq_2}
\end{align}
where
\begin{align*}
\overline{R} =& \partial_t \overline{\psi} + \frac{1}{\varepsilon} \beta \cdot \nabla \overline{\psi} + \frac{1}{\varepsilon^2} (\overline{\psi} - \overline{T}^4) \\
=& \partial_t \overline{T}^4 + \frac{1}{\varepsilon} \beta \cdot \nabla \overline{T}^4-\beta \cdot \nabla (\beta \cdot  \nabla \overline{T}^4) - \frac{1}{\varepsilon} \beta \cdot \nabla \overline{T}^4 -\partial_t \overline{T}^4 +\beta \cdot \nabla (\beta \cdot  \nabla \overline{T}^4) \\
&-\varepsilon \beta \cdot \nabla \partial_t \overline{T}^4+\varepsilon \beta \cdot \nabla (-\partial_t \overline{T}^4+\beta \cdot \nabla (\beta \cdot  \nabla \overline{T}^4)) 
-\varepsilon^2\partial_t^2 \overline{T}^4+\varepsilon^2 \beta \cdot \nabla (\beta \cdot  \nabla \overline{T}^4) \\
=&\varepsilon \beta \cdot \nabla (-2 \partial_t \overline{T}^4+ \beta \cdot \nabla (\beta \cdot  \nabla \partial_t\overline{T}^4)) - \varepsilon^2(\partial_t \overline{T}^4-\beta \cdot \nabla (\beta \cdot  \nabla \partial_t \overline{T}^4)).
\end{align*}

We define the energy function
\begin{align*}
H(T_\varepsilon,\psi_\varepsilon):=\int_{\Omega}& \frac{T_\varepsilon^5}{5} dx + \iint_{\Omega\times\mathbb{S}^2} \frac{\psi_\varepsilon^2}{2}  d\beta dx.
\end{align*}
The relative energy function is defined to be
\begin{align*}
H(T_\varepsilon&,\psi_\varepsilon|\overline{T},\overline{\psi})\\
 :=& E(T_\varepsilon,\psi_\varepsilon) - E(\overline{T},\overline{\psi}) - \left\langle\frac{\delta E}{\delta T}(\overline{T},\overline{\psi}),T_\varepsilon-\overline{T}) \right \rangle - \left\langle\frac{\delta E}{\delta \psi}(\overline{T},\overline{\psi}),\psi_\varepsilon-\overline{\psi}) \right \rangle \\
=& \int_\Omega \frac{T_\varepsilon^5-\overline{T}^5-5\overline{T}^4 (T_\varepsilon - \overline{T})}{5} dx + \int_\Omega\int_{\mathbb{S}^2} \frac{(\psi_\varepsilon - \overline{\psi})^2}{2} d\beta dx.
\end{align*}

We will apply the relative entropy method to compare the solutions $(T_\varepsilon,\psi_\varepsilon)$ and $(\overline{T},\overline{\psi})$ under three different boundary conditions: in the torus \eqref{b1}, Dirichlet boundary condition \eqref{b3} and Robin boundary condition \eqref{b2}. The main result of this section is the following theorem.
\begin{theorem}\label{thmre}
	Assume  $(T_\varepsilon,\psi_\varepsilon)$ is a weak solution of the system \eqref{eq:Teps}-\eqref{eq:psieps} and  $\overline{T}$ is a strong solution of the equation \eqref{eq:Teps} with $\overline{T}\in H^2(\Omega)$. Assume the well-prepared boundary condition \eqref{eq:wellbc} holds. Suppose $ \overline{T} \ge c >0$. Then the following inequality holds:
	\begin{align}\label{eq:Heps}
	\int_\Omega& (T_\varepsilon-\overline{T})^2 + (T_\varepsilon-\overline{T})^4 \bigg|_tdx + \iint_{\Omega\times\mathbb{S}^2} (\psi_\varepsilon-\overline{\psi})^2 \bigg|_t d\beta dx  \nonumber\\
	&\le \int_\Omega (T_{\varepsilon0}-\overline{T}_0)^2 + (T_{\varepsilon0}-\overline{T}_0)^4 dx + \iint_{\Omega\times\mathbb{S}^2} (\psi_{\varepsilon0}-\overline{\psi}_0)^2 d\beta dx  + C\varepsilon^s.
	\end{align}
	Here $s=2$ for the case of torus, $s=\min\{1,r\}$ for the case of Robin boundary condition \eqref{b2} with $r>0$ and $s=1$ for the case of nonhomogeneous Dirichlet condition \eqref{b3}.

	Furthermore, if the initial data is well-prepared that \eqref{eq:wellinitial} holds, and $T_{\varepsilon0}-\overline{T}_0 \to 0$ as $\varepsilon \to 0$, then 
	$T_\varepsilon \to \overline{T}$ and $\psi_\varepsilon \to \overline{\psi}$ strongly in $L^2(\Omega)$ and $L^2(\Omega\times\mathbb{S}^2)$, respectively for any $t>0$. 
\end{theorem}

\subsection{The case of torus}
We next derive the relative entropy inequality for the case of torus $\Omega=\mathbb{T}^3$.
\begin{lemma} \label{lm1}
	Assume  $T_\varepsilon$ is a weak solution of the system \eqref{eq:Teps}-\eqref{eq:psieps}, and $\overline{T}$ is a smooth solution of the equation \eqref{hgm1.0}.  The following inequality holds:
	\begin{align}\label{eq:Hevol}
		H(&T_\varepsilon,\psi_\varepsilon | \overline{T},\overline{\psi}) \bigg|_{t}+ \frac{16}{25}\int_0^t\int_\Omega (\nabla T_\varepsilon^{\frac{5}{2}} - \nabla \overline{T}^{\frac{5}{2}})^2 dxd\tau \nonumber\\
		\le& H(T_\varepsilon,\psi_\varepsilon | \overline{T},\overline{\psi}) \bigg|_{0}   +\frac{32}{25} \int_0^t\int_\Omega (T_\varepsilon^{\frac{5}{2}} - \overline{T}^{\frac{5}{2}} - \frac{5}{2} \overline{T}^{\frac{3}{2}}(T_\varepsilon - \overline{T}))\Delta \overline{T}^{\frac{5}{2}} dxd\tau \nonumber\\
		&+\frac{1}{\varepsilon^2} \int_0^t\iint_{\Omega\times\mathbb{S}^2}\left(T_\varepsilon^4 - \overline{T}^4 - 4 \overline{T}^3(T_\varepsilon - \overline{T})\right) \left(\overline{\psi} - \overline{T}^4\right)  d\beta dxd\tau \nonumber\\
		& - \int_0^t\iint_{\Omega\times\mathbb{S}^2} (\psi_\varepsilon-\overline{\psi}) \overline{R} d\beta dxd\tau- \frac{1}{\varepsilon^2}\int_0^t\iint_{\Omega\times\mathbb{S}^2} (\psi_\varepsilon - T_\varepsilon^4 - (\overline{\psi}-\overline{T}^4))^2 d\beta dxd\tau .
	\end{align}
\end{lemma}
\begin{proof}
	First, we recall that from \eqref{eq:energythm1}, the energy function for the equations \eqref{eq:Teps}-\eqref{eq:psieps} satisfies
	\begin{align}\label{eq:ene1}
	H(T_\varepsilon,\psi_\varepsilon) \bigg|_{0}^t + \frac{16}{25}\int_0^t \int_{\Omega} \left|\nabla (T_\varepsilon)^{\frac{5}{2}}\right|^2dxd\tau + \frac{1}{\varepsilon^2}\int_0^t\int_{\Omega} \int_{\mathbb{S}^{2}} (\psi_\varepsilon-(T_\varepsilon)^4  )^2 d \beta dxd\tau \le 0.
	\end{align}
	The function $(\overline{T}, \overline{\psi})$ also satisfies a similar equality:
	\begin{align}\label{eq:ene2}
	&H(\overline{T},\overline{\psi}) \bigg|_{0}^t + \frac{16}{25}\int_0^t \int_{\Omega} \left|\nabla (\overline{T})^{\frac{5}{2}}\right|^2dxd\tau + \frac{1}{\varepsilon^2}\int_0^t\int_{\Omega} \int_{\mathbb{S}^{2}} (\overline{\psi}-\overline{T}^4  )^2 d \beta dxd\tau \nonumber\\
	&\quad= \int_0^t\iint_{\Omega\times\mathbb{S}^2} \overline{\psi}\cdot \overline{R} d\beta dxd\tau.
	\end{align}
	Next we consider the equations of the difference $(T_\varepsilon - \overline{T},\,\psi_\varepsilon-\overline{\psi})$ using the definition of weak solutions \eqref{eq:weakt1}-\eqref{eq:weakt2} :
	\begin{align}
		-\int_0^\infty&\int_\Omega \varphi_t(T_\varepsilon-\overline{T})dxdt - \int_\Omega \varphi (T_\varepsilon-\overline{T}) \bigg|_{t=0}dx \nonumber\\
		=& \int_0^\infty \int_\Omega \Delta \varphi (T_\varepsilon - \overline{T}) dxdt + \frac{1}{\varepsilon^2} \int_0^\infty \iint_{\Omega\times\mathbb{S}^2}(\psi_\varepsilon - \overline{\psi} - T_\varepsilon^4 + \overline{T}^4) \varphi d\beta dxdt, \label{eq:we1} \\
		-\int_0^\infty&\iint_{\Omega\times\mathbb{S}^2} \rho_t (\psi_\varepsilon - \overline{\psi})d\beta dxdt - \int_\Omega\int_{\mathbb{S}^2} \rho(\psi_\varepsilon- \overline{\psi}) \bigg|_{t=0} d\beta dx \nonumber\\&\quad- \frac{1}{\varepsilon} \int_0^\infty\iint_{\Omega\times\mathbb{S}^2}(\psi_\varepsilon-\overline{\psi}) \beta \cdot \nabla \rho d\beta dxdt \nonumber\\
		=& -\frac{1}{\varepsilon^2} \int_0^\infty\iint_{\Omega\times\mathbb{S}^2} (\psi_\varepsilon - \overline{\psi} - T_\varepsilon^4 + \overline{T}^4)\rho d\beta dxdt - \int_0^\infty\iint_{\Omega\times\mathbb{S}^2} \rho \overline{R} d\beta dxdt.\label{eq:we2}
	\end{align}
	We introduce the following test function
	\begin{align}\label{eq:testfun}
		\varphi = \theta(\tau) \overline{T}^4,\quad \rho = \theta(\tau) \overline{\psi},
	\end{align}
	where
	\begin{align*}
		\theta(\tau) := \left\{\begin{array}{cl}
			1, &\text{ for } 0\le \tau < t, \\
			\frac{t-\tau}{\delta} + 1, &\text{ for } t \le \tau < t+\delta, \\
			0, &\text{ for } \tau \ge t+\delta.
		\end{array}\right.
	\end{align*}
	Taking these test functions into \eqref{eq:we1}-\eqref{eq:we2} and let $\delta \to 0$, we obtain
	\begin{align*}
		\int_{\Omega} &\overline{T}^4(T_\varepsilon-\overline{T}) \bigg|_{\tau=0}^t dx  - \int_0^t \int_\Omega \partial_\tau(\overline{T}^4) (T_\varepsilon-\overline{T}) dxd\tau \\
		&= \int_0^t\int_\Omega \Delta \overline{T}^4(T_\varepsilon-\overline{T})dxd\tau + \frac{1}{\varepsilon^2} \int_0^t \int_\Omega\int_{\mathbb{S}^2} \overline{T}^4(\psi_\varepsilon-\overline{\psi}-T_\varepsilon^4+\overline{T}^4) d\beta dxd\tau. \\
		\int_\Omega &\int_{\mathbb{S}^2} \overline{\psi} (\psi_\varepsilon-\overline{\psi}) \bigg|_{\tau=0}^t d\beta dx - \int_0^t\int_\Omega\int_{\mathbb{S}^2} (\partial_\tau \overline{\psi})(\psi_\varepsilon-\overline{\psi}) d\beta dxd\tau \\
		&\qquad -\frac{1}{\varepsilon} \int_0^t\int_\Omega\int_{\mathbb{S}^2} (\psi_\varepsilon-\overline{\psi})\beta \cdot \nabla \overline{\psi} d\beta dx d\tau \\
		&= - \frac{1}{\varepsilon^2}\int_0^t\int_\Omega\int_{\mathbb{S}^2} \overline{\psi}(\psi_\varepsilon-\overline{\psi}-T_\varepsilon^4+\overline{T}^4) d\beta dxd\tau - \int_0^t \int_\Omega\int_{\mathbb{S}^2} \overline{\psi} \cdot \overline{R} d\beta dxd\tau.
	\end{align*}
	Using the equations \eqref{eq_1} and \eqref{eq_2}, the above equations become
	\begin{align*}
		\int_{\Omega}& \overline{T}^4(T_\varepsilon-\overline{T}) \bigg|_{\tau=0}^t dx  \\
		=& \int_0^t \int_\Omega 4\overline{T}^3 \Delta \overline{T} (T_\varepsilon-\overline{T}) dxd\tau 
		+ \int_0^t\int_\Omega \Delta \overline{T}^4(T_\varepsilon-\overline{T})dxd\tau \\
		& +\frac{1}{\varepsilon^2} \int_0^t\iint_{\Omega\times\mathbb{S}^2}4\overline{T}^3(\overline{\psi}-\overline{T}^4)(T_\varepsilon-\overline{T}) d\beta dxd\tau \\
		&+ \frac{1}{\varepsilon^2} \int_0^t \int_\Omega\int_{\mathbb{S}^2} \overline{T}^4(\psi_\varepsilon-\overline{\psi}-T_\varepsilon^4+\overline{T}^4) d\beta dxd\tau, \\
		\int_\Omega& \int_{\mathbb{S}^2} \overline{\psi} (\psi_\varepsilon-\overline{\psi}) \bigg|_{\tau=0}^t d\beta dx \\
		= &  \int_0^t\int_\Omega\int_{\mathbb{S}^2} (-\frac{1}{\varepsilon} \beta \cdot \nabla \overline{\psi})(\psi_\varepsilon-\overline{\psi}) d\beta dxd\tau +\frac{1}{\varepsilon} \int_0^t\int_\Omega\int_{\mathbb{S}^2} (\psi_\varepsilon-\overline{\psi})\beta \cdot \nabla \overline{\psi} d\beta dx d\tau \\
		&-\frac{1}{\varepsilon^2} \int_0^t\int_\Omega\int_{\mathbb{S}^2}(\overline{\psi}-\overline{T}^4)(\psi_\varepsilon-\overline{\psi})d\beta dxd\tau + \int_0^t \int_\Omega\int_{\mathbb{S}^2} \overline{R} (\psi_\varepsilon-\overline{\psi}) d\beta dxd\tau\\
		& - \frac{1}{\varepsilon^2}\int_0^t\int_\Omega\int_{\mathbb{S}^2} \overline{\psi}(\psi_\varepsilon-\overline{\psi}-T_\varepsilon^4+\overline{T}^4) d\beta dxd\tau - \int_0^t \int_\Omega\int_{\mathbb{S}^2} \overline{\psi} \cdot \overline{R} d\beta dxd\tau.
	\end{align*}
	Adding them together gives
\begin{align*}
		\int_{\Omega} &\overline{T}^4(T_\varepsilon-\overline{T}) \bigg|_{\tau=0}^t dx + \iint_{\Omega\times\mathbb{S}^2} \overline{\psi} (\psi_\varepsilon-\overline{\psi}) \bigg|_{\tau=0}^t d\beta dx \\
		=& \int_0^t \int_\Omega 4\overline{T}^3 \Delta \overline{T} (T_\varepsilon-\overline{T}) + \Delta \overline{T}^4(T_\varepsilon-\overline{T})dxd\tau \\
		& -\frac{1}{\varepsilon^2} \int_0^t\iint_{\Omega\times\mathbb{S}^2}(\overline{\psi}-\overline{T}^4)(T_\varepsilon^4-\overline{T}^4-4\overline{T}^3(T_\varepsilon-\overline{T}) )d\beta dxd\tau \\
		&- \frac{2}{\varepsilon^2} \int_0^t \int_\Omega\int_{\mathbb{S}^2} (\overline{\psi}-\overline{T}^4)(\psi_\varepsilon-\overline{\psi}-T_\varepsilon^4+\overline{T}^4) d\beta dxd\tau \\
		&+ \int_0^t\int_\Omega\int_{\mathbb{S}^2} \overline{R}(\psi_\varepsilon-\overline{\psi}) d\beta dxd\tau -\int_0^t\int_\Omega\int_{\mathbb{S}^2} \overline{\psi} \cdot \overline{R} d\beta dxd\tau. 
	\end{align*}
	We substract the above equation from the difference between \eqref{eq:ene1} and \eqref{eq:ene2} and arrive at the following inequality:
	\begin{align}\label{eq:Hincal}
	 	H(T_\varepsilon&,\psi_\varepsilon | \overline{T},\overline{\psi}) \bigg|_{t} \nonumber \\
	 	\le& H(T_\varepsilon,\psi_\varepsilon | \overline{T},\overline{\psi}) \bigg|_{0}  - \frac{16}{25} \int_0^t\int_\Omega \left(\left|\nabla(T_\varepsilon)^{\frac{5}{2}}\right|^2 - \left|\nabla(\overline{T})^{\frac{5}{2}}\right|^2\right) dxd\tau \nonumber\\
		&-\int_0^t \int_\Omega 4\overline{T}^3 \Delta \overline{T} (T_\varepsilon-\overline{T}) + \Delta \overline{T}^4(T_\varepsilon-\overline{T})dxd\tau \nonumber\\
		& - \int_0^t\iint_{\Omega\times\mathbb{S}^2} (\psi_\varepsilon-\overline{\psi}) \overline{R} d\beta dxd\tau- \frac{1}{\varepsilon^2}\int_0^t\iint_{\Omega\times\mathbb{S}^2} (\psi_\varepsilon - T_\varepsilon^4 - (\overline{\psi}-\overline{T}^4))^2 d\beta dxd\tau \nonumber\\
		&+\frac{1}{\varepsilon^2} \int_0^t\iint_{\Omega\times\mathbb{S}^2}\left(T_\varepsilon^4 - \overline{T}^4 - 4 \overline{T}^3(T_\varepsilon - \overline{T})\right) \left(\overline{\psi} - \overline{T}^4\right)  d\beta dxd\tau.
	 \end{align}
	 To simplify the inequality, we rewrite the third term on the right hand side as:
	 \begin{align}
        -\int_0^t& \int_\Omega 4\overline{T}^3 \Delta \overline{T} (T_\varepsilon-\overline{T}) + \Delta \overline{T}^4(T_\varepsilon-\overline{T})dxd\tau \nonumber\\
        =& \int_0^t\int_\Omega (T_\varepsilon^4 - \overline{T}^4 - 4 \overline{T}^3(T_\varepsilon - \overline{T}))\Delta \overline{T} dxd\tau \nonumber\\
        &- \int_0^t\int_\Omega (T_\varepsilon^4 - \overline{T}^4) \Delta \overline{T}  + \Delta \overline{T}^4 (T_\varepsilon - \overline{T}) dxd\tau \nonumber\\
        =& \int_0^t\int_\Omega (T_\varepsilon^4 - \overline{T}^4 - 4 \overline{T}^3(T_\varepsilon - \overline{T}))\Delta \overline{T} dxd\tau  -\frac{32}{25} \int_0^t\int_\Omega \left|\nabla \overline{T}^{\frac{5}{2}}\right|^2 dxd\tau\nonumber\\& - \int_0^t \int_\Omega T_\varepsilon^4 \Delta \overline{T} + T_\varepsilon \overline{T}^4 dxd\tau.\label{eq:4t3d}
    \end{align}
    Here we use the fact that 
    \begin{align}\label{eq:calt52}
    	\int_\Omega \overline{T}^4 \Delta \overline{T} dx = \int_\Omega \overline{T} \Delta \overline{T}^4 dx = - \frac{16}{25} \int_\Omega \left|\nabla \overline{T}^{\frac{5}{2}}\right|^2 dx.
    \end{align}
    We calculate the last term in \eqref{eq:4t3d} as
	\begin{align*}
		-\int_{0}^t \int_\Omega (T_\varepsilon^4 \Delta \overline{T} + T_\varepsilon \Delta \overline{T}^4) dxd\tau =& -\int_0^t\int_\Omega T_\varepsilon^4 \Delta \overline{T} + 4T_\varepsilon \overline{T}^3 \Delta \overline{T} + 12 T_\varepsilon \overline{T}^2 |\nabla \overline{T}|^2 dxd\tau.
	\end{align*}
	Using
	\[\Delta \overline{T}^{\frac{5}{2}} = \nabla \cdot \left(\frac{5}{2} \overline{T}^{\frac{3}{2}}\nabla \overline{T} \right) = \frac{5}{2} \overline{T}^{\frac{3}{2}}\Delta\overline{T} + \frac{15}{4} \overline{T}^{\frac{1}{2}}|\nabla \overline{T}|^2,\]
	we obtain 
	\begin{align}
		-\int_{0}^t &\int_\Omega (T_\varepsilon^4 \Delta \overline{T} + T_\varepsilon \Delta \overline{T}^4) dxd\tau \nonumber \\
		=& -\int_0^t\int_\Omega T_\varepsilon^4 \Delta \overline{T} + 4T_\varepsilon \overline{T}^3\Delta \overline{T} + 12 T_\varepsilon \overline{T}^{\frac{3}{2}} \cdot \frac{4}{15}(\Delta \overline{T}^{\frac{5}{2}} - \frac{5}{2}\overline{T}^{\frac{3}{2}}\Delta \overline{T})dxd\tau \nonumber\\
		=& -\int_0^t \int_\Omega T_\varepsilon^4\overline{T} + 4T_\varepsilon \overline{T}^3 \Delta \overline{T} + \frac{16}{5} T_\varepsilon \overline{T}^{\frac{3}{2}} \Delta \overline{T}^{\frac{5}{2}} - 8 T_\varepsilon \overline{T}^3\Delta \overline{T} dxd\tau \nonumber\\
		=& -\int_0^t \int_\Omega T_\varepsilon^4\overline{T} - 4T_\varepsilon \overline{T}^3 \Delta \overline{T} + \frac{16}{5} T_\varepsilon \overline{T}^{\frac{3}{2}} \Delta \overline{T}^{\frac{5}{2}}dxd\tau \nonumber\\
		=& -\int_0^t\int_\Omega (T_\varepsilon^4 - \overline{T}^4 - 4\overline{T}^3(T_\varepsilon - \overline{T}))\Delta \overline{T} dxd\tau  + 3\int_0^t\int_\Omega \overline{T}^4 \Delta \overline{T} dxd\tau\nonumber\\
		& - \frac{16}{5} \int_0^t\int_\Omega T_\varepsilon \overline{T}^{\frac{3}{2}}\Delta \overline{T}^{\frac{5}{2}} dxd\tau. \label{eq:T4Tcal}
	\end{align}
	The last term in the above equation can be calculate as
	\begin{align*}
		-\frac{16}{5} \int_0^t\int_\Omega T_\varepsilon \overline{T}^{\frac{3}{2}}\Delta \overline{T}^{\frac{5}{2}} dxd\tau =&\frac{16}{5}\int_0^t\int_\Omega  \frac{2}{5}(T_\varepsilon^{\frac{5}{2}} - \overline{T}^{\frac{5}{2}} - \frac{5}{2} \overline{T}^{\frac{3}{2}}(T_\varepsilon - \overline{T}))\Delta \overline{T}^{\frac{5}{2}} dxd\tau \\
		& - \frac{32}{25} \int_0^t\int_\Omega T_\varepsilon^{\frac{5}{2}} \Delta \overline{T}^{\frac{5}{2}} dxd\tau - \frac{16}{5} \frac{3}{5} \int_0^t\int_\Omega \overline{T}^{\frac{5}{2}} \Delta \overline{T}^{\frac{5}{2}} dxd\tau.
	\end{align*}
	Taking this equation into \eqref{eq:T4Tcal} and using \eqref{eq:calt52}, we obtain 
	\begin{align*}
		-\int_{0}^t& \int_\Omega (T_\varepsilon^4 \Delta \overline{T} + T_\varepsilon \Delta \overline{T}) dxd\tau\\
		 =& -\int_0^t\int_\Omega (T_\varepsilon^4 - \overline{T}^4 - 4\overline{T}^3(T_\varepsilon - \overline{T}))\Delta \overline{T} dxd\tau \\
		& +\frac{32}{25} \int_0^t\int_\Omega (T_\varepsilon^{\frac{5}{2}} - \overline{T}^{\frac{5}{2}} - \frac{5}{2} \overline{T}^{\frac{3}{2}}(T_\varepsilon - \overline{T}))\Delta \overline{T}^{\frac{5}{2}} dxd\tau -\frac{32}{25}\int_0^t\int_\Omega T_\varepsilon^{\frac{5}{2}} \Delta \overline{T}^{\frac{5}{2}} dxd\tau .
	\end{align*}
	Taking it into \eqref{eq:4t3d} and using 
	\begin{align*}
	&-\frac{16}{25} \int_0^t \left|\nabla T_\varepsilon^{\frac{5}{2}}\right|^2-\frac{16}{25} \int_0^t \left|\nabla \overline{T}^{\frac{5}{2}}\right|^2dxd\tau + \frac{32}{25} \int_0^t \nabla T_\varepsilon^{\frac{5}{2}}\cdot\nabla \overline{T}^{\frac{5}{2}} \\
	&\quad= -\frac{16}{25}\int_0^t\int_\Omega (\nabla T_\varepsilon^{\frac{5}{2}} - \nabla \overline{T}^{\frac{5}{2}})^2 dxd\tau,
	\end{align*}
	inequality \eqref{eq:Hincal} becomes \eqref{eq:Hevol} and finished the proof.
\end{proof}

We now prove Theorem \ref{thmre}.

\begin{proof}[Proof of Theorem \ref{thmre}]
	From Lemma \ref{lm1}, we have 
		\begin{align}\label{relpf}
		H(T_\varepsilon&,\psi_\varepsilon | \overline{T},\overline{\psi}) \bigg|_{t}+ \frac{16}{25}\int_0^t\int_\Omega (\nabla T_\varepsilon^{\frac{5}{2}} - \nabla \overline{T}^{\frac{5}{2}})^2 dxd\tau \nonumber\\
		\le& H(T_\varepsilon,\psi_\varepsilon | \overline{T},\overline{\psi}) \bigg|_{0}   +\frac{32}{25} \int_0^t\int_\Omega (T_\varepsilon^{\frac{5}{2}} - \overline{T}^{\frac{5}{2}} - \frac{5}{2} \overline{T}^{\frac{3}{2}}(T_\varepsilon - \overline{T}))\Delta \overline{T}^{\frac{5}{2}} dxd\tau \nonumber\\
		&+\frac{1}{\varepsilon^2} \int_0^t\iint_{\Omega\times\mathbb{S}^2}\left(T_\varepsilon^4 - \overline{T}^4 - 4 \overline{T}^3(T_\varepsilon - \overline{T})\right) \left(\overline{\psi} - \overline{T}^4\right)  d\beta dxd\tau \nonumber\\
		& - \int_0^t\iint_{\Omega\times\mathbb{S}^2} (\psi_\varepsilon-\overline{\psi}) \overline{R} d\beta dxd\tau- \frac{1}{\varepsilon^2}\int_0^t\iint_{\Omega\times\mathbb{S}^2} (\psi_\varepsilon - T_\varepsilon^4 - (\overline{\psi}-\overline{T}^4))^2 d\beta dxd\tau \nonumber\\
		=& I_1 + I_2 + I_3 + I_4 + I_5.
	\end{align}
	To control the relative entropy, we need the following lemma.
		\begin{lemma}\label{lmtg} Let $c>0$. Suppose $A\ge c,\, A+g\ge 0$, then
	\begin{align*}
	(A+g)^5 - A^5 - 5A^4g \ge (c^3|g|^2+c|g|^4).
	\end{align*}
	\end{lemma}
	\begin{proof}
		We can prove this lemma by direct calculations:
		\begin{align*}
		(A+g)^5&- A^5 - 5A^4g \\ 
		=& A^5 + 5 A^4g + 10 A^3 g^2+10A^2g^3+5Ag^4+g^5-A^5-5A^4g \\
		=&10 A^3g^2+10 A^2g^3+5Ag^4+g^5 \\
		\ge&  10A^3g^2+10 A^2g^3+5Ag^4-Ag^4 \\
		=&  10A^3 g^2+10 A^2g^3+4A g^4 \\
		=&A^3g^2 +\left(9A^3g^2 + 10 A^2g^3 + \frac{25}{9} Ag^4\right) + \frac{11}{9}A g^4 \\
		\ge& A^3g^2 + A^5\left(3\frac{g}{A} + \frac{5}{3}\frac{g^2}{A^2}\right)^2 + \frac{11}{9} Ag^4 \\
		\ge & c^3 g^2 + cg^4.% \ge \min(c^3,c)(g^2+g^4).
		\end{align*}	
	\end{proof}

	By applying Lemma \ref{lmtg} with $g:=T_\varepsilon-\overline{T}$ and $A=\overline{T}\ge c$, we have 
	\begin{align}\label{eq:Hes}
	H(T_\varepsilon,\psi_\varepsilon|\overline{T},\overline{\psi}) \ge& \int_\Omega 
	\frac{T_{\varepsilon}^5}{5} - \frac{\overline{T}^5}{5}-\overline{T}^4(T_\varepsilon-\overline{T}) dx 
	\ge C \int_\Omega (T_\varepsilon-\overline{T})^2 + (T_\varepsilon-\overline{T})^4 dx.
	\end{align} 

	Now we estimate the right hand side of inequality \eqref{relpf}. 

	We first consider $I_2$. Using the mean value theorem, we obtain
	\begin{align*}
		T_\varepsilon^{\frac{5}{2}} - \overline{T}^{\frac{5}{2}} - \frac{5}{2} \overline{T}^{\frac{3}{2}}(T_\varepsilon-\overline{T}) =& \frac{15}{4}\int_0^1\int_0^r \left(s(T_\varepsilon-\overline{T}) + \overline{T} \right)^{\frac{1}{2}} ds dr \cdot (T_\varepsilon - \overline{T})^2.\\
		\le& \frac{15}{4} \int_0^1\int_0^r ((s|T_\varepsilon-\overline{T}|)^{\frac{1}{2}} + \overline{T}^{\frac{1}{2}} )dsdr \cdot (T_\varepsilon - \overline{T})^2 \\
		\le & C |T_\varepsilon - \overline{T}|^2 + C|T_\varepsilon - \overline{T}|^4.
	\end{align*}
	Therefore, we have
	\begin{align}\label{I1}
		I_2 =& \frac{32}{25} \int_0^t\int_\Omega (T_\varepsilon^{\frac{5}{2}} - \overline{T}^{\frac{5}{2}} - \frac{5}{2} \overline{T}^{\frac{3}{2}}(T_\varepsilon - \overline{T}))\Delta \overline{T}^{\frac{5}{2}} dxd\tau\nonumber\\
		 \le& C \int_0^t \int_\Omega (T_\varepsilon - \overline{T})^2 + (T_\varepsilon - \overline{T})^4 dxd\tau.
	\end{align}
	Next we consider $I_3$.
	From the property that 
	\begin{align*}
		 T_\varepsilon^4 - \overline{T}^4-4\overline{T}^3(T_\varepsilon-\overline{T}) =& 6\overline{T}^2(T_\varepsilon-\overline{T})^2 + 4\overline{T}(T_\varepsilon-\overline{T})^3 + (T_\varepsilon-\overline{T})^4\\
	 \le& C(T_\varepsilon-\overline{T})^2 + C(T_\varepsilon-\overline{T})^4,
	\end{align*}
	we have
	\begin{align*}
		\int_0^t\int_\Omega(T_\varepsilon^4 - \overline{T}^4 - 4\overline{T}^3(T_\varepsilon-\overline{T}))\Delta \overline{T} dxd\tau \le C \int_0^t\int_\Omega (T_\varepsilon-\overline{T})^2 + (T_\varepsilon-\overline{T})^4 dxd\tau.
	\end{align*}
	So $I_3$ can be estimated as
	\begin{align}\label{I3}
	I_3=& \frac{1}{\varepsilon^2} \int_0^t\int_\Omega\int_{\mathbb{S}^2} \left(T_\varepsilon^4 - \overline{T}^4 - 4 \overline{T}^3(T_\varepsilon - \overline{T})\right) \left(\overline{\psi} - \overline{T}^4\right)  d\beta dx d\tau \nonumber\\
	=&\frac{1}{\varepsilon^2}\int_0^t\int_\Omega\int_{\mathbb{S}^2} (T_\varepsilon^4 - \overline{T}^4 - 4 \overline{T}^3(T_\varepsilon - \overline{T})	\nonumber\\
	&\qquad \cdot (\overline{T}^4 - \varepsilon \beta \cdot \nabla \overline{T}^4 -\varepsilon^2 \partial_t \overline{T}^4 + \varepsilon^2 \beta\cdot \nabla(\beta \cdot \nabla \overline{T}^4) - \overline{T}^4) d\beta dx  d\tau \nonumber\\
	=& \int_0^t\int_\Omega \left(T_\varepsilon^4 - \overline{T}^4 - 4 \overline{T}^3(T_\varepsilon - \overline{T})\right)(-4\pi \partial_t \overline{T}^4 + \frac{4}{3}\pi \Delta\overline{T}^4)  dxd\tau \nonumber\\
	\le& C(\|\partial_t \overline{T}\|_{L^\infty} + \|\Delta \overline{T}\|_{L^\infty})\int_0^t\int_\Omega\left(T_\varepsilon^4 - \overline{T}^4 - 4 \overline{T}^3(T_\varepsilon - \overline{T})\right) dxd\tau \nonumber\\
	\le& C\int_0^t \int_\Omega (6\overline{T}^2(T_\varepsilon-\overline{T})^2 + 4\overline{T}(T_\varepsilon-\overline{T})^3 + (T_\varepsilon-\overline{T})^4)dxd\tau \nonumber\\
	\le& C\int_0^t \int_\Omega (T_\varepsilon-\overline{T})^2 + (T_\varepsilon-\overline{T})^4 dxd\tau.
	\end{align}
	For $I_4$, we have
	\begin{align}\label{I2}
	I_4 =&\int_0^t\int_\Omega\int_{\mathbb{S}^2} (\psi_\varepsilon-\overline{\psi}) \overline{R} d\beta dxd\tau\nonumber \\
	\le& \int_0^t\int_\Omega\int_{\mathbb{S}^2} (\psi_\varepsilon-\overline{\psi})^2 d\beta dx d\tau +\int_0^t\int_\Omega\int_{\mathbb{S}^2}\overline{R}^2 d\beta dx d\tau\nonumber\\
	\le&\int_0^t\int_\Omega\int_{\mathbb{S}^2} (\psi_\varepsilon-\overline{\psi})^2 d\beta dxd\tau + C \varepsilon^2 .
	\end{align}
	Taking the above estimate and \eqref{I1}-\eqref{I2} into \eqref{relpf}, we get the estimate
	\begin{align*}
		\int_\Omega& (T_\varepsilon-\overline{T})^2 + (T_\varepsilon-\overline{T})^4 \bigg|_tdx + \iint_{\Omega\times\mathbb{S}^2} (\psi_\varepsilon-\overline{\psi})^2 \bigg|_t d\beta dx \\
		&\quad+ \frac{1}{\varepsilon^2} \int_0^t\iint_{\Omega\times\mathbb{S}^2} (\psi_\varepsilon - T_\varepsilon^4 - (\overline{\psi}-\overline{T}^4))^2 d\beta dxd\tau \\
		&\quad +\frac{16}{25} \int_0^t\int_\Omega \left(\nabla (T_\varepsilon)^{\frac{5}{2}} -\nabla (\overline{T})^{\frac{5}{2}} \right)^2 dxd\tau\\
		&\le C\int_0^t \int_\Omega (T_\varepsilon-\overline{T})^2 + (T_\varepsilon-\overline{T})^4dxd\tau + \int_0^t\iint_{\Omega\times\mathbb{S}^2} (\psi_\varepsilon-\overline{\psi})^2  d\beta dx d\tau +C\varepsilon^2.
	\end{align*}
	Applying Gronwall's lemma to the above inequality leads to \eqref{eq:Heps} and finishes the proof.
\end{proof}

\subsection{Dirichlet boundary conditions.}
In this case, we can do the similar calculations as before and use the boundary condition
\[T_\varepsilon =\overline{T}=T_b, \quad \text{for } x\in\partial \Omega,\]
to get the relative entropy inequality.
\begin{lemma}
	Assume  $T_\varepsilon$ is the weak solution of the system \eqref{eq:Teps}-\eqref{eq:psieps} with boundary conditions \eqref{bpsi} and  \eqref{b3},  $\overline{T}$ is a smooth solution of the equation \eqref{hgm1.0} with boundary condition $\overline{T} (t,x) =T_b $ for $x \in \partial \Omega$.  We have the following inequality:
	\begin{align}\label{eq:reformuladirichlet}
		H(T_\varepsilon&,\psi_\varepsilon | \overline{T},\overline{\psi}) \bigg|_{t}+ \frac{16}{25}\int_0^t\int_\Omega (\nabla T_\varepsilon^{\frac{5}{2}} - \nabla \overline{T}^{\frac{5}{2}})^2 dxd\tau \nonumber\\
		\le& H(T_\varepsilon,\psi_\varepsilon | \overline{T},\overline{\psi}) \bigg|_{0}   +\frac{32}{25} \int_0^t\int_\Omega (T_\varepsilon^{\frac{5}{2}} - \overline{T}^{\frac{5}{2}} - \frac{5}{2} \overline{T}^{\frac{3}{2}}(T_\varepsilon - \overline{T}))\Delta \overline{T}^{\frac{5}{2}} dxd\tau \nonumber\\
		&+\frac{1}{\varepsilon^2} \int_0^t\iint_{\Omega\times\mathbb{S}^2}\left(T_\varepsilon^4 - \overline{T}^4 - 4 \overline{T}^3(T_\varepsilon - \overline{T})\right) \left(\overline{\psi} - \overline{T}^4\right)  d\beta dxd\tau \nonumber\\
		& - \int_0^t\iint_{\Omega\times\mathbb{S}^2} (\psi_\varepsilon-\overline{\psi}) \overline{R} d\beta dxd\tau- \frac{1}{\varepsilon^2}\int_0^t\iint_{\Omega\times\mathbb{S}^2} (\psi_\varepsilon - T_\varepsilon^4 - (\overline{\psi}-\overline{T}^4))^2 d\beta dxd\tau \nonumber\\
		&-\frac{1}{2\varepsilon}\int_0^t\iint_{\Sigma_{+}} (\beta \cdot n)(\psi_\varepsilon- \overline{\psi})^2 d\sigma_xdxd\tau \nonumber\\
		&- \frac{1}{2\varepsilon} \int_0^t\iint_{\Sigma_{-}} (\beta \cdot n)(\alpha T_b^4 + (1-\alpha)\psi_\varepsilon' - \overline{\psi})^2 d\sigma_xdxd\tau.
	\end{align}
\end{lemma}
	Here, the above inequality does not include boundary terms of $T_\varepsilon$ since $T_\varepsilon-\overline{T}$ vanishes on the boundary. 
\begin{proof}
	We can slightly modify the proof of Theorem \ref{thmexistd} to show that the following energy inequality holds
		\begin{align}\label{eq:ene1b}
	H(T_\varepsilon&,\psi_\varepsilon) \bigg|_{0}^t + \frac{16}{25}\int_0^t \int_{\Omega} \left|\nabla (T_\varepsilon)^{\frac{5}{2}}\right|^2dxd\tau + \frac{1}{\varepsilon^2}\int_0^t\int_{\Omega} \int_{\mathbb{S}^{2}} (\psi_\varepsilon-(T_\varepsilon)^4  )^2 d \beta dxd\tau\nonumber\\
	&+ \int_0^t\int_{\partial \Omega} T_b^4 n\cdot \nabla T_\varepsilon d\sigma_x d\tau+ \frac{1}{2\varepsilon} \int_0^t\iint_{\Sigma_{+}}  (\beta \cdot n) \psi_\varepsilon^2 d\sigma_xdx d\tau\nonumber\\
	&+ \frac{1}{2\varepsilon} \int_0^t\iint_{\Sigma_{-}} (\beta \cdot n) (\alpha T_b^4 + (1-\alpha) \psi_\varepsilon')^2d\sigma_xdxd\tau\le 0.
	\end{align}

		Similarily the equations \eqref{eq_1} and \eqref{eq_2} also  satisfy
	\begin{align}\label{eq:ene2b}
	H(\overline{T}&,\overline{\psi}) \bigg|_{0}^t + \frac{16}{25}\int_0^t \int_{\Omega} \left|\nabla (\overline{T})^{\frac{5}{2}}\right|^2dxd\tau + \frac{1}{\varepsilon^2}\int_0^t\int_{\Omega} \int_{\mathbb{S}^{2}} (\overline{\psi}-\overline{T}^4  )^2 d \beta dxd\tau \nonumber\\
	&- \int_0^t \int_{\partial \Omega} T_b^4 n \cdot \nabla \overline{T} d\sigma_x d\tau + \frac{1}{2\varepsilon}\int_0^t \iint_{\Sigma} (\beta \cdot n) \overline{\psi}^2 d\sigma_xdx d\tau\nonumber\\
	=& \int_0^t\iint_{\Omega\times\mathbb{S}^2} \overline{\psi}\cdot \overline{R} d\beta dxd\tau.
	\end{align}

	We recall that from the definition of weak solutions \eqref{eq:weakd1}-\eqref{eq:weakd2}, the difference $T_\varepsilon-\overline{T},\psi_\varepsilon-\overline{\psi}$ satisfy
		\begin{align}
		-&\int_0^\infty\int_\Omega \varphi_t(T_\varepsilon-\overline{T})dxdt - \int_\Omega \varphi (T_\varepsilon-\overline{T}) \bigg|_{t=0}dx \nonumber\\
		&= \int_0^\infty \int_\Omega \Delta \varphi (T_\varepsilon - \overline{T}) dxdt + \int_0^\infty \int_{\partial \Omega} n\cdot \nabla T_\varepsilon \varphi d\sigma_x dt \nonumber\\
		&\quad- \int_0^\infty \int_{\partial \Omega} \varphi n \cdot \nabla \overline{T} d\sigma_x dt - \int_0^\infty \int_{\partial \Omega} (T_\varepsilon - \overline{T}) n \cdot \nabla \varphi dxdt \nonumber\\
		&\quad + \frac{1}{\varepsilon^2} \int_0^\infty \iint_{\Omega\times\mathbb{S}^2}(\psi_\varepsilon - \overline{\psi} - T_\varepsilon^4 + \overline{T}^4) \varphi d\beta dxdt, \label{eq:we1b} \\
		-&\int_0^\infty\iint_{\Omega\times\mathbb{S}^2} \rho_t (\psi_\varepsilon - \overline{\psi})d\beta dxdt - \int_\Omega\int_{\mathbb{S}^2} \rho(\psi_\varepsilon- \overline{\psi}) \bigg|_{t=0} d\beta dx\nonumber \\
		&\quad- \frac{1}{\varepsilon} \int_0^\infty\iint_{\Omega\times\mathbb{S}^2}(\psi_\varepsilon-\overline{\psi}) \beta \cdot \nabla \rho d\beta dxdt + \frac{1}{\varepsilon}\int_0^\infty\iint_{\Sigma_{+}} (\beta \cdot n) \psi_\varepsilon \rho d\sigma_xdxdt \nonumber \\
		&\quad+ \frac{1}{\varepsilon}\int_0^t\iint_{\Sigma_{-}} (\beta \cdot n) (\alpha T_b^4 + (1-\alpha) \psi_\varepsilon) \rho d\sigma_xdxdt - \frac{1}{\varepsilon} \int_0^t\iint_{\Sigma} (\beta \cdot n) \overline{\psi} \rho d\sigma_xdxdt \nonumber\\
		&= -\frac{1}{\varepsilon^2} \int_0^\infty\iint_{\Omega\times\mathbb{S}^2} (\psi_\varepsilon - \overline{\psi} - T_\varepsilon^4 + \overline{T}^4)\rho d\beta dxdt - \int_0^\infty\iint_{\Omega\times\mathbb{S}^2} \rho \overline{R} d\beta dxdt.\label{eq:we2b}
	\end{align}
	We choose the test function same as \eqref{eq:testfun} and let $\delta \to 0$. We will get
	\begin{align*}
		\int_{\Omega}& \overline{T}^4 (T_\varepsilon-\overline{T}) \bigg|_{\tau=0}^t dx  \\
		=& \int_0^t \int_\Omega 4\overline{T}^3 \Delta \overline{T} (T_\varepsilon-\overline{T}) dxd\tau 
		+ \int_0^t\int_\Omega \Delta \overline{T}^4(T_\varepsilon-\overline{T})dxd\tau \\
		& +\frac{1}{\varepsilon^2} \int_0^t\iint_{\Omega\times\mathbb{S}^2}4\overline{T}^3(\overline{\psi}-\overline{T}^4)(T_\varepsilon-\overline{T}) d\beta dxd\tau \\
		&+ \frac{1}{\varepsilon^2} \int_0^t \int_\Omega\int_{\mathbb{S}^2} \overline{T}^4(\psi_\varepsilon-\overline{\psi}-T_\varepsilon^4+\overline{T}^4) d\beta dxd\tau \\
		&+ \int_0^t \int_{\partial \Omega} n\cdot \nabla T_\varepsilon\overline{T}^4 d\sigma_x d\tau - \int_0^t \int_{\partial \Omega} T_b^4 n \cdot \nabla \overline{T} d\sigma_x d\tau \nonumber\\
		&- \int_0^t \int_{\partial \Omega} (T_\varepsilon - \overline{T}) n \cdot \nabla \overline{T}^4 dxd\tau , \\
		\iint_{\Omega\times\mathbb{S}^2} &\overline{\psi} (\psi_\varepsilon-\overline{\psi}) \bigg|_{\tau=0}^t d\beta dx \\
		= & 
		 - \frac{1}{\varepsilon}\int_0^t\iint_{\Sigma_{+}} (\beta \cdot n) \psi_\varepsilon \overline{\psi} d\sigma_xdxd\tau \nonumber\\
		 &- \frac{1}{\varepsilon}\int_0^t\iint_{\Sigma_{-}} (\beta \cdot n) (\alpha T_b^4 + (1-\alpha) \psi_\varepsilon) \overline{\psi} d\sigma_xdxd\tau \\
		  &+ \frac{1}{\varepsilon} \int_0^t\iint_{\Sigma} (\beta \cdot n) \overline{\psi} \cdot\overline{\psi} d\sigma_xdxd\tau -\frac{1}{\varepsilon^2} \int_0^t\int_\Omega\int_{\mathbb{S}^2}(\overline{\psi}-\overline{T}^4)(\psi_\varepsilon-\overline{\psi})d\beta dxd\tau \\&+ \int_0^t \int_\Omega\int_{\mathbb{S}^2} \overline{R} (\psi_\varepsilon-\overline{\psi}) d\beta dxd\tau - \frac{1}{\varepsilon^2}\int_0^t\int_\Omega\int_{\mathbb{S}^2} \overline{\psi}(\psi_\varepsilon-\overline{\psi}-T_\varepsilon^4+\overline{T}^4) d\beta dxd\tau \\&- \int_0^t \int_\Omega\int_{\mathbb{S}^2} \overline{\psi} \cdot \overline{R} d\beta dxd\tau.
	\end{align*}
	We substract the summation of the above two equations from the difference of equations \eqref{eq:ene1b} and \eqref{eq:ene2b} to get
	\begin{align}\label{eq:Hincalb}
	 	H(T_\varepsilon&,\psi_\varepsilon | \overline{T},\overline{\psi}) \bigg|_{t}\\
	 	 \le& H(T_\varepsilon,\psi_\varepsilon | \overline{T},\overline{\psi}) \bigg|_{0}  - \frac{16}{25} \int_0^t\int_\Omega \left(\left|\nabla(T_\varepsilon)^{\frac{5}{2}}\right|^2 - \left|\nabla(\overline{T})^{\frac{5}{2}}\right|^2\right) dxd\tau \nonumber\\
		&-\int_0^t \int_\Omega 4\overline{T}^3 \Delta \overline{T} (T_\varepsilon-\overline{T}) + \Delta \overline{T}^4(T_\varepsilon-\overline{T})dxd\tau \nonumber\\
		& - \int_0^t\iint_{\Omega\times\mathbb{S}^2} (\psi_\varepsilon-\overline{\psi}) \overline{R} d\beta dxd\tau- \frac{1}{\varepsilon^2}\int_0^t\iint_{\Omega\times\mathbb{S}^2} (\psi_\varepsilon - T_\varepsilon^4 - (\overline{\psi}-\overline{T}^4))^2 d\beta dxd\tau \nonumber\\
		&+\frac{1}{\varepsilon^2} \int_0^t\iint_{\Omega\times\mathbb{S}^2}\left(T_\varepsilon^4 - \overline{T}^4 - 4 \overline{T}^3(T_\varepsilon - \overline{T})\right) \left(\overline{\psi} - \overline{T}^4\right)  d\beta dxd\tau \nonumber \\
		&  + \int_0^t \int_{\partial \Omega} (T_\varepsilon - \overline{T}) n\cdot \nabla \overline{T}^4 d\sigma_xd\tau \nonumber\\
		&-\frac{1}{2\varepsilon}\int_0^t\iint_{\Sigma_{+}} (\beta \cdot n)(\psi_\varepsilon- \overline{\psi})^2 d\sigma_xdxd\tau \nonumber\\&- \frac{1}{2\varepsilon} \int_0^t\iint_{\Sigma_{-}}(\beta \cdot n) (\alpha T_b^4 + (1-\alpha)\psi_\varepsilon' - \overline{\psi})^2 d\sigma_xdxd\tau.
	 \end{align}
By considering the boundary conditions, the equation \eqref{eq:4t3d} becomes 
	 \begin{align*}
        -\int_0^t& \int_\Omega 4\overline{T}^3 \Delta \overline{T} (T_\varepsilon-\overline{T}) + \Delta \overline{T}^4(T_\varepsilon-\overline{T})dxd\tau \nonumber\\
        =& \int_0^t\int_\Omega (T_\varepsilon^4 - \overline{T}^4 - 4 \overline{T}^3(T_\varepsilon - \overline{T}))\Delta \overline{T} dxd\tau  +  \int_0^t\int_\Omega (\overline{T}^4 \Delta \overline{T} + \overline{T} \Delta \overline{T}^4)dxd\tau \\
        &- \int_0^t \int_\Omega( T_\varepsilon^4 \Delta \overline{T} + T_\varepsilon \Delta \overline{T}^4) dxd\tau.
    \end{align*}
    The last term is
	\begin{align*}
		-\int_{0}^t& \int_\Omega (T_\varepsilon^4 \Delta \overline{T} + T_\varepsilon \Delta \overline{T}^4) dxd\tau \\
		=& -\int_0^t\int_\Omega (T_\varepsilon^4 - \overline{T}^4 - 4\overline{T}^3(T_\varepsilon - \overline{T}))\Delta \overline{T} dxd\tau \\
		& +\frac{32}{25} \int_0^t\int_\Omega (T_\varepsilon^{\frac{5}{2}} - \overline{T}^{\frac{5}{2}} - \frac{5}{2} \overline{T}^{\frac{3}{2}}(T_\varepsilon - \overline{T}))\Delta \overline{T}^{\frac{5}{2}} dxd\tau -\frac{32}{25}\int_0^t\int_\Omega T_\varepsilon^{\frac{5}{2}} \Delta \overline{T}^{\frac{5}{2}} dxd\tau \\
		& + 3\int_0^t\int_\Omega \overline{T}^4\Delta \overline{T} dxd\tau - \frac{48}{25}\int_0^t\int_\Omega \overline{T}^{\frac{5}{2}} \Delta \overline{T}^{\frac{5}{2}} dxd\tau.
	\end{align*}
	Adding the above two equations and using the integration-by-parts formulas
	\begin{align*}
		\int_\Omega \overline{T}^4 \Delta \overline{T} dx 
		=& - \frac{16}{25}\int_\Omega \left|\nabla \overline{T}^{\frac{5}{2}}\right|^2 dx + \int_{\partial \Omega} \overline{T}^4 n \cdot \nabla \overline{T} d\sigma_x, \\
		\int_\Omega \overline{T} \Delta \overline{T}^4 =& - \frac{16}{25}\int_\Omega \left|\nabla \overline{T}^{\frac{5}{2}}\right|^2 dx + \int_{\partial \Omega} \overline{T} n \cdot \nabla \overline{T}^4 d\sigma_x, \\
		\int_{\Omega} \overline{T}^{\frac{5}{2}}\Delta \overline{T}^{\frac{5}{2}} dx =& -\int_\Omega \left|\nabla \overline{T}^{\frac{5}{2}}\right|^2 dx + \frac{5}{2} \int_{\partial \Omega} \overline{T}^4 n \cdot \nabla \overline{T} dx, \\
		\int_{\Omega} {T}_\varepsilon^{\frac{5}{2}}\Delta \overline{T}^{\frac{5}{2}} dx =& -\int_\Omega \nabla {T}_\varepsilon^{\frac{5}{2}} \cdot \nabla \overline{T}^{\frac{5}{2}} dx + \frac{5}{2} \int_{\partial \Omega} T_\varepsilon^{\frac{5}{2}} \overline{T}^{\frac{3}{2}} n \cdot \nabla \overline{T} dx,
	\end{align*}
	we obtain
	\begin{align}\label{eq:4t3DeltaT}
		-\int_0^t& \int_\Omega 4\overline{T}^3 \Delta \overline{T} (T_\varepsilon-\overline{T}) + \Delta \overline{T}^4(T_\varepsilon-\overline{T})dxd\tau \nonumber\\
		=& \frac{32}{25} \int_0^t\int_\Omega (T_\varepsilon^{\frac{5}{2}} - \overline{T}^{\frac{5}{2}} - \frac{5}{2} \overline{T}^{\frac{3}{2}}(T_\varepsilon - \overline{T}))\Delta \overline{T}^{\frac{5}{2}} dxd\tau \\
		&- \frac{32}{25} \int_0^t\int_{\Omega} \left|\nabla \overline{T}^{\frac{5}{2}}\right|^2 dxd\tau + \frac{32}{25} \int_0^t\int_{\Omega} \nabla {T}_\varepsilon^{\frac{5}{2}} \cdot \nabla \overline{T}^{\frac{5}{2}} dxd\tau \nonumber\\
		&+\frac{16}{5}\int_0^t\int_{\partial \Omega} \overline{T}^4 n\cdot \nabla \overline{T} d\sigma_xd\tau - \frac{16}{5} \int_0^t\int_{\partial \Omega} T_\varepsilon^{\frac{5}{2}} \overline{T}^{\frac{3}{2}} n\cdot \nabla \overline{T} d\sigma_x d\tau \nonumber\\
		=& \frac{32}{25} \int_0^t\int_\Omega (T_\varepsilon^{\frac{5}{2}} - \overline{T}^{\frac{5}{2}} - \frac{5}{2} \overline{T}^{\frac{3}{2}}(T_\varepsilon - \overline{T}))\Delta \overline{T}^{\frac{5}{2}} dxd\tau - \frac{32}{25} \int_0^t\int_{\Omega} \left|\nabla \overline{T}^{\frac{5}{2}}\right|^2 dxd\tau \nonumber\\
		&+ \frac{32}{25} \int_0^t\int_{\Omega} \nabla {T}_\varepsilon^{\frac{5}{2}} \cdot \nabla \overline{T}^{\frac{5}{2}} dxd\tau \nonumber\\
		& -\frac{16}{5}\int_0^t\int_{\partial \Omega} \overline{T}^{\frac{3}{2}}(T_\varepsilon^{\frac{5}{2}} - \overline{T}^{\frac{5}{2}} - \frac{5}{2} \overline{T}^{\frac{3}{2}}(T_\varepsilon - \overline{T})) n\cdot \nabla \overline{T} d\sigma_xd\tau \nonumber \\&- 2 \int_0^t\int_{\partial \Omega} (T_\varepsilon - \overline{T}) n\cdot \nabla \overline{T}^4d\sigma_xd\tau\nonumber\\
		=&  \frac{32}{25} \int_0^t\int_\Omega (T_\varepsilon^{\frac{5}{2}} - \overline{T}^{\frac{5}{2}} - \frac{5}{2} \overline{T}^{\frac{3}{2}}(T_\varepsilon - \overline{T}))\Delta \overline{T}^{\frac{5}{2}} dxd\tau - \frac{32}{25} \int_0^t\int_{\Omega} \left|\nabla \overline{T}^{\frac{5}{2}}\right|^2 dxd\tau \nonumber\\
		&+ \frac{32}{25} \int_0^t\int_{\Omega} \nabla {T}_\varepsilon^{\frac{5}{2}} \cdot \nabla \overline{T}^{\frac{5}{2}} dxd\tau.
	\end{align}
	Taking the above equation into \eqref{eq:Hincal} will lead to the inequality  and finishes the proof.
\end{proof}

We now proceed to prove Theorem \ref{thmre}.
\begin{proof}[Proof of Theorem \ref{thmre}.]
	Notice that the relative entropy formula \eqref{eq:reformuladirichlet} only differs from \eqref{eq:Hevol} of the torus case by the last two boundary terms on the right hand side of \eqref{eq:reformuladirichlet}. To control these two terms, we recall
\[\overline{\psi}|_{\partial \Omega} =T_b^4-\varepsilon\beta \cdot \nabla \overline{T}^4-\varepsilon^2\partial_t\overline{T}^4 + \varepsilon^2 \beta \cdot \nabla(\beta \cdot \nabla \overline{T}^4)=T_b^4+\varepsilon\overline{R}_b,\]
with $\overline{R}_b = -\beta \cdot \nabla \overline{T}^4-\varepsilon\partial_t\overline{T}^4 + \varepsilon \beta \cdot \nabla(\beta \cdot \nabla \overline{T}^4)$ bounded. So we have
\begin{align*}
	-\frac{1}{2\varepsilon}&\int_0^t\iint_{\Sigma_{+}} (\beta \cdot n)(\psi_\varepsilon- \overline{\psi})^2 d\sigma_xdxd\tau \\
	=& -\frac{1}{2\varepsilon}\int_0^t\iint_{\Sigma_{+}} (\beta \cdot n)(\psi_\varepsilon- T_b^4 - \varepsilon \overline{R}_b)^2 d\sigma_xdxd\tau \\
	=& - \frac{1}{2\varepsilon} \int_0^t\iint_{\Sigma_{+}} (\beta \cdot n)(\psi_\varepsilon - T_b^4)^2 d\sigma_xdxd\tau +\int_0^t\iint_{\Sigma_{+}} (\beta \cdot n) (\psi_\varepsilon - T_b^4)\overline{R}_b d\sigma_xdxd\tau \\
	&- \frac{\varepsilon}{2}  \int_0^t\iint_{\Sigma_{+}} \overline{R}_b^2 d\sigma_xdxd\tau \\
\end{align*}
From the coordinate transform $\beta' = \beta - 2n(n\cdot \beta)$, we get $n \cdot \beta = -n\cdot \beta'$. We can also get $\beta = \beta' + 2n(n \cdot \beta')$ and $\overline{R}_b$ can be also expressed using $\beta'$. We denote $\overline{R}_b'[\beta'] = \overline{R}_b[\beta]$. We have
\begin{align*}
	- \frac{1}{2\varepsilon}& \int_0^t\iint_{\Sigma_{-}}(\beta \cdot n) (\alpha T_b^4 + (1-\alpha)\psi_\varepsilon' - \overline{\psi})^2 d\sigma_xdxd\tau \\
	=& - \frac{1}{2\varepsilon} \int_0^t\iint_{\Sigma_{-}}(\beta \cdot n) (\alpha T_b^4 + (1-\alpha)\psi_\varepsilon' - T_b^4 - \varepsilon \overline{R}_b)^2 d\sigma_xdxd\tau\\
	=&  \frac{1}{2\varepsilon} \int_0^t\iint_{\Sigma_{+}}(\beta' \cdot n) ( (1-\alpha)(\psi_\varepsilon' - T_b^4) - \varepsilon \overline{R}_b')^2 dS_{\beta',x}d\tau \\
	=& \frac{(1-\alpha)^2}{2\varepsilon}\int_0^t\iint_{\Sigma_{+}} (\beta \cdot n) (\psi_\varepsilon - T_b^4)^2 d\sigma_xdx d\tau \\
	&-  \int_0^t\iint_{\Sigma_{+}}(\beta\cdot n)(\psi_\varepsilon - T_b^4)  \overline{R}_b' d\sigma_xdx d\tau+\frac{\varepsilon}{2} \int_0^t\iint_{\Sigma_{+}} (\overline{R}_b')^2 d\sigma_xdxd\tau.
\end{align*}
Adding the above two equations together, we have
\begin{align}\label{eq:I3b}
    -\frac{1}{2\varepsilon}&\int_0^t\iint_{\Sigma_{+}} (\beta \cdot n)(\psi_\varepsilon- \overline{\psi})^2 d\sigma_xdxd\tau \nonumber\\&\quad- \frac{1}{2\varepsilon} \int_0^t\iint_{\Sigma_{-}}(\beta \cdot n) (\alpha T_b^4 + (1-\alpha)\psi_\varepsilon' - \overline{\psi})^2 d\sigma_xdxd\tau \nonumber\\
    \le& -\frac{(2\alpha - \alpha^2)}{2\varepsilon} \int_0^t\iint_{\Sigma_{+}} (\beta \cdot n) (\psi_\varepsilon - T_b^4)^2 d\sigma_xdx d\tau \nonumber\\
    &+ \int_0^t\iint_{\Sigma_{+}} (\beta \cdot n) (\psi_\varepsilon - T_b^4)(\overline{R}_b - \overline{R}_b') d\sigma_xdx d\tau + C\varepsilon \nonumber\\
    \le& -\frac{(2\alpha - \alpha^2)}{2\varepsilon} \int_0^t\iint_{\Sigma_{+}} (\beta \cdot n) (\psi_\varepsilon - T_b^4)^2 d\sigma_xdx d\tau \nonumber\\
    &+ \frac{(2\alpha-\alpha^2)}{4\varepsilon} \int_0^t\iint_{\Sigma_{+}} (\beta \cdot n) (\psi_\varepsilon - T_b^4)^2 d\sigma_xdx d\tau \nonumber\\
    &+ \frac{C\varepsilon}{2\alpha-\alpha^2} \int_0^t\iint_{\Sigma_{+}} (\overline{R}_b-\overline{R}_b')^2 d\sigma_xdxd\tau + C\varepsilon \nonumber\\
    \le& -\frac{(2\alpha - \alpha^2)}{4\varepsilon} \int_0^t\iint_{\Sigma_{+}} (\beta \cdot n) (\psi_\varepsilon - T_b^4)^2 d\sigma_xdx d\tau + C\varepsilon.
\end{align}
Combining the above estimates with the estimates of other terms in proof of the torus case and applying Gronwall's inequality lead to \eqref{eq:Heps} with $s=1$ and finishes  the proof.
\end{proof}

\subsection{Robin boundary condition with $r>0$}
For the case of Robin boundary condition \eqref{b2}, a similar relative entropy inequality can also be derived like for the Dirichlet case. The result is
% We now proceed to consider the case with Robin boundary \eqref{b2}. First, we will establish the relative entropy inequality in the case with Robin conductive boundary condition \eqref{b2} and specular radiative condition \eqref{bpsi}. 
% \begin{lemma}
% 	Assume  $T_\varepsilon$ is the weak solution to the equations \eqref{eq:Teps}-\eqref{eq:psieps} with boundary conditions \eqref{b1} and \eqref{bpsi},  $\overline{T}$ is a smooth solution to equation \eqref{hgm1.0} with boundary condition $\overline{T} (t,x) =T_b $ for $x \in \partial \Omega$.  We have the following inequality:
	\begin{align}\label{eq:Hevolb}
		H(T_\varepsilon&,\psi_\varepsilon | \overline{T},\overline{\psi}) \bigg|_{t}+ \frac{16}{25}\int_0^t\int_\Omega (\nabla T_\varepsilon^{\frac{5}{2}} - \nabla \overline{T}^{\frac{5}{2}})^2 dxd\tau \nonumber\\
		\le& H(T_\varepsilon,\psi_\varepsilon | \overline{T},\overline{\psi}) \bigg|_{0}   +\frac{32}{25} \int_0^t\int_\Omega (T_\varepsilon^{\frac{5}{2}} - \overline{T}^{\frac{5}{2}} - \frac{5}{2} \overline{T}^{\frac{3}{2}}(T_\varepsilon - \overline{T}))\Delta \overline{T}^{\frac{5}{2}} dxd\tau \nonumber\\
		&+\frac{1}{\varepsilon^2} \int_0^t\iint_{\Omega\times\mathbb{S}^2}\left(T_\varepsilon^4 - \overline{T}^4 - 4 \overline{T}^3(T_\varepsilon - \overline{T})\right) \left(\overline{\psi} - \overline{T}^4\right)  d\beta dxd\tau \nonumber\\
		& - \int_0^t\iint_{\Omega\times\mathbb{S}^2} (\psi_\varepsilon-\overline{\psi}) \overline{R} d\beta dxd\tau- \frac{1}{\varepsilon^2}\int_0^t\iint_{\Omega\times\mathbb{S}^2} (\psi_\varepsilon - T_\varepsilon^4 - (\overline{\psi}-\overline{T}^4))^2 d\beta dxd\tau \nonumber\\
		& + \int_0^t\int_{\partial \Omega} (T_\varepsilon^4 - \overline{T}^4) \frac{T_b - T_\varepsilon}{ \varepsilon^r} d\sigma_xd\tau - \int_0^t \int_{\partial \Omega} (T_\varepsilon - \overline{T}) n\cdot \nabla \overline{T}^4 d\sigma_xd\tau \nonumber\\
		&-\frac{16}{5}\int_0^t\int_{\partial \Omega} \overline{T}^{\frac{3}{2}}(T_\varepsilon^{\frac{5}{2}} - \overline{T}^{\frac{5}{2}} - \frac{5}{2} \overline{T}^{\frac{3}{2}}(T_\varepsilon - \overline{T})) n\cdot \nabla \overline{T} d\sigma_xd\tau \nonumber \\
		&-\frac{1}{2\varepsilon}\int_0^t\iint_{\Sigma_{+}} (\beta \cdot n)(\psi_\varepsilon- \overline{\psi})^2 d\sigma_xdxd\tau \nonumber\\
		&- \frac{1}{2\varepsilon} \int_0^t\iint_{\Sigma_{-}} (\beta \cdot n)(\alpha T_b^4 + (1-\alpha)\psi_\varepsilon' - \overline{\psi})^2 d\sigma_xdxd\tau .
	\end{align}
With the above relative entropy inequality we now proceed to prove Theorem \ref{thmre}.

\begin{proof}[Proof of Theorem \ref{thmre}]
Compare the relative entropy inequality \eqref{eq:Hevolb} with \eqref{eq:reformuladirichlet} of the Dirichlet case, there are two additional terms with the boundary of $T_\varepsilon$ that need to control. First, by \eqref{EstiL5} we have
 % \eqref{eq:Hevol}, there are four additional boundary terms on the right hand side. We now estimate one by one. First we consider 
\begin{align}\label{eq:I1b}
	\int_0^t&\int_{\partial \Omega} (T_\varepsilon^4 - \overline{T}^4) \frac{T_b - T_\varepsilon}{ \varepsilon^r} d\sigma_x d\tau \nonumber\\
	 =& - \int_0^t\int_{\partial \Omega} (T_\varepsilon^4 - T_b^4) \frac{T_b - T_\varepsilon}{ \varepsilon^r} d\sigma_x d\tau \nonumber\\
	=& -\frac{1}{\varepsilon^r}\int_0^t\int_{\partial \Omega} (T_\varepsilon - T_b)^2 (T_\varepsilon + T_b)(T_\varepsilon^2 + T_b^2) d\sigma_x d\tau \nonumber\\
	\le& -\frac{1}{ \varepsilon^r} \int_0^t\int_{\partial \Omega} |T_\varepsilon - T_b|^5 d\sigma_x d\tau,
\end{align}
%where we use the assumption that $T_\varepsilon>0$ and $T_b\ge c>0$.
The second bounary term is
\begin{align}\label{eq:I2b}
	\int_0^t& \int_{\partial \Omega} (T_\varepsilon - \overline{T}) n\cdot \nabla \overline{T}^4 d\sigma_xd\tau \nonumber\\=& \int_0^t \int_{\partial \Omega} (T_\varepsilon - T_b) n\cdot \nabla \overline{T}^4 d\sigma_xd\tau \nonumber\\
	\le& \frac{1}{2 \varepsilon^r}  \int_0^t\int_{\partial \Omega} |T_\varepsilon - T_b|^5 d\sigma_x d\tau + C\varepsilon^r \int_0^t\int_{\partial \Omega} |n\cdot \nabla \overline{T}^4|^2 d\sigma_xd\tau \nonumber \\
	\le& \frac{1}{2 \varepsilon^r}  \int_0^t\int_{\partial \Omega} |T_\varepsilon - T_b|^5 d\sigma_x d\tau + C\varepsilon^r .
\end{align}
Combining \eqref{eq:I1b}, \eqref{eq:I2b}, and result of the Dirichlet case, we can conclude that the following inequality holds:
	\begin{align*}
		\int_\Omega& (T_\varepsilon-\overline{T})^2 + (T_\varepsilon-\overline{T})^4 \bigg|_tdx + \iint_{\Omega\times\mathbb{S}^2} (\psi_\varepsilon-\overline{\psi})^2 \bigg|_t d\beta dx \\
		&\quad + \frac{1}{\varepsilon^2} \int_0^t\iint_{\Omega\times\mathbb{S}^2} (\psi_\varepsilon - T_\varepsilon^4 - (\overline{\psi}-\overline{T}^4))^2 d\beta dxd\tau \\&\quad + \int_0^t\int_\Omega \left(\nabla (T_\varepsilon)^{\frac{5}{2}} -\nabla (\overline{T})^{\frac{5}{2}} \right)^2 dxd\tau + \frac{1}{2\varepsilon^r} \int_0^t\int_{\partial \Omega} |T_\varepsilon -T_b|^5 d\sigma_x d\tau \\
		&\quad+ \frac{(2\alpha - \alpha^2)}{\varepsilon} \int_0^t\iint_{\Sigma_{+}} (\beta \cdot n) (\psi_\varepsilon - T_b^4)^2 d\sigma_xdx d\tau\\
		&\le C\int_0^t \int_\Omega (T_\varepsilon-\overline{T})^2 + (T_\varepsilon-\overline{T})^4dxd\tau + \int_0^t\iint_{\Omega\times\mathbb{S}^2} (\psi_\varepsilon-\overline{\psi})^2  d\beta dx d\tau +C\varepsilon + C\varepsilon^r.
	\end{align*}
	Applying Gronwall's inequality leads to \eqref{eq:Heps} and finishes the proof.	
\end{proof}

For the case of Robin boundary condition with $r=0$, a boundary layer exists for $T_\varepsilon$, thus we can not apply the above relative entropy method directly to show the convergence of $T_\varepsilon$, although with the compactness method this is done in Theorem \ref{LimitProof}. This boundary layer problem will be investigated in our future paper.
% \subsection{Robin boundary condition with $r=0$.} 
% {\color{red} We have difficulties to prove Theorem \ref{thmre} for this case. Although \eqref{eq:I1b} still holds, the other term \eqref{eq:I2b} seems unable to handle.

% \begin{align}
% 	\int_0^t& \int_{\partial \Omega} (T_\varepsilon - \overline{T}) n\cdot \nabla \overline{T}^4 d\sigma_xd\tau \nonumber\\=& \int_0^t \int_{\partial \Omega} (T_\varepsilon - T_b) n\cdot \nabla \overline{T}^4 d\sigma_xd\tau \nonumber\\
% 	\le& c^3  \int_0^t\int_{\partial \Omega} (T_\varepsilon - T_b)^2 d\sigma_x d\tau + C \int_0^t\int_{\partial \Omega} |n\cdot \nabla \overline{T}^4|^2 d\sigma_xd\tau \nonumber \\
% 	\le& c^3 \int_0^t\int_{\partial \Omega} (T_\varepsilon - T_b)^2 d\sigma_x d\tau + C .
% \end{align}
% }
% {\color{red} We have difficulties to control the boundary terms
% \begin{align*}
% 	\int_0^t&\int_{\partial\Omega} (T_\varepsilon^4-\overline{T}^4)(T_b-T_\varepsilon) d\sigma_x d\tau - \int_0^t\int_{\partial\Omega} (T_\varepsilon - \overline{T}) n\cdot \nabla \overline{T}^4 d\sigma_x d\tau \\
% 	=&-\int_0^t\int_{\partial\Omega} (T_\varepsilon^4-\overline{T}^4)(T_\varepsilon-\overline{T}) d\sigma_x d\tau + \int_0^t\int_{\partial\Omega} (T_\varepsilon^4-\overline{T}^4-4\overline{T}^3(T_\varepsilon-\overline{T})) (T_b-\overline{T}) d\sigma_x d\tau \\
% 	\le & -c\int_0^t\int_{\partial\Omega} (T_\varepsilon-\overline{T})^2 d\sigma_x d\tau + C\int_0^t\int_{\partial\Omega}(T_\varepsilon - \overline{T})^2 d\sigma_x d\tau,
% \end{align*}
% but it is not known $c<C$.}
\section*{Acknowledgements} The work of M.G and N. M is supported by Tamkeen under the NYU Abu Dhabi Research Institute grant of the center SITE.

\bibliographystyle{siam}
\bibliography{ReferenceHGM}

\appendix

\section{Positivity of the solution of system \eqref{eq:Teps}-\eqref{eq:psieps}}\label{Postivsolution} 

{\color{black}
This section is devoted to the proof of Lemma \ref{lm:uniformbd} such that the non-negativity of the solution for \eqref{eq:Teps}-\eqref{eq:psieps}, associated to the radiative boundary condition \eqref{bpsi} and for a general boundary conditions
	 \begin{align}
	 	a\varepsilon^r n \cdot \nabla T_\varepsilon(t,x) = -bT_\varepsilon(t,x) + bT_b(t,x), \text{ for any } x\in \partial \Omega. \label{b2g}
	 \end{align}
where $a$ and $b$ are real numbers satisfying the condition of $|a|+|b|>0$.  We add the reals number a and b to \eqref{b2} in order to cover different boundary conditions. 

The non-negativity of the solutions to the approximate Galerkin approximate systems that constructed in section 2 can be proved in a same way.

\begin{proof}[Proof of Lemma \ref{lm:uniformbd}]
	For a nonnegative initial and boundary radiative data $\psi_{\varepsilon0}\ge 0,\psi_{b}\ge 0$, and due to $T_\varepsilon^4\ge 0$, following from the maximum principle for linear transport equations, the  solution $\psi_{\varepsilon}$ of the radiative transfer radiative \eqref{eq:psieps} is nonnegative. 

	Next we show $T_\varepsilon\ge 0$. We define $F$ on $\mathbb{R}_{+}\times\Omega \times \mathbb{R}$ by 
	\begin{equation}%\label{F}
		F(t,x,y)=\frac{1}{\varepsilon^2}\langle \psi_\varepsilon(t,x,\beta) - y^4(t,x)\rangle.
		\end{equation}
		The system \eqref{eq:Teps}, \eqref{b2g} can be rewritten as
		\begin{align}
			&\partial_t T_\varepsilon = \Delta T_\varepsilon +F(t,x,T_\varepsilon ) , \text{ for any } t>0\,\,\, x\in \Omega
			 % \frac{1}{\varepsilon^2}\int_{\mathbb{S}^2}(\psi_\varepsilon - T_\varepsilon^4) d\beta
			 , \label{eq:TepsS} \\
		&a\varepsilon^r n \cdot \nabla T_\varepsilon(t,x) = -bT_\varepsilon(t,x) + bT_b(t,x), \text{ for any } x\in \partial \Omega, \label{b2gx}\\
		&	T_\varepsilon(t=0,x) = T_{\varepsilon 0}(x), ~~ \text{ for any } x\in\Omega. \label{eq:ic1c}
		\end{align}
		Now, we define $\bar{F}$ in $\mathbb{R}_{+}\times\Omega \times \mathbb{R}$ by
		\begin{align}\label{eq:FbarAp}
		\bar{F}(t, x,y)=\left \{\begin{array}{lll}
		\frac{1}{\varepsilon^2}\langle \psi_\varepsilon(t,x,\beta) - y^4(t,x)\rangle,&{\rm if}& y\geq 0,\\
		\frac{1}{\varepsilon^2}\langle \psi_\varepsilon(t,x,\beta)\rangle, &{\rm if}& y<0.
		\end{array}
		\right.
		\end{align}
		Let us consider  $\bar{T}_\varepsilon$ the solution of the following system
		\begin{equation}\label{eq2S}
		\begin{aligned}
		&\partial_t \bar{T}_\varepsilon = ~ \Delta \bar{T}_\varepsilon +F(t,x,\bar{T}_\varepsilon ) , \text{ for any } t>0\,\,\, x\in \Omega, \\
		&a\varepsilon^r n \cdot \nabla \bar{T}_\varepsilon(t,x) = -b\bar{T}_\varepsilon(t,x) + bT_b(t,x), \text{ for any } x\in \partial \Omega,\\
		&\bar{T}_\varepsilon(t=0,x) = ~ \bar{T}_{\varepsilon 0}(x), ~~ \text{ for any } x\in\Omega.
		\end{aligned}
		\end{equation}
		 The objective is to show  that the solution $\bar{T}_\varepsilon$ of this equation remains nonnegative over the time. Indeed, in this case  $\bar{F}$ and $F$ coincide, therefore we have by the uniqueness of the solution $T_\varepsilon = \bar{T}_\varepsilon $ which is  nonnegative.
		
		 We set $\bar{T}_\varepsilon^{+}=\max(T_\varepsilon,0)$ and $\bar{T}_\varepsilon^{-}=\max(-T_\varepsilon,0)$, such that $\bar{T}_\varepsilon=\bar{T}_\varepsilon^{+}-\bar{T}_\varepsilon^{-}$. Multiplying  the equation \eqref{eq2S} by $(-\bar{T}_{\varepsilon}^{-})$ and integrating over $\Omega$,  we obtain
		\begin{equation*}
		-\int_{\Omega}\partial_{t}\bar{T}_\varepsilon(t, x)\bar{T}_\varepsilon^{-}(t, x)dx+\int_{\Omega}\Delta \bar{T}_\varepsilon(t, x)\bar{T}_\varepsilon^{-}(t, x)dx= -\int_{\Omega}\bar{F}(t, x,\bar{T}_\varepsilon) \bar{T}_\varepsilon^{-}(t, x)dx.
		\end{equation*}
		Now, we have
		\begin{equation}\label{derivetemp}
		-\int_{\Omega}\partial_{t}\bar{T}_\varepsilon(t, x)\bar{T}_\varepsilon^{-}(t, x)dx= \frac{1}{2}\partial_{t}\int_{\Omega}(\bar{T}_\varepsilon^{-}(t, x))^{2}dx,
		\end{equation}
		\begin{equation}\label{F}
		\begin{aligned}
		-\int_{\Omega}\bar{F}(t, x,\bar{T}_\varepsilon) \bar{T}_\varepsilon^{-}(t, x)dx&=-\int_{\{\bar{T}_\varepsilon<0\}}\bar{F}(t, x,\bar{T}_\varepsilon) \bar{T}_\varepsilon^{-}(t, x)dx\\&=-\int_{\{\bar{T}_\varepsilon<0\}}\frac{1}{\varepsilon^2}\langle \psi_\varepsilon(t,x,\beta)\rangle \bar{T}_\varepsilon^{-}(t, x)dx\leq 0,
		\end{aligned}
		\end{equation}
		and 
		\begin{equation*}
		\begin{aligned}
		\int_{\Omega}\Delta \bar{T}_\varepsilon(t, x)\bar{T}_\varepsilon^{-}(t, x)dx=&\int_{\Omega}(\nabla \bar{T}_\varepsilon^{-}(t, x))^{2}dx+\int_{\partial\Omega}n \cdot \nabla\bar{T}_\varepsilon(t, x)\bar{T}_\varepsilon^{-}(t, x)d\sigma_{x}.
		\end{aligned}
		\end{equation*}
		If $a\neq0$ (Robin or Neumann boundary conditions), then
		\begin{equation}
		\begin{aligned}
		\int_{\partial\Omega}n \cdot \nabla\bar{T}_\varepsilon(t, x)\bar{T}_\varepsilon^{-}(t, x)d\sigma_{x}=&-\frac{b}{a\varepsilon^r}\int_{\partial\Omega}\bar{T}_\varepsilon(t, x)\bar{T}_\varepsilon^{-}(t, x)d\sigma_{x} \\&+\frac{b}{a\varepsilon^r}\int_{\partial\Omega}T_{b}(t, x)\bar{T}_\varepsilon^{-}(t, x)d\sigma_{x}\\=&\frac{b}{a\varepsilon^r}\int_{\partial\Omega}\left(\bar{T}_\varepsilon^{-}(t, x)\right)^{2}d\sigma_{x}\\&+\frac{b}{a\varepsilon^r}\int_{\partial\Omega}T_{b}(t, x)\bar{T}_\varepsilon^{-}(t, x)d\sigma_{x} \geqslant 0.
		\end{aligned}
		\label{Robnew}
		\end{equation}
		Now, if we have $a\neq0$ (thus $b>0$), since  $\bar{T}_\varepsilon^{-}=0$ on $\partial \Omega$ then
		\begin{equation}
		\begin{aligned}
		\int_{\partial\Omega}n \cdot \nabla\bar{T}_\varepsilon(t, x)\bar{T}_\varepsilon^{-}(t, x)d\sigma_{x}=0.
		\end{aligned}
		\label{Dirch}
		\end{equation}
		In the both cases, we have 
		\begin{equation}
		\int_{\Omega}\Delta \bar{T}_\varepsilon(t, x)\bar{T}_\varepsilon^{-}(t, x)dx\geqslant 0.
		\label{Lap}
		\end{equation}
		Consequently, \eqref{derivetemp}, \eqref{F} and \eqref{Lap} imply
		\begin{equation}
		\frac{1}{2}\partial_{t}\int_{\Omega}(\bar{T}_\varepsilon^{-}(t, x))^{2}dx\leq 0.
		\label{eqpos}
		\end{equation}
		As $T_{\varepsilon 0}$ is nonnegative, we deduce from \eqref{eqpos} that $\bar{T}_\varepsilon^{-}\equiv 0$. It follows that $\bar{T}_\varepsilon$ and consequently $T_\varepsilon$ are nonnegative in $\mathbb{R}_{+}\times\Omega$.

		Finally, we show if $T_{\varepsilon0}, T_b\le \gamma$ and $\psi_{\varepsilon0},\psi_b \le \gamma^4$, then the solution $(T_\varepsilon,\psi_\varepsilon)$ to system \eqref{eq:Teps}-\eqref{eq:psieps} satisfy $T_\varepsilon\le \gamma,\psi_\varepsilon\le \gamma$ in $\Omega$ for any $t>0.$
		Let $\gamma-T_\varepsilon  = g,  \gamma^4 - \psi_\varepsilon = \phi,$ then $(g,\phi)$ satisfies 
		\begin{align}
			&\partial_t g =\Delta g + \frac{1}{\varepsilon^2} \langle \phi -\gamma^4 + (\gamma-g)^4\rangle, \label{eq:A14}\\
			&\partial_t \phi + \frac{1}{\varepsilon}\beta\cdot \nabla \phi = -\frac{1}{\varepsilon^2}(\phi - \gamma^4 + (g+\gamma)^4),\label{eq:A15}
		\end{align}
		subject to the initial condition $g(t=0)=\gamma-T_{\varepsilon0}\ge 0,\psi(t=0)=\gamma^4 - \psi_{\varepsilon0} \ge 0$ and boundary condition 
		\begin{align*}
			a \varepsilon^r n\cdot \nabla g =- b(g-(T_b-\gamma)).
		\end{align*}
		We can take define $\bar{G}$ \eqref{eq:FbarAp} to be 
		\[
		\bar{G}=\left \{\begin{array}{lll}
		\frac{1}{\varepsilon^2}( \phi - (\gamma^4 - (\gamma-y)^4)),&{\rm if}& y\geq 0,\\
		\frac{1}{\varepsilon^2} \phi,&{\rm if}& y<0.
		\end{array}
		\right.,
		\]
		and due to $\gamma - (\gamma-y)^4\ge 0$ for $y\ge 0$, the same argument as before leads to the non-negativity of the solutions to 
		$$\partial_t g = \Delta g + \langle G(\bar{g})\rangle$$
		$$\partial_t \phi + \frac{1}{\varepsilon}\beta\cdot \nabla \phi = -\bar{G}(g).$$
		Due to $g\ge 0, \bar{G}(g) = (\phi-(\gamma^4-(\gamma-g)^4))$ and so the solution to \eqref{eq:A14}-\eqref{eq:A15} are non-negative $g=\gamma-T_\varepsilon\ge 0, \phi = \gamma^4 - \psi_\varepsilon \ge 0$ and hence $T_\varepsilon\le \gamma,\psi_\varepsilon\le \gamma^4$ holds.
\end{proof}}

\section{Lemmas used in the Compactness Method}\label{CompactnessLemma} 
We recall now the Compactness method  to prove the weak convergence.
\begin{lemma} \cite[Lemma 5.1]{lions1996mathematical}\label{Lemmalions1996mathematical}
Let $g^n$,$h^n$ converge weakly to $g$, $h$ respectively in $L^{p_1} (0, \tau; L^{p_{2}}(\Omega))$, $L^{q_1}(0,\tau; L^{q_2} (\Omega))$ where $1\le p_1, p_2 \le +\infty$,
$$
\frac{1}{p_1}+\frac{1}{q_1}=\frac{1}{p_2}+\frac{1}{q_2}=1
$$
We assume in addition that
\[\frac{\partial g^{n}}{\partial t} \text{ is bounded in } L^{1}(0, \tau; W^{-m,1}(\Omega)) \text{ for some }m\geqslant 0 \text{ independent of } n, \]
	\[\left\| h^{n}(.,t)-h^{n}(.+\xi,t)\right\|_{L^{q_1}(0,\tau; L^{q_2} (\Omega))} \to 0, \text{ as } |\xi | \to 0 \text{ uniformly in } n. \]
Then $g^{n}h^{n}$ converges to $gh$ in the sense of distributions on $\Omega \times (0,\tau)$.
\end{lemma}

Let us recall the averaging lemma see \cite{golse1988regularity,diperna1991lp,masmoudi2007diffusion}.
\begin{lemma}\label{AveragingLemma}
	Let $\tau>0$. Assume that $\psi_\varepsilon$ is bounded in $L^{2}((0,\tau)\times \Omega\times\mathbb{S}^{2})$, that $g_\varepsilon$ is bounded in $L^{2}((0,\tau)\times \Omega\times\mathbb{S}^{2})$, and that 
	\[\varepsilon\partial_t \psi_\varepsilon + \beta\cdot \nabla_x \psi_\varepsilon = g_\varepsilon.\]
	Then, for all $\rho \in C_0^\infty(\mathbb{S}^{2})$,
	\begin{equation}\label{EqaveragingLemma}
	\left\|\int_{\mathbb{S}^{2}} \left(\psi_\varepsilon(t,x+y,\beta)-h(t,x,v)\right)\rho(\beta)d\beta\right\|_{L^2_{t,x}} \to 0, \text{ as } y \to 0 \text{ uniformly in } \varepsilon,
	\end{equation}
	where $\psi_\varepsilon$ has been prolonged by $0$ for $x\notin \Omega$.
\end{lemma}
We consider only an average on the sphere and for the $L^{2}$ regularity of the solution. However, Lemma \ref{AveragingLemma} can be obtained  from the proof of \cite[Theorem 3 and Theorem 6]{diperna1991lp}. In our proof we follow\cite[Theorem 3 ]{diperna1991lp} with $q=p=2$, $m=1$ and $\tau=0$ and we prove that $
\int_{\mathbb{R}^{3}}\psi_\varepsilon(t,x,\beta)\rho(\beta)d\beta \in L^{r,\infty}(0,\tau;B^{s,r}_{\infty,\infty})$ where $r=2$ and $s=\frac{1}{4}$ which implies the compactness result given in \eqref{EqaveragingLemma}. 

\begin{proof}
We start by rewriting the problem in the whole domain. Let us introduce the following cut-off functions $\chi_{1}$ and $\chi_{2}$ such that $\chi_{1}(t)=1$ on $\left(\delta, \tau-\delta\right)$ for $\delta$ small enough, and $\chi_{2}(x)=1$ on $\{ x\in \Omega,\,\,|\,\, dis\left(x,\partial \Omega\right)>\delta \}$. Denoting by $\chi(t,x)=\chi_{1}(t)\chi_{1}(x)$ and $\tilde \psi_\varepsilon(t,x)=\psi_\varepsilon(t,x)\chi(t,x)$, then 
\begin{equation}\label{a1Lemma}
\varepsilon\partial_{t}{\tilde \psi_\varepsilon}+\beta\cdot\nabla_{x}{\tilde \psi_\varepsilon}=\chi g_\varepsilon+\left(\varepsilon\partial_{t}+\beta\cdot\nabla_{x}\right)\chi(t,x)\psi_\varepsilon(t,x,\beta).
\end{equation}
From the uniform bound of $\psi_\varepsilon$ and $g_\varepsilon$ in $L^{2}$ space 
	\[
	\left\|\int_{\mathbb{S}^{2}} \left(\psi_\varepsilon(t,x,\beta)-{\tilde \psi_\varepsilon}(t,x,\beta)\right)\rho(\beta)d\beta\right\|_{L^2\left(\left(0,\tau\right)\times\Omega\right)} \to 0,
	 \]
	 when  $\delta \to 0 $ uniformly in  $\varepsilon$.
	 
Now, we have $g_\varepsilon$ in $L^{2}$ it is enough to prove the lemma for $\mathbb{R}_{t}\times \mathbb{R}^{3}_{x}$ instead $(0,\tau)\times\Omega$. As the proof of \cite[Lemma 4.2]{masmoudi2007diffusion} is enough to prove 
\[
\int_{\mathbb{R}^{3}}\psi_\varepsilon(t,x,\beta)\rho(\beta)d\beta \in L^{r,\infty}(0,\tau;B^{s,r}_{\infty,\infty}),
\]
 where $r=2$ and $s=\frac{1}{4}$ for more details about the definition of Besov space built on the Lorentz space $L^{r,\infty}$ see \cite{jabin2004real}. As we said in the beginning we follow the proof of \cite[Lemma 4.2]{masmoudi2007diffusion}. So, we add  $\lambda \psi_\varepsilon$ in both side of equation \eqref{a1Lemma}, we obtain
	 
	 	\[
		\lambda \psi_\varepsilon+\varepsilon\partial_t \psi_\varepsilon + \beta\cdot \nabla_x \psi_\varepsilon = g_\varepsilon+\lambda \psi_\varepsilon.
		\]
Then 
\[
\int_{\mathbb{S}^{2}}\psi_\varepsilon(t,x,\beta) \rho(\beta)d\beta=T_{\lambda}\left(g_\varepsilon + \lambda \psi_\varepsilon\right),
\]
where
\[
T_{\lambda}(g)=\int_{0}^{\infty}\int_{\mathbb{S}^{2}}g(t-s,x-s\beta,\beta)e^{\lambda s}\rho(\beta)d\beta ds.
\]
Consequently, from \cite[Proposition 3.1]{jabin2004real} it follows
\[
\|T_{\lambda}(g)\|_{L^{2}_{t}H_{x}^{1/2}}\le \lambda^{-1/2} \|g\|_{L^{2}(\mathbb{R}\times\mathbb{R}^{3}\times\mathbb{R}^{3})},
\]
and 	 
\[
\|T_{\lambda}(g)\|_{\lambda^{-3/2}L^{2}_{t}H_{x}^{-1/2}+\lambda^{-1/2}L^{2}_{t}H_{x}^{1/2}}\le C \|g\|_{L^{2}(L^{2}(\mathbb{R}\times\mathbb{R}^{3};H_{x}^{-1}(\mathbb{R}^{3}))}
\]	 
Moreover, we have 
\[
\int_{\mathbb{S}^{2}}\psi_\varepsilon(t,x,\beta) \rho(\beta)d\beta=\eta=\eta^{1}+\eta^{2},
\]
where
\[
\|\eta^{1}\|_{L^{2}_{t}H_{x}^{1/2}}\le C \lambda^{1/2}\|\psi_\varepsilon\|_{L^{2}},
\]
and	 
\[
\|\eta^{2}\|_{\lambda^{-3/2}L^{2}_{t}H_{x}^{-1/2}+\lambda^{-1/2}L^{2}_{t}H_{x}^{1/2}}\le C\|g_{\varepsilon}\|_{L^{2}}.
\]
Then by rewriting $\eta^{2}=\eta^{2}_{1}+\eta^{2}_{2}$ as
\[
\|\eta^{2}_{1}\|_{L^{2}_{t}H_{x}^{-1/2}}\le C\lambda^{-3/2}\|g_{\varepsilon}\|_{L^{2}}
\]

\[
\|\eta^{2}_{2}\|_{L^{2}_{t}H_{x}^{1/2}}\le C\lambda^{-1/2}\|g_{\varepsilon}\|_{L^{2}}
\]
Then we deduce that $\eta \in \left(L^{2}_{t}H_{x}^{1/2},L^{2}_{t}H_{x}^{-1/2}\right)$ for the real interpolation of order $(\frac{1}{4},\infty)$, see \cite[Proposition 3.1]{jabin2004real}. Then, for all $t\in\mathbb{R}_{+}$, we have
\[
K(t)=\inf_{a_{1}+a_{2}=\eta} \|a_{1}\|_{L^{2}_{t}H_{x}^{1/2}}+t\|a_{2}\|_{L^{2}_{t}H_{x}^{-1/2}}
\] 	 
Thus from  \cite{jabin2004real}, we need to show that $K(t)\le Ct^{1/4}$. We choose $\lambda$ such that $t=\lambda^{2}$ for $t>0$. 
\begin{itemize}
\item If $0<t<1$, we have $\|\eta^{2}_{2}\|_{L^{2}_{t}H_{x}^{-1/2}}\le C \lambda^{-3/2}$. Then $\eta_{2}^{2}$ and $a_{1}=\eta^{1}$ and $a_{2}=\eta^{2}$ we conclude that $K(t)\le Ct^{1/4}$. 

\item For $t>1$, we can rewrite $\eta^{2}$ as $\eta^{2}=\eta_{3}^{2}+\eta_{4}^{2}$ such that  $\eta_{3}^{2}\in \lambda^{-3/2}L^{2}_{t}H_{x}^{-1/2}$ and $\eta_{4}^{2}\in \lambda^{1/2}L^{2}_{t}H_{x}^{3/2}$, then we define 
\[
K_{1}(t)=\inf_{a_{1}+a_{2}=\eta} \|a_{1}\|_{L^{2}_{t}H_{x}^{1/2}+L^{2}_{t}H_{x}^{3/2}}+t\|a_{2}\|_{L^{2}_{t}H_{x}^{-1/2}}
\]
we obtain $K_{1}(t)\le Ct^{1/4}$ for $a_{1}=\eta_{1}+\eta^{2}_{4}$ and $a_{2}=\eta^{2}_{1}+\eta^{2}_{3}$ which implies  that 
\[
\eta\in \left(L^{2}_{t}H^{1/2}+L_{t}^{2}H^{3/2}, L^{2}_{t}H_{x}^{-1/2}\right).
\]
\end{itemize}
Finally, we deduce the compactness result. This finishes the proof.

\end {proof}

% \begin{theorem}\label{Traceoperator}
% Let $\Omega$ be a $C^{k-1,1}$-domain. Let  $\frac{1}{2}< s \leqslant k$, the trace operator 
% $$
% \gamma^{s}_{D}: H^{s}(\Omega)\longrightarrow H^{s-\frac{1}{2}}(\partial \Omega),
% $$
% is continuous where $\gamma^{s}_{D}(v)=v_{|_{\partial \Omega}}$ is bounded. There exists $C=C(\Omega)>0$ such that
% $$
% \|\gamma^{s}_{D}(v)\|_{H^{s-\frac{1}{2}}(\partial \Omega)}\leqslant C \|v\|_{H^{s}(\Omega)}, \forall v\in H^{s}(\Omega).
% $$
% \end{theorem}

% \begin{lemma}\label{minty}
% 	(Minty's trick, \cite{droniou2018gradient}) Let $f$ be a nondecreasing function. Let $(X,\mu)$ be a measurable set and assume $\{u_n\}_{n\in\mathcal{N}}) \subset L^2(X)$ satisfy
% 	 \begin{enumerate}
% 	 	\item $\{u_n\}$ converges weakly in $L^2(X)$ to $u$ as $n\to \infty$.
% 	 	\item $\{f(u_n)\}$ converges weakly in $L^2(X)$ to $v$ as $n\to \infty$.
% 	 	\item $\{u_n f(u_n)\}$ converges to $uv$ in the sense of distributions.
% 	 \end{enumerate}
% 	 Then $v=f(u)$ almost everywhere. If futhermore $f$ is strictly increasing, then
% 	 \[u_n \to u, \mbox{a.e.}\]
% \end{lemma}

\section{Basis in $L^2(\Omega\times\mathbb{S}^2)$} \label{appendixb}
\begin{lemma}\label{lem:basis}
	There exists an orthonormal basis $\{\varphi_k\}_{k=1}^\infty$ of $L^2(\Omega\times\mathbb{S}^2)$ with $\varphi_k \in H^1(\Omega)\otimes L^2(\mathbb{S}^2)$.
\end{lemma}
\begin{proof}
	We take $\{\phi_i\}_{i=1}^\infty$ to be a orthogonal basis in $H^1(\Omega)$ which is also an orthonormal basis in $L^2(\Omega)$ and $\{\chi_j\}_{j=1}^\infty$ to be a orthonormal basis in $L^2(\mathbb{S}^2)$. Then 
	$$\varphi_{ij} = \phi_i \chi_j$$
	is a orthonormal basis in $L^2(\Omega\times\mathbb{S}^2)$. 
	% We can rearrange the index and take
	% \begin{align*}
	% 	\varphi_{i} = (\phi_i \chi_1, \varphi_i\chi_2,\ldots,\varphi_i \chi_i),
	% \end{align*}
	% and  $\{\varphi_k\}_{k=1}^\infty$ will be an orthonormal basis of $L^2(\Omega\times\mathbb{S}^2)$.
\end{proof}

\end{document}